\let\oldvec\vec

\documentclass[invmat,numbook,final]{svjour}

\let\newvec\vec
\let\vec\oldvec
\usepackage{amsmath}
\usepackage{calc}
\usepackage[T1]{fontenc}
\usepackage[full]{textcomp}

\usepackage[utopia]{mathdesign}

\usepackage{mathtools}
\DeclareSymbolFont{cmcal}{OMS}{cmsy}{m}{n}
\SetSymbolFont{cmcal}{bold}{OMS}{cmsy}{b}{n}
\DeclareSymbolFontAlphabet{\mathcal}{cmcal}

\let\vec\newvec

\usepackage{pifont}

\usepackage[all,cmtip]{xy}
\usepackage{xcolor}
\usepackage{enumitem} 
\usepackage[parfill]{parskip}

\usepackage{mathtools}
\usepackage[colorlinks=true]{hyperref} 

\spnewtheorem{notation}[definition]{Notation}{\bfseries}{\itshape}

\usepackage[super]{nth}

\usepackage{etoolbox}
\tracingpatches
\makeatletter
\patchcmd{\@citeo}{\hskip0.1em}{\kern0.1em}{}{}
\patchcmd{\@citex}{\hskip0.1em}{\kern0.1em}{}{}
\makeatother

\newcommand{\eq}[1]{eq.\ (\ref{#1})}

\newcommand{\eqn}[2]{
\begin{equation}\label{#1}
#2
\end{equation}
}

\newcommand{\eqnalign}[2]{
	\begin{equation}
	\begin{aligned}\label{#1}
	#2
	\end{aligned}
	\end{equation}
}

\makeatletter
\newcommand*{\Xbar}[1]{}%
\DeclareRobustCommand*{\Xbar}[1]{%
  \mathpalette\@Xbar{#1}%
}
\newcommand*{\@Xbar}[2]{%
  \sbox0{$#1\mathrm{#2}\m@th$}%
  \sbox2{$#1#2\m@th$}%
  \rlap{%
    \hbox to\wd2{%
      \hfill
      $\overline{%
        \vrule width 0pt height\ht0 %
        \kern\wd0 %
      }$%
    }%
  }%
  \copy2 %
}
\newcommand*{\Frozenbar}[1]{}%
\DeclareRobustCommand*{\Frozenbar}[1]{%
  \mathpalette\@Frozenbar{#1}%
}
\newcommand*{\@Frozenbar}[2]{%
  \sbox0{$#1\mathrm{W}\m@th$}%
  \sbox2{$#1#2\m@th$}%
  \rlap{%
    \hbox to\wd2{%
      \hfill
      $\overline{%
        \vrule width 0pt height\ht0 %
        \kern\wd0 %
      }$%
    }%
  }%
  \copy2 %
}
\makeatother

\newcommand{\eqalign}[1]{
	\begin{align*}#1 \end{align*}
	}

\usepackage{isomath}  
\SetMathAlphabet{\mathbf}{normal}{OML}{mdput}{b}{n}

\def\mb#1{{\mathbf{#1}}}
\def\bm#1{{\boldsymbol{#1}}}

\usepackage{tocbibind}

\def\a{\alpha}
\def\b{\beta}

\def\d{\delta}  
\def\e{\varepsilon} \def\ep{\epsilon}
  
\def\g{\gamma}

\def\l{\lambda}  
\def\m{\mu}

\def\r{\rho}

\def\o{\omega}  \def\O{\Omega}
\def\p{\psi}  
\def\s{\sigma}  \def\vs{\varsigma} 
  
\def\t{\tau}
\def\w{\varphi}

\def\ups{\upsilon}

\def\CE{{\cal E}}

\def\CG{{\cal G}}

\def\CL{{\cal L}}

\def\Z{\mathbb{Z}}

\def\I{\mathbb{I}}

\def\mq{\mathfrak{q}}

\def\mp{\mathfrak{p}}

\def\mr{\mathfrak{r}}
\def\ms{\mathfrak{s}}

\def\ml{\mathfrak{l}}

\def\mg{\mathfrak{g}}
\def\mG{\mathfrak{G}}

\def\sN{\mathscr{N}}

\newcommand{\Fr}[2]{\dfrac{#1}{#2}}

\def\pr{\prime}
\def\ppr{{\prime\!\prime}}

\DeclareMathOperator{\End}{End}
\DeclareMathOperator{\Aut}{Aut}
\DeclareMathOperator{\Ker}{Ker}

\DeclareMathOperator{\Hom}{Hom}

\usepackage{graphicx}
\makeatletter
\DeclareRobustCommand*\uell{\mathpalette\@uell\relax}
\newcommand*\@uell[2]{
  \setbox0=\hbox{$#1\ell$}
  \setbox1=\hbox{\rotatebox{10}{$#1\ell$}}
  \dimen0=\wd0 \advance\dimen0 by -\wd1 \divide\dimen0 by 2
  \mathord{\lower 0.1ex \hbox{\kern\dimen0\unhbox1\kern\dimen0}}
}

\makeatletter
\def\widebreve{\mathpalette\wide@breve}
\def\wide@breve#1#2{\sbox\z@{$#1#2$}%
     \mathop{\vbox{\m@th\ialign{##\crcr
\kern0.06em\brevefill#1{0.6\wd\z@}\crcr\noalign{\nointerlineskip}%
                    $\hss#1#2\hss$\crcr}}}\limits}
\def\brevefill#1#2{$\m@th\sbox\tw@{$#1($}%
  \hss\resizebox{#2}{\wd\tw@}{\rotatebox[origin=c]{90}{\upshape(}}\hss$}
\makeatletter

\def\fieldk{\Bbbk}

\def\cstar{\hbox{\scriptsize\ding{75}}}

\newcommand{\functor}[1]{\mathbf{#1}}
\newcommand{\dgcat}[1]{\overline{\category{#1}}\,}

\newcommand{\category}[1]{\emph{\textbf{#1}}\,}

\newcommand{\HOM}{\hbox{$\mathop{\bm{\Hom}}$}}

\DeclareMathOperator{\THom}{THom}
\newcommand{\THOM}{\hbox{$\mathop{\bm{\THom}}$}}

\usepackage{pifont}

\def\cp{{\mathrel\vartriangle}}

\def\cdga{{\category{cdgA}(\Bbbk)}}
\def\hcdga{{\mathit{ho}\category{cdgA}(\Bbbk)}}

\def\cdgh{{\category{cdgH}(\Bbbk)}}
\def\hcdgh{{\mathit{ho}\category{cdgH}(\Bbbk)}}

\DeclareMathAlphabet{\mathpzc}{OT1}{pzc}{m}{it}
\def\ide{\mathpzc{e}}

\usepackage{shuffle}

\begin{document}

\title{ Affine group dg-schemes and linear representations  I}
\subtitle{Basic theory and Tannakian reconstructions}


\author{
Jaehyeok Lee\thanks{jhlee1228lee@postech.ac.kr} \and
Jae-Suk Park\thanks{jaesuk@postech.ac.kr}
}

\institute{
Department of Mathematics, POSTECH, Pohang 37673, Republic of Korea
}

\institute{
Department of Mathematics, POSTECH, Pohang 37673, Republic of Korea
}
\maketitle
\begin{abstract}
We  develop a basic theory of affine group dg-schemes, their Lie algebraic counterparts and linear representations. 
We prove  Tannaka type reconstruction theorems that an affine group dg-scheme 
can be recovered from  the dg-tensor category of its linear representations 
as well as from the rigid dg-tensor category of its finite dimensional linear representations 
along with the forgetful functors to the underlying dg-tensor category of cochain complexes.
\end{abstract}

{\footnotesize
\tableofcontents
}

\section{Introduction}
The theory of  affine group schemes and their representations is a classic subject 
in algebraic geometry---we refer to  \cite{Waterhouse,Milne}  for a review. 
It has led to the notion of  Tannakian category  as envisioned by Grothendieck and established  
by  Saavedra Rivano \cite{Rivano} and Deligne \cite{Deligne90}
realizing the Tannaka-Klein duality for the group objects in the category of affine schemes---see \cite{DM,Szamuely} for a review.
Beside from its initially  intended  role in algebraic geometry (Tannakian fundamental group schemes and motives), 
 the ideas of Tannakian category find applications
in wider branches in mathematics--we refer to the book \cite{EGNO} and references therein.

We recall that an affine group scheme 
over  a field $\Bbbk$
is a representable functor 
$G:\category{cAlg}(\Bbbk)\rightsquigarrow \category{Grp}$ from the
category $\category{cAlg}(\Bbbk)$ of  commutative algebras over $\Bbbk$
to the category $\category{Grp}$ of groups, whose representing object is a commutative
Hopf algebra $B^0$.  The category $\category{Rep}(G)$ of  linear representations
of $G$ is isomorphic to the category $\category{Comod}_R(B^0)$ of right comodules over $B^0$.
Then a Tannakian reconstruction theorem is
that $G$ can be reconstructed from the category $\category{Rep}(G)_f$ 
of  its finite dimensional linear representations via the
forgetful functor $\o: \category{Rep}(G)_{\!f}\rightsquigarrow \category{Vec}(\Bbbk)_{\!f}$
to the category of the underlying vector spaces.  
The category $\category{Rep}(G)_{\!f}$ with the functor $\o$
is the prototype a of neutral Tannakian category, which is a rigid $\Bbbk$-linear abelian tensor category with
a fiber functor to the category $\category{Vec}(\Bbbk)_f$ satisfying certain conditions.
The other side of the duality is that
any neutral Tannakian category $\big(\category{T}, \o\big)$ is equivalent to 
the category of the finite dimensional linear representations 
of an affine group scheme $G$, which is  called the Tannakian fundamental group scheme of $\big(\category{T}, \o\big)$.

This paper is about a basic theory of affine group dg-schemes,
their Lie algebraic counterparts,  and their linear representations as well as  Tannakian reconstruction theorems. 
This paper is a companion to our recent paper on the similar study of 
representable presheaves of groups on the homotopy category of cocommutative dg-coalgebras \cite{JLJSP}.
The prefix "dg" stands for "differential graded" and
 a dg-category  is an enriched category 
whose morphism sets are endowed with the structure of a cochain complex, where
cochain complexes over $\Bbbk$ form the  prototypical dg-category 
denoted by
$\dgcat{CoCh}(\Bbbk)$.

In Section $3$, we define an affine group dg-scheme  over $\Bbbk$ as a representable  functor 
$\bm{\mG}:\mathit{ho}\category{cdgA}(\Bbbk)\rightsquigarrow \category{Grp}$ 
from the homotopy category $\mathit{ho}\category{cdgA}(\Bbbk)$ of commutative dg-algebras over $\Bbbk$
to the category $\category{Grp}$ of groups.  A representing object of $\bm{\mG}$  is
a commutative dg-Hopf algebra  $B$, and we use the notation $\bm{\mG}^B$ for it.  The category of affine group dg-schemes
is anti-equivalent to the homotopy category $\mathit{ho}\category{cdgH}(\Bbbk)$ of commutative dg-Hopf algebras.
For the Lie theoretic counterpart to affine group dg-schemes, 
we construct a functor $\category{T}\!\bm{\mG}^B:\mathit{ho}\category{cdgA}(\fieldk)\rightsquigarrow \category{Lie}(\Bbbk)$ 
to the category $\category{Lie}(\Bbbk)$ of Lie algebras over $\Bbbk$
so that $\category{T}\!\bm{\mG}^B$ and $\category{T}\!\bm{\mG}^{B^\pr}$ are naturally isomorphic whenever 
$\bm{\mG}^B$ and $\bm{\mG}^{B^\pr}$ are naturally isomorphic.  
For a pro-unipotent affine group dg-scheme   $\bm{\mG}^B$, where $B$ is  conilpotent,
 we show that the underlying set-valued functors 
 $\grave{\bm{\mG}}^{{}_B}:\mathit{ho}\category{cdgA}(\fieldk)\rightsquigarrow \category{Set}$ 
 and $\grave{\category{T}\!\bm{\mG}}^{{}_{B}}:\mathit{ho}\category{cdgA}(\fieldk)\rightsquigarrow \category{Set}$
are naturally isomorphic---we can recover the group ${\bm{\mG}}^B(A)$ from the Lie algebra ${\category{T}\!\bm{\mG}}^B(A)$
for every cdg-algebra $A$. 

In Sect. $4$,  we define a linear representation of $\bm{\mG}^B$ via a linear representation of
the associated representable functor 
$\bm{\CG}^B:\category{cdgA}(\Bbbk)\rightsquigarrow \category{Grp}$ 
from the category $\category{cdgA}(\Bbbk)$ of commutative dg-algebras over $\Bbbk$.
$\bm{\CG}^B$ is also represented by $B$
and induces $\bm{\mG}^B$ on the homotopy category $\mathit{ho}\category{cdgA}(\Bbbk)$.
Morevoer, a linear representation of $\bm{\CG}^B$ induces  a linear representation of $\bm{\mG}^B$.
We shall  form  a dg-tensor category $\dgcat{Rep}(\bm{\CG}^B)$ of  linear representations of  $\bm{\CG}^B$ 
and show that it is isomorphic to a dg-tensor category $\dgcat{dgComod}_R(B)$ formed by right dg-comodules over $B$.

In Sect. $5$, 
we reconstruct $\bm{\mG}^B$ via the forgetful functor $\bm{\o}:\dgcat{dgComod}_R(B)\rightsquigarrow \dgcat{CoCh}(\Bbbk)$
to the underlying  dg-tensor category $\dgcat{CoCh}(\Bbbk)$ of  cochain complexes.
Out of $\bm{\o}$,  we shall construct two functors 
$\bm{\CG}^{\bm{\o}}_{\!\otimes}:\category{cdgA}(\Bbbk) \rightsquigarrow \category{Grp}$
{and }
$\bm{\mG}^{\bm{\o}}_{\!\otimes}:\mathit{ho}\category{cdgA}(\Bbbk) \rightsquigarrow \category{Grp}$
and establish 
natural isomorphisms of functors
$$
\bm{\CG}^{\bm{\o}}_{\!\otimes}\cong \bm{\CG}^{\!\!B}:\category{cdgA}(\Bbbk) \rightsquigarrow \category{Grp}
\hbox{ and }
\bm{\mG}^{\bm{\o}}_{\!\otimes}\cong \bm{\mG}^{\!B}:\mathit{ho}\category{cdgA}(\Bbbk) \rightsquigarrow \category{Grp},
$$
which is our $1$st reconstruction theorem of  an affine group dg-scheme 
from the dg-tensor category of linear representations.
We also consider the dg-tensor category $\dgcat{Rep}(\bm{\CG}^B)_{\!f}$ of \emph{finite} dimensional linear representations of  $\bm{\CG}^B$ 
which is isomorphic to the dg-tensor category $\dgcat{dgComod}_R(B)_{\!f}$ of \emph{finite} dimensional  right dg-comodules over $B$.
From the forgetful  functor $\bm{\o}_{\!f}: \dgcat{dgComod}_R(B)_{\!f}\rightsquigarrow \dgcat{CoCh}(\Bbbk)_{\!f}$, we 
shall  construct two functors
$\bm{\CG}^{\bm{\o}_{\!f}}_{\!\otimes}:\category{cdgA}(\Bbbk) \rightsquigarrow \category{Grp}$ and
$\bm{\mG}^{\bm{\o}_{\!}}_{\!\otimes}:\mathit{ho}\category{cdgA}(\Bbbk) \rightsquigarrow \category{Grp}$
and  establish natural isomorphisms 
$\bm{\CG}^{\bm{\o}_{\!f}}_{\!\otimes}\cong \bm{\CG}^{\!\!B}$ and
$\bm{\mG}^{\bm{\o}_{\!f}}_{\!\otimes}\cong \bm{\mG}^{\!B}$,
which is our $2$nd reconstruction theorem of  an affine group dg-scheme 
from the dg-tensor category of finite dimensional linear representations.
The $2$nd reconstruction theorem is obtained by  the reductions of  $1$st reconstruction theorem to
the finite dimensional cases, based on the basic fact that every right dg-comodule over $B$ is a
filtered colimit of its finite dimensional subcomodules over $B$.\footnote{We do not have the similar reduction
to finite dimensional cases for the dg-tensor category of dg-modules over cocommutative dg-Hopf algebra studied in
\cite{JLJSP}.}
We shall also give an independent proof using the dg-version of rigidity  of  $\dgcat{dgComod}_R(B)_{\!f}$.

Our study of affine group dg-schemes is originally motivated from  de Rham side of the rational homotopy theory
of Sullivan \cite{Sullivan} and Chen \cite{Chen} as well as Deligne's rational de Rham fundamental group scheme of
algebraic varieties and their periods \cite{Deligne89}.  We like to have a natural extension of  
rational de Rham fundamental group scheme of a space  to a rational de Rham fundamental group dg-scheme 
encoding more general invariants of rational homotopy types of those spaces, including all higher rational homotopy groups---but these 
and other applications will be the subjects of a sequel to this paper \cite{JLJSP2}. 
Another motivation for this paper is to do some groundworks toward the theory of "Tannakian dg-tensor categories" 
as many  Tannakian categories come naturally with dg-tensor categories---see \cite{Simpson} for examples.  
We note that some progress along this line is reported in \cite{Pridam}.

\begin{acknowledgement}
The work of JL was supported by NRF(National Research Foundation of Korea) 
Grant funded by the Korean Government(NRF-2016-Global Ph.D. Fellowship Program).
JSP is grateful to Cheolhyun Cho and Gabriel Drummond Cole for useful discussions.
\end{acknowledgement}

\section{Notations}

We use the similar notations and conventions as in our previous paper \cite{JLJSP}, except
that we use the cohomological $\Z$-grading---every differential has degree $1$ 
and a dg-category is a category enriched in the category of cochain complexes.

Throughout this paper $\Bbbk$ is a ground field of characteristic $0$.
Unadorned tensor product $\otimes$ is over $\Bbbk$.
By an element in a $\Z$-graded vector space we shall usually mean a homogeneous element $x$
whose degree will be denoted $|x|$. 
Let $V=\bigoplus_{n\in \Z}V^n$ and $W=\bigoplus_{n\in \Z}W^n$ be $\Z$-graded vector spaces.
Then $V\!\otimes\! W = \bigoplus_{n\in \Z}(V\! \otimes\! W)^n$,
where $(V \!\otimes\! W)^n= \bigoplus_{i+j=n\in \Z}V^i\! \otimes\! W^j$,
is also a $\Z$-graded vector space.
Denote $\Hom(V,W)= \bigoplus_{n\in \Z}\Hom(V,W)^n$ as the $\Z$-graded vector space of 
$\Bbbk$-linear maps from $V$ to $W$.
A cochain complex $\big(V, d_V\big)$ is often denoted by $V$ for simplicity.
The ground field $\Bbbk$ is a cochain complex with the zero differential.
If $V$ and $W$ are cochain complexes  $V\otimes W$ 
and $\Hom(V, W)$ are also cochain complexes with the following differentials
\eqn{tshmdiff}{
\begin{cases}
d_{V\otimes  W} = d_V\otimes \I_W +\I_V\otimes d_W, &
\cr
d_{V\!,W} f= d_W\circ f -(-1)^{|f|}f\circ d_V, &\forall f\in \Hom(V,W)^{|f|}
.
\end{cases}
}
A \emph{cochain map} $f:\big(V, d_V\big)\rightarrow \big(W, d_W\big)$ is an $f \in \Hom(V,W)^0$
satisfying $d_{\mathit{V\!,W}} f =d_V\circ f -f\circ d_W =0$.  
Two cochain maps $f$ and $\tilde f$ are \emph{homotopic}, denoted by $f \sim \tilde f$,
or have the \emph{same homotopy type}, denoted by $[f]=[\tilde f]$, if there is a cochain homotopy  $\l \in \Hom(V,W)^{-1}$
such that $\tilde f - f = d_{\mathit{V\!,W}} \l$.  

The set of morphisms from an object $C$ to another object $C^\pr$ in a category 
$\category{C}$ is denoted by $\HOM_{\category{C}}(C, C^\pr)$. 
We denote the set of natural transformations of functors
$\functor{F}\Rightarrow\functor{G}:\category{C}\rightsquigarrow \category{D}$ by $\mathsf{Nat}(\functor{F},\functor{G})$.
For any functor $\functor{F}:\category{C}\rightsquigarrow\category{D}$, where $\category{D}$ is small, we use the notation
$\grave{\functor{F}}:\category{C}\to\category{Set}$ for the underlying set valued functor obtained by composing
$\functor{F}$ with the forgetful functor
$\functor{Forget}:\category{D}\rightsquigarrow\category{Set}$.
Such a functor $\functor{F}:{\category{C}} \rightsquigarrow \category{D}$
is called \emph{representable} if $\grave{\functor{F}}$ is representable.

Remark that the ground field $\Bbbk$ is an algebra $(\Bbbk, u_\Bbbk, m_\Bbbk)$
where $u_\Bbbk =\I_\Bbbk$ and $m_\Bbbk(a\otimes b)=a\cdot b$,
and a coalgebra $\Bbbk^\vee=(\Bbbk, \ep_\Bbbk,\cp_\Bbbk)$
with $\ep_\Bbbk =\I_\Bbbk$ and $\cp_\Bbbk(1)= 1\otimes 1$.
The canonical isomorphisms $\Bbbk\otimes V\cong V$ 
and $V\otimes \Bbbk \cong V$
will be denoted by $\imath_V: \Bbbk\otimes V\rightarrow  V$ and $\imath^{-1}_V:V\rightarrow \Bbbk\otimes V$,  
as well as  $\jmath_V: V\otimes\Bbbk\rightarrow V$ and $\jmath^{-1}_V:V\rightarrow   V\otimes \Bbbk$.

A \emph{commutative dg Hopf algebra (cdg-Hopf algebra)} is a
tuple 
$$B =\big(B, u_B, m_B, \ep_B, \cp_B, \vs_B, d_B\big),$$ 
where

\begin{itemize}

\item $(B,d_B)$ is a cochain complex;
\item $\big(B, u_B, m_B,d_B\big)$ is a \emph{commutative dg algebra (cdg-algebra)}
that both the \emph{unit}
$u_B:\Bbbk\to B$ and the \emph{product} $m_B:B\otimes B\to B$ are cochain maps
satisfying the unit axiom and the associativity axiom:
\eqnalign{dgalgebra}{
\begin{cases}
d_B\circ u_B=0
,\cr
m_B\circ d_{B\otimes B} = d_{B}\circ m_B
,
\end{cases}
\quad
\begin{cases}
m_B\circ (u_B\otimes \I_B)=\imath_B\cong  m_B\circ (\I_B\otimes u_B)=\jmath_B,
\cr
m_B\circ (m_B\otimes \I_B)= m_B\circ (\I_B\otimes m_B)
,
\end{cases}
}
and the commutativity $m_B = m_B\circ \t$, 
where $\t(x\otimes y) = (-1)^{|x||y|}y\otimes x$, $\forall x,y \in A$;

\item $(B, \ep_B, \cp_B, d_B)$ is a \emph{dg coalgebra (dg-coalgebra)}
that both the \emph{counit} $\ep_B:B\to \Bbbk$ and the \emph{coproduct} $\cp_B:B\to B\otimes B$
are cochain maps satisfying the counit axiom and the coassociativity axiom:
\eqnalign{dgcoalgebra}{
\begin{cases}
\ep_B\circ d_B=0
,\cr
\cp_B\circ d_{B} = d_{B\otimes B}\circ\cp_B
,
\end{cases}
\quad
\begin{cases}
(\ep_B\otimes \I_B)\circ \cp_B=\imath^{-1}_B\cong  (\I_B\otimes \ep_B)\circ \cp_B=\jmath^{-1}_B,
\cr
(\cp_B\otimes \I_B)\circ \cp_B=(\I_B\otimes \cp_B)\circ \cp_B
.
\end{cases}
}

\item $\vs_B: B \rightarrow B$ is a cochain map called the \emph{antipode},
\end{itemize}
such that  the counit $\ep_B:B\to \Bbbk$ and the coproduct $\cp_B:B\to B\otimes B$ are
algebra maps:
\eqnalign{bialgebra}{
\begin{cases}
\ep_B\circ u_B = u_\fieldk
,\cr
\ep_B\circ m_B=m_\fieldk\circ (\ep_B\otimes \ep_B)
,
\end{cases}
\quad
\begin{cases}
\cp_B\circ u_B=  u_{B\otimes B}\circ \cp_\fieldk
,\cr
\cp_B\circ m_B= m_{B\otimes B}\circ (\cp_B\otimes \cp_B)
,
\end{cases}
}
and  $\vs_B:B\rightarrow B$ satisfies the antipode axiom:
\eqn{antipodeaxiom}{
m_B\circ(\vs_B\otimes \I_B)\circ\cp_B = m_B\circ (\I_B\otimes \vs_B)\circ\cp_B = u_B\circ \ep_B.
}

\begin{remark}
We have used the structure 
$\big(B\otimes B, u_{B\otimes B}, m_{B\otimes B} , d_{B\otimes B}\big)$
of cdg-algebra on $B\otimes B$ induced from the cdg-algebra on $B$
where 
$$
u_{B\otimes B}=u_{B}\otimes u_B
,\qquad m_{B\otimes B} = (m_B\otimes m_B)\circ (\I_B\otimes\t\otimes \I_B).
$$
We also have the structure 
$\big(B\otimes B, \ep_{B\otimes B}, \cp_{B\otimes B} , d_{B\otimes B}\big)$
of dg-coalgebra on $B\otimes B$ induced from the dg-coalgebra on $B$, where
$$
\ep_{B\otimes B}=\ep_{B}\otimes \ep_B,
\qquad
\cp_{B\otimes B} = (\I_B\otimes\t\otimes \I_B)
\circ (\cp_B\otimes \cp_B).
$$
The conditions in \eq{bialgebra} is equivalent to the conditions that
the unit $u_B:\Bbbk\to B$ and the product $m_B:B\otimes B \to B$
are coalgebra maps.  A cdg-Hopf algebra, after forgetting the antipode,
is a \emph{cdg-bialgebra}.  In other words a cdg-Hopf algebra is a cdg-bialgebra with
an antipode. 
An antipode $\vs_B$ of a cdg-bialgebra $B$ is unique if exists,  and 
is a morphism of cdg-algebras and an anti-morphism of dg-coalgebras:
\eqnalign{antialgandcoalg}{
\begin{cases}
\vs_B\circ u_B=u_B
,\cr 
\vs_B \circ m_B = m_B\circ(\vs_B\otimes \vs_B)
,
\end{cases}
\qquad
\begin{cases}
\ep_B\circ \vs_B=\ep_B
,\cr
\cp_B \circ \vs_B = \t\circ(\vs_B\otimes \vs_B)\circ\cp_B
.
\end{cases}
}
\end{remark}

A \emph{morphism $f:B \rightarrow B^\pr$ of cdg-Hopf algebras} is
simultaneously 
\begin{itemize}
\item a morphism $f:\big(B, u_B, m_B, d_B\big)\to \big(B^\pr, u_{B^\pr}, m_{B^\pr}, d_{B^\pr}\big)$ of cdg-algebras:
\eqn{cdgalgebramap}{
f\circ d_B = d_{B^\pr}\circ f, \qquad
f\circ u_B = u_{B^\pr},\qquad
f\circ m_B = m_{B^\pr}\circ (f\otimes f)
,
}
\item and 
a morphism $f:\big(B, \ep_B, \cp_B, d_B\big)\to \big(B^\pr, \ep_{B^\pr}, \cp_{B^\pr}, d_{B^\pr}\big)$ of dg-coalgebras:
\eqn{dgcoalgebramap}{
f\circ d_B = d_{B^\pr}\circ f, \qquad
\ep_{B^\pr}\circ f=\ep_B,\qquad
\cp_{B^\pr}\circ f = (f\otimes f)\circ \cp_B
.
}
\end{itemize}
Then it is automatic that $f$ commutes with the antipodes: 
\eqn{antipodecommutes}{
f\circ \vs_B =\vs_{B^\pr}\circ f.
}
The composition of morphisms of cdg-Hopf algebras 
is defined by the composition as linear maps, which is
obviously a morphism of cdg-Hopf algebras.
The category of cdg-Hopf algebras, dg-algebras, cdg-algebras and dg-coalgebras are denoted by
$\category{cdgH}(\fieldk)$, $\category{dgA}(\Bbbk)$, $\category{cdgA}(\fieldk)$ and $\category{dgC}(\fieldk)$, respectively.

A \emph{homotopy pair} on $\HOM_{\category{cdgH}(\Bbbk)}(B, B^\pr)$
is a pair of one parameter families $\big(f(t), \xi(t)\big)\in \Hom(B, B^\pr)^0[t]\oplus \Hom(B, B^\pr)^{-1}[t]$, 
parametrized by the time variable $t$ with polynomial dependences, 
satisfying the \emph{homotopy flow equation} $\Fr{d}{dt}f(t)= d_{B, B^\pr} \xi(t)$ generated by $\xi(t)$,
subject to the following two types of conditions:
\begin{itemize}
\item \emph{infinitesimal algebra map}:  $f(0) \in  \HOM_{\category{cdgA}(\fieldk)}(B, B^\pr)$ and 
$$
\xi(t)\circ u_B=0, \qquad
\xi(t) \circ m_B =m_{B^\pr}\circ \big(f(t)\otimes \xi(t) +\xi(t)\otimes f(t)\big)
;
$$
\item\emph{infinitesimal coalgebra map}:  $f(0) \in  \HOM_{\category{dgC}(\fieldk)}(B, B^\pr)$
and
$$
\ep_{B^\pr}\circ\xi(t)=0,\qquad
\cp_{B^\pr}\circ\xi(t)=\big(f(t)\otimes \xi(t) +\xi(t)\otimes f(t)\big)\circ \cp_B
.
$$
\end{itemize}

%
%
%
%
Let $\big(f(t), \xi(t)\big)$ be a homotopy pair on $\HOM_{\category{cdgH}(\Bbbk)}(B, B^\pr)$.
By the homotopy flow equation, 
${f}(t)$ is uniquely determined by $\xi(t)$ modulo an initial condition $f(0)$ such that
${f}(t) = {f}(0) + d_{B, B^\pr}\int^t_0\xi(s)\mathit{ds}$, and we can check that $f(t)$ is a family of morphisms of cdg-Hopf algebras.
We say ${f}(1)$ is \emph{homotopic to} ${f}(0)$ by the homotopy $\int^1_0\xi(t)dt$,
and denote ${f}(0)\sim {f}(1)$, which is clearly an equivalence relation.
In other words,  two morphisms $f$ and $\tilde f$  of cdg-Hopf algebras are homotopic if there is a homotopy flow connecting 
them (by the time $1$ map).  Then, we also say that $f$ and $\tilde f$ have the \emph{same homotopy type}, denoted by $[f]=[\tilde f]$.
For any diagram $\xymatrix{B\ar@/^/[r]^-{f}\ar@/_/[r]_-{\tilde f}&B^\pr\ar@/^/[r]^-{f^\pr}\ar@/_/[r]_-{\tilde f^\pr}&B^\ppr}$ 
in the category $\category{cdgH}(\fieldk)$, where $f\sim \tilde f$ and $f^\pr\sim \tilde f^\pr$, 
it is straightforward to check that
$f^\pr\circ f \sim \tilde f^\pr\circ \tilde f$, and the homotopy type of $f^\pr\circ f$
only depends on the homotopy types of $f$ and $f^\pr$, so that we have the well-defined composition
$[f^\pr]\circ_h [f] :=[f^\pr\circ f]$ of homotopy types.
A morphism  $\xymatrix{B\ar[r]^f & B^\pr}$ of cdg-Hopf algebras is a \emph{homotopy equivalence} if there
is a morphism  $\xymatrix{B & \ar[l]_{h} B^\pr}$ of cdg-Hopf algebras from the opposite direction
such that $h\circ f \sim \I_B$ and $f\circ h \sim \I_{B^\pr}$. 

The homotopy category $\mathit{ho}\category{cdgH}(\fieldk)$ of cdg-Hopf algebras over $\Bbbk$ is defined such
that the objects are cdg-Hopf algebras and morphisms are homotopy types of morphisms of cdg-Hopf algebras.  
Note that a homotopy equivalence of cdg-Hopf algebras is an isomorphism in the homotopy category 
$\mathit{ho}\category{cdgH}(\fieldk)$.

We define a homotopy pair on the morphisms of cdg-algebras as the case of cdg-Hopf algebras
but without imposing the infinitesimal coalgebra map condition.
Then we have corresponding notions for homotopy types of morphisms of cdg-algebras and a homotopy equivalence of cdg-algebras.
Thus we can form the homotopy category $\mathit{ho}\category{cdgA}(\fieldk)$ of cdg-algebras,
whose morphisms are homotopy types of morphisms of cdg-algebras.

A \emph{dg-category} $\dgcat{C}$   over $\Bbbk$
is a category enriched in the category $\category{CoCh}(\fieldk)$ 
of cochain complexes over $\fieldk$. 
A dg-category shall be distinguished from an ordinary category by putting an "overline".
We follows \cite{Simpson} for the notion of dg-tensor categories.
Denoted by $\HOM_{\dgcat{C}}(X,Y)$ with  differential $d_{\HOM_{\dgcat{C}}(X,Y)}$
is the cochain complex of morphisms from an object $X$ to another object $Y$ in a dg-category $\dgcat{C}$.
A morphism $f \in \HOM_{\dgcat{C}}(X,Y)$ between two objects $X$ and $Y$ in $\dgcat{C}$  
is an isomorphism if  
$f\in \HOM_{\dgcat{C}}(X,Y)^0$  and satisfies $d_{\HOM_{\dgcat{C}}(X,Y)}f=0$, with its inverse
$g \in \HOM_{\dgcat{C}}(Y,X)^0$ satisfying $d_{\HOM_{\dgcat{C}}(Y,X)}g=0$.

A \emph{dg-functor} $\functor{F}:\dgcat{C} \rightsquigarrow \dgcat{D}$ 
is a functor which induces cochain maps  
$\HOM_{\dgcat{C}}(X,Y)\rightarrow \HOM_{\dgcat{D}}(\functor{F}(X),\functor{F}(Y))$ for every pair $(X,Y)$ of objects in $\dgcat{C}$.
The set ${\mathsf{Nat}}(\functor{F}, \functor{G})$ of natural transformations of dg-functors 
is a cochain complex $\big({\mathsf{Nat}}(\functor{F}, \functor{G}), \bm{\d}\big)$,
where
\begin{itemize}
\item
its degree $n$ element is a collection of  morphisms 
$\eta=\{\eta^X:\functor{F}(X)\to\functor{G}(X)|X\in \text{Ob}(\dgcat{C})\}$
of degree $n$,
where $\eta^X$ is called the \emph{component} of $\eta$ at $X$, 
with the super-commuting naturalness condition, i.e. $\functor{G}(f)\circ\eta^X=(-1)^{mn}\eta^Y\circ\functor{F}(f)$ 
for every morphism $f:X\to Y$ of degree $m$.

\item
for every $\eta \in {\mathsf{Nat}}(\functor{F}, \functor{G})$ of degree $n$ 
we have $\d\eta \in {\mathsf{Nat}}(\functor{F}, \functor{G})$ of degree $n+1$,
whose component at $X$ is defined by $\big(\d\eta\big)^X:=d_{\HOM_{\dgcat{C}}\big(\functor{F}(X),\functor{G}(X)\big)}\eta^X$,
and $\d\circ \d=0$.
\end{itemize}
The dg-functors from $\dgcat{C}$ to $\dgcat{D}$ form a dg-category, with morphisms 
as the above natural transformations. 
In particular, the set $\mathsf{End}(\functor{F}):=\mathsf{Nat}(\functor{F},\functor{F})$ of natural endomorphisms  has a canonical structure of dg-algebra.
A natural  transformation $\eta$ 
from a dg-functor $\functor{F}$ to another dg-functor
$\functor{G}$ is often indicated by a diagram $\eta: \functor{F}\Rightarrow \functor{G}: \dgcat{C}\rightsquigarrow \dgcat{D}$.
A natural transformation $\eta$ is an \emph{(natural) isomorphism} 
if the component morphism  $\eta^X:\functor{F}(X)\rightarrow \functor{G}(X)$ is an isomorphism in
$ \dgcat{D}$ for every object $X$ of  $\dgcat{C}$.

The notion of tensor categories \cite{Rivano,Deligne90} has a natural generalization to dg-tensor categories.
For a dg-category $\dgcat{C}$ we have a new dg-category $\dgcat{C}\boxtimes \dgcat{C}$, whose objects
are pairs denoted by $X\boxtimes Y$ and whose Hom complexes are the tensor products of
Hom complexes of  $\dgcat{C}$, i.e., $\HOM_{\dgcat{C}\boxtimes \dgcat{C}}(X\boxtimes Y, X^\pr\boxtimes Y^\pr)
=\HOM_{\dgcat{C}}(X, X^\pr)\otimes \HOM_{\dgcat{C}}(Y, Y^\pr)$ with the natural composition operation and 
differentials.  Then  we have a natural equivalence of dg-categories  
$(\dgcat{C}\boxtimes \dgcat{C})\boxtimes \dgcat{C} \cong \dgcat{C}\boxtimes (\dgcat{C}\boxtimes \dgcat{C})$.
A dg-category $\dgcat{C}$ is a \emph{dg-tensor category} if 
we have dg-functor $\bm{\otimes}: \dgcat{C}\boxtimes \dgcat{C}\rightsquigarrow \dgcat{C}$
and a unit object $\bm{1}_{\dgcat{C}}$ satisfying the associativity, the commutativity and the unit axioms
subject to coherence conditions.
(See pp $40$-$41$ in  \cite{Simpson} for the details.)

The fundamental example of dg-tensor 
categories over $\Bbbk$ is the dg-category $\dgcat{CoCh}(\Bbbk)$ of cochain complexes, whose 
set of morphisms $\HOM_{\dgcat{CoCh}(\Bbbk)}(V,W)$
from a cochain complex $V$ to a cochain complex $W$
is the $\Bbbk$-linear Hom complex $\Hom(V,W)$ with the differential $d_{\HOM_{\dgcat{CoCh}(\Bbbk)}(V,W)}=d_{V\!,W}$.  
The dg-functor $\bm{\otimes}: \dgcat{CoCh}(\Bbbk)\boxtimes \dgcat{CoCh}(\Bbbk)\rightsquigarrow \dgcat{CoCh}(\Bbbk)$
sends $(V, d_V)\boxtimes(W, d_W)$ to the  cochain complex $(V\otimes W, d_{V\otimes W})$ 
and the unit object is the ground field $\Bbbk$ as a cochain complex $(\Bbbk,0)$,
where all coherence isomorphisms  are the obvious ones.

A \emph{dg-tensor functor} $\functor{F}:\dgcat{C}\rightsquigarrow\dgcat{D}$ 
between dg-tensor categories is a dg-functor 
satisfying
$\functor{F}(X\bm{\otimes} Y)\cong\functor{F}(X)\bm{\otimes}\functor{F}(Y)$
and $\functor{F}(\bm{1}_{\dgcat{C}})\cong
\bm{1}_{\dgcat{D}}$.
A \emph{tensor natural transformation} $\eta:\functor{F}\Rightarrow\functor{G}$ 
of dg-tensor functors is a natural transformation  of degree $0$
satisfying $\eta^{X\bm{\otimes} Y}\cong\eta^X\otimes\eta^Y$ 
and $\eta^{\bm{1}_{\dgcat{C}}}\cong\I_{\bm{1}_{\dgcat{D}}}$.

We use the  notation $[\a]$ for the homotopy type of a morphism 
$\a$ as well as for the cohomology class of  a cocycle $\a$, depending on the context.

\section{Affine group dg-schemes}

An affine group dg-scheme  over $\Bbbk$
is a representable functor
$\bm{\mG}:\mathit{ho}\category{cdgA}(\fieldk)\rightsquigarrow \category{Grp}$ 
from 
the homotopy category $\mathit{ho}\category{cdgA}(\fieldk)$ of cdg-algebras
over $\fieldk$ to the category $\category{Grp}$ of groups.
A representing object of $\bm{\mG}$ is a 
cdg-Hopf algebra $B$, and we use the notation $\bm{\mG}^B$ for it. 
The category formed by affine  group dg-schemes over $\Bbbk$ is anti-equivalent to the
homotopy category $\mathit{ho}\category{cdgH}(\fieldk)$ of cdg-Hopf algebras over $\Bbbk$.

For each cdg-Hopf algebra $B$, we also construct a functor 
$\category{T}\!\bm{\mG}^B:\mathit{ho}\category{cdgA}(\fieldk)\rightsquigarrow \category{Lie}(\Bbbk)$ to the category $\category{Lie}(\Bbbk)$ of Lie algebras
so that $\category{T}\!\bm{\mG}^B$ and $\category{T}\!\bm{\mG}^{B^\pr}$ are naturally isomorphic whenever 
$\bm{\mG}^B$ and $\bm{\mG}^{B^\pr}$ are naturally isomorphic.

\subsection{Representable functors $\bm{\mG}:\mathit{ho}\category{cdgA}(\fieldk)\rightsquigarrow \category{Grp}$ 
and cdg-Hopf algebras}

The purpose of this subsection is to prove the following:

\begin{theorem}[Definition]\label{affinegrdgsch}
For every cdg-Hopf algebra  $B=\big(B, u_B, m_B, \ep_B, \cp_B,\vs_B, d_B\big)$ 
we have a  functor $\bm{\mG}^B:\mathit{ho}\category{cdgA}(\fieldk)\rightsquigarrow \category{Grp}$
represented by $B$ defined
as follows:

\begin{itemize}

\item for each  cdg-algebra $\big(A, u_A, m_A, d_A\big)$ we have a  group $\bm{\mG}^B(A)$ defined by
$$\bm{\mG}^B(A):=\Big(\HOM_\hcdga(B,A), \ide_{B\!,\,A}, \ast_{\!{B\!,\,A}} \Big)$$
with  the identity element $\ide_{B\!,\,A}:=\big[ u_A\circ \ep_B\big]$,  the group operation 
$[g_1]\ast_{B\!,\,A} [g_2] :=\big[m_A\circ (g_1\otimes g_2) \circ\cp_B\big]$,  and  the inverse $[g]^{-1}:=\big[g\circ \vs_B\big]$ of $[g]$,
where  $g\in \HOM_\cdga(B,A)$ is an arbitrary representative
of the homotopy type $[g]\in \HOM_\hcdga(B,A)$;

\item for each morphism $[f]\in \HOM_\hcdga(A,A^\pr)$ 
we have a group homomorphism $\bm{\mG}^B\big([f]\big): \bm{\mG}^B(A)\rightarrow \bm{\mG}^B(A^\pr)$ 
defined by, $\forall [g]\in \HOM_\hcdga(B,A)$,
$$
\bm{\mG}^B([f])\big([g]\big):=\big[f\circ g\big],
$$
where 
$f \in \HOM_{\cdga}\big(A, A^\pr\big)$ and
$g \in   \HOM_\cdga(B,A)$ 
are  arbitrary representatives
of the homotopy types $[f]$ and $[g]$, respectively,
such that $\bm{\mG}^B([f])$ is an isomorphism of groups
whenever $f: A \rightarrow A^\pr$ is a homotopy equivalence of cdg-algebras.

\end{itemize}

For each morphism $[\p] \in \HOM_{\hcdgh}(B,B^\pr)$ in the homotopy category 
of cdg-Hopf algebras,  we have a natural transformation 
$\sN_{[\p]}\in\mathsf{Nat}\big(\bm{\mG}^{B^\pr},\bm{\mG}^{B}\big)$ defined as follows:
for each cdg-algebra $A$ and for every $[g^\pr] \in \HOM_{\hcdga}\big(B^\pr, A\big)$ we have
$\sN_{[\p]}^{A}\big([g^\pr]\big):= \big[g^\pr\circ \p\big]$,
where 
$\p \in \HOM_{\cdgh}\big(B, B^\pr\big)$ and $g^\pr \in \HOM_{\cdga}\big(B^\pr, A\big)$ 
are 
arbitrary representative of the homotopy types $[\p]$ and $[g^\pr]$. 

We have $\HOM_{\hcdgh}(B,B^\pr)\cong \mathsf{Nat}\big(\bm{\mG}^{B^\pr},\bm{\mG}^{B}\big)$
such that $\sN_{[\p]}$ is a natural isomorphism  whenever $\p: B\rightarrow B^\pr$ is a homotopy equivalence of 
cdg-Hopf algebras.
\end{theorem}

We need the forthcoming three lemmas for the proof.

\begin{lemma}\label{grdgschemeone}
For every cdg-Hopf algebra $B$
we have a functor 
$\bm{\CE}^{B}:{\category{cdgA}}(\Bbbk) \rightsquigarrow \category{dgA}(\Bbbk)$,
sending 
\begin{itemize}
\item each cdg-algebra $A$ to the  dg-algebra
$\bm{\CE}^{B}(A):=\big(\Hom\big(B, A\big), u_A\circ\ep_B, \star_{B\!,A}, d_{B\!,A}\big)$,
where $\a_1\star_{B\!, A}\a_2:= m_A\circ (\a_1\otimes \a_2)\circ \cp_B$ for all $\a_1,\a_2 \in \Hom(B, A)$,
and 

\item each morphism $A\xrightarrow{f} A^\pr$ of cdg-algebras
to a morphism $\bm{\CE}^B(f): \bm{\CE}^{B}(A)\rightarrow\bm{\CE}^{B}(A^\pr)$
of dg-algebras defined by $\bm{\CE}^{B}(f)(\a):= f\circ \a$ for all $\a \in \Hom\big(B, A\big)$.
\end{itemize}
\end{lemma}

\begin{proof}
It is a standard  fact that $\bm{\CE}^{B}(A)$ is a dg-algebra.
We check that $\bm{\CE}^B(f)$ is a morphism of
dg-algebras, 
as follows:
\eqalign{
\bm{\CE}^B(f)(u_A\!\circ\!\ep_{B}) = & f\circ u_A\circ\ep_{B} = u_{A^\pr}\!\circ\!\ep_B
,\cr
\bm{\CE}^B(f)\left(d_{B\!,A}\a\right) 
= 
&
f\circ d_{A}\circ \a -(-1)^{|\a|} f\circ\a\circ d_{B}
\cr
=
&
d_{A^\pr}\circ f\circ \a -(-1)^{|\a|} f\circ\a\circ d_{B}
=
d_{B\!,A^\pr}\left(\bm{\CE}^B(f)(\a)\right)
,\cr
\bm{\CE}^B(f)\left(\a_1\star_{B\!, A}\a_2\right) 
= &f\circ  m_A\circ (\a_1\otimes \a_2)\circ \cp_{B}
= m_{A^\pr}\circ (f\circ\a_1 \otimes f\circ\a_2)\circ \cp_{B}
\cr
=&
\bm{\CE}^B(f)(\a_1)\star_{B\!,A^\pr}\bm{\CE}^B(f)(\a_2),
}
The functoriality of $\bm{\CE}^B$ is obvious.
\qed
\end{proof}

\begin{lemma}\label{grdgschemetwo}
For every cdg-Hopf algebra $B$
we have a functor $\bm{\CG}^B:{\category{cdgA}}(\Bbbk) \rightsquigarrow \category{Grp}$
represented by $B$,
sending 
\begin{itemize}
\item
each cdg-algebra $A$ to a
group  
$\bm{\CG}^B(A):=\big(\HOM_{\cdga}\big(B, A\big), u_A\circ\ep_B, \star_{B\!,A}\big)$,
where the inverse $g^{-1}$ of  $g\in \HOM_{\cdga}\big(B, A\big)$ is $g^{-1}:=g\circ\vs_B$, 
and 
\item
each morphism $A\xrightarrow{f} A^\pr$ of cdg-algebras
to a morphism $\bm{\CG}^B(f): \bm{\CG}^B(A)\rightarrow\bm{\CG}^B(A^\pr)$
of groups defined by $\bm{\CG}^B(f)(g):= f\circ g$ for all $g\in \HOM_{\cdga}\big(B, A\big)$.
\end{itemize}
\end{lemma}

\begin{proof}
1. We show that $\bm{\CG}^B(A)$ is a group for every cdg-algebra $A$.
By Lemma \ref{grdgschemeone} the tuple $\big(\Hom(B,A), u_A\circ\ep_B, \star_{B\!,\,A},  {d}_{B\!,\,A}\big)$ is
a dg-algebra, and
$$
\HOM_\cdga\big(B,A\big):= \left\{g\in \Hom\big(B, A\big)^0 \Big| {d}_{B\!,\,A}g=0,\;\a\circ u_B = u_A,\; g\circ m_B = m_A\circ (g\otimes g)\right\}.
$$
Therefore, 
we need to check that
$u_A\circ\ep_B, g_1\star_{B\!,\,A}g_2, g^{-1} \in \HOM_\cdga\big(B,A\big)$
and $g\star_{B\!,\,A} g^{-1} =g^{-1}\star_{B\!,\,A} g = u_A\circ\ep_B$
for all $g_1,g_2,g \in \HOM_\cdga \big(B,A\big)$.

\begin{itemize}
\item $u_A\circ\ep_B\in \HOM_\cdga(B,A)$: 
It is trivial that ${d}_{B\!,\,A} (u_A\circ\ep_B)=d_A\circ u_A\circ \ep_B -u_A\circ\ep_B\circ d_B=0$
and $(u_A\circ\ep_B)\circ u_B=u_A$.  We check that $u_A\circ\ep_B$ is an algebra map as follows:
\eqalign{
u_A\circ\ep_B\circ m_B &=u_A\circ m_\fieldk\circ (\ep_B\otimes \ep_B)
,\cr
m_A\circ\left( ( u_A\circ\ep_B)\otimes (u_A\circ\ep_B)\right)
&=m_A\circ (u_A\otimes u_A)\circ  (\ep_B\otimes \ep_B)= u_A\circ m_\fieldk\circ (\ep_B\otimes \ep_B)
.
}

\item ${g}_1\star_{B\!,\,A}{g}_2 \in \HOM_\cdga(B,A)$:
Note that
$d_{B\!,\,A}({g}_1\star_{B\!,\,A} {g}_2) = {d}_{B\!,\,A}{g}_1\star_{B,A} {g}_2 + {g}_1\star_{B\!,\,A} {d}_{B\!,\,A}{g}_2 =0$.
The property $({g}_1\star_{B\!,\,A}{g}_2)\circ u_B=u_A$ can be checked as follows:
\eqalign{
({g}_1\star_{B\!,\,A}{g}_2)\circ u_B = & m_A\circ ({g}_1\otimes {g}_2)\circ \cp_B\circ u_B = m_A\circ ({g}_1\otimes {g}_2)\circ (u_B\otimes u_B)\circ \cp_\fieldk
\cr
=&m_A\circ (u_A\otimes u_A)\circ\cp_\fieldk =u_A.
}
Now we   check that ${g}_1\star_{B\!,A}{g}_2$ is an algebra map. Consider
\eqalign{
({g}_1\star_{B\!,\,A} {g}_2)\circ m_B  =
&m_A\circ ({g}_1\otimes {g}_2)\circ \cp_B\circ m_B
\cr
=
&
m_A\circ ({g}_1\otimes {g}_2)\circ (m_B\otimes m_B)\circ (\I_B\otimes \t\otimes \I_B)\circ (\cp_B\otimes \cp_B)
\cr
=
&m_A\circ (m_A\otimes m_A)\circ ({g}_1\otimes {g}_1\otimes {g}_2\otimes {g}_2) \circ (\I_B\otimes \t\otimes \I_B)\circ (\cp_B\otimes \cp_B).
}
From the commutativity $m_A\circ\t =m_A$ of $m_A$ we have
$m_A\circ (m_A\otimes m_A)\circ ({g}_1\otimes {g}_1\otimes {g}_2\otimes {g}_2) \circ (\I_B\otimes \t\otimes \I_B)
=m_A\circ (m_A\otimes m_A)\circ ({g}_1\otimes {g}_2\otimes {g}_1\otimes {g}_2)$, so that
\eqalign{
({g}_1\star {g}_2)\circ m_B  =
&
m_A\circ (m_A\otimes m_A)\circ ({g}_1\otimes {g}_2\otimes {g}_1\otimes {g}_2)\circ (\cp_B\otimes \cp_B)
\cr
=
&m_A\circ\big( ({g}_1\star_{B\!,\,A}{g}_2)\otimes ({g}_1\star_{B\!,\,A}{g}_2)\big).
}

\item $g^{-1}\in \HOM_\cdga(B,A)$:
From the condition $d_B\circ\vs_B =\vs_B\circ d_B$, we have ${d}_{B\!,\,A}{g}^{-1}=0$.
It is trivial that ${g}^{-1}\circ u_B = {g}\circ\vs_B\circ u_B ={g}\circ u_B = u_A$.
We  check that ${g}^{-1}$ is an algebra map as follows: 
\eqalign{
{g}^{-1}\circ m_B &= {g}\circ\vs_B\circ m_B= {g}\circ m_B \circ (\vs_B\otimes \vs_B)\circ \t = {g}\circ m_B \circ (\vs_B\otimes \vs_B)
\cr
&=m_A\circ({g}\otimes {g})\circ  (\vs_B\otimes \vs_B)=m_A\circ ({g}^{-1}\otimes {g}^{-1})
,
}
where we have used the  commutativity $m_B\circ \t =m_B$ of $m_B$. 

\item
${g}\star_{B\!,\,A}{g}^{-1} = m_A\circ ({g}\otimes {g})\circ(\I_B\otimes\vs_B)\circ\cp_B= {g}\circ m_B\circ(\I_B\otimes\vs_B)\circ\cp_B
={g}\circ u_B \circ \ep_B =u_A\circ\ep_B$.

\item
${g}^{-1}\star_{B\!,\,A}{g} = m_A\circ ({g}\otimes {g})\circ(\vs_B\otimes\I_B)\circ\cp_B= {g}\circ m_B\circ(\vs_B\otimes\I_B)\circ\cp_B
={g}\circ u_B \circ \ep_B =u_A\circ\ep_B$.
\end{itemize}

2. We show that 
$\bm{\CG}^B(f): \bm{\CG}^B(A)\rightarrow\bm{\CG}^B(A^\pr)$ is a group homomorphism.
We first check that $\bm{\CG}^B(f)(g)\in \HOM_\cdga(B,A^\pr)$
whenever  $g \in \HOM_\cdga(B,A)$:
\eqalign{
{d}_{{B\!,\,A}^\pr}\left( \bm{\CG}^B(f)({g})\right)=&{d}_{{B\!,\,A}^\pr}\big(f\circ {g}\big) =d_{A^\pr}\circ f \circ {g} -  f \circ {g}\circ d_B
= \big(d_{A^\pr}\circ f - f \circ d_A\big)\circ {g} =0
,\cr
\bm{\CG}^B(f)({g})\circ u_B =& f\circ{g} \circ u_B = f\circ u_A = u_{A^\pr}
,\cr
\bm{\CG}^B(f)({g})\circ m_{B} = & f\circ{g}\circ m_B
= f\circ m_A\circ ({g}\otimes {g})=
m_{A^\pr}\circ (f\otimes f)\circ ({g}\otimes {g})
\cr
=&
m_{A^\pr} \circ \big(\bm{\CG}^B(f)({g})\otimes \bm{\CG}^B(f)({g})\big).
}
Now we check that $\bm{\CG}^B(f)$ is a group homomorphism:
\eqalign{
\bm{\CG}^B(f)(u_A\circ\ep_B)=&f\circ u_{A}\circ \ep_B= u_{A^\pr}\circ\ep_B
,\cr
\bm{\CG}^B(f)(g_1\star_{B\!,A}g_2)=& f\circ(g_1\star_{B,A} g_2) =f\circ m_A\circ(g_1\otimes g_2)\circ\cp_B
=m_{A^\pr}\circ (f\otimes f)\circ (g_1 \otimes g_2)\circ\cp_B 
\cr
=& (f\circ g_1)\star_{B\!,A^\pr} (f\circ g_2)=\bm{\CG}^B(f)(g_1)\star_{B\!,A^\pr} \bm{\CG}^B(f)(g_2).
}
The functoriality of $\bm{\CG}^B$ is obvious.
\qed
\end{proof}

\begin{lemma}\label{grdgschemefour}
Suppose $\bm{\CG}:\category{cdgA}(\Bbbk)\rightsquigarrow\category{Grp}$ is a representable functor. 
Then $\bm{\CG}\cong\bm{\CG}^B$ for some cdg-Hopf algebra $B$.
\end{lemma}
\begin{proof}
Since  the functor $\grave{\bm{\CG}}:\category{cdgA}(\Bbbk) \rightsquigarrow \category{Set}$ is representable, 
we have an isomorphism $\grave{\bm{\CG}}\cong\HOM_{\category{cdgA}(\Bbbk)}(B,-)$ for some cdg-algebra $B$. 
%
%
We show that $B$ in fact a cdg-Hopf algebra.
We can restate the condition of $\grave{\bm{\CG}}$ factoring through $\category{Grp}$ as follows:
\begin{itemize}
\item For each cdg-algebra $A$ there is a structure of group on $\grave{\bm{\CG}}(A)$, i.e.,
there are three functions
$\m^A:\grave{\bm{\CG}}(A)\times\grave{\bm{\CG}}(A)\to\grave{\bm{\CG}}(A)$,
$e^A:\{*\}\to\grave{\bm{\CG}}(A)$ and $i^A:\grave{\bm{\CG}}(A)\to\grave{\bm{\CG}}(A)$ satisfying the group axioms.
\item For each morphism $f:A\to A^\pr$ of cdg-algebras, 
the function $\grave{\bm{\CG}}(f):\grave{\bm{\CG}}(A)\to\grave{\bm{\CG}}(A^\pr)$ is a homomorphism of groups.
\end{itemize}
This is equivalent to the existence of  three natural transformations 
$\m:\grave{\bm{\CG}}\times\grave{\bm{\CG}}\to\grave{\bm{\CG}}$,
$e:\underline{\{*\}}\to\grave{\bm{\CG}}$ and $i:\grave{\bm{\CG}}\to\grave{\bm{\CG}}$ satisfying the group axioms.
Here, $\underline{\{*\}}$ is a functor $\category{cdgA}(\Bbbk)\to\category{Set}$ 
sending every cdg-algebra $A$ to a one-point set $\{*\}$.

Let $B\otimes B$ be the cdg-algebra obtained
by the tensor product of the cdg-algebra $B$.  We claim that there are natural isomorphisms of functors
\eqn{clniso}{
\underline{\{*\}}\cong\HOM_{\category{cdgA}(\Bbbk)}(\Bbbk,-)
,\qquad
\grave{\bm{\CG}}\times\grave{\bm{\CG}}\cong\HOM_{\category{cdgA}(\Bbbk)}(B\otimes B,-)
.
}
Then, by the Yoneda lemma, the natural transformations $\m$, $e$, and $i$
are completely determined by morphisms $\cp_B:B\to B\otimes B$, $\ep_B:B\to \Bbbk$ 
and $\vs_B:B\to B$ of cdg-algebras, respectively. 
Applying the Yoneda lemma again, a plain calculation shows that
\begin{enumerate}
\item $\m\circ(\m\times\I_{\grave{\bm{\CG}}})=\m\circ(\I_{\grave{\bm{\CG}}}\times \m)$ implies the coassociativity of $\cp_B$.
\item $\m\circ(e\times\I_{\grave{\bm{\CG}}})=\m\circ(\I_{\grave{\bm{\CG}}}\times e)=\I_{\grave{\bm{\CG}}}$ 
implies the counit axiom of $\ep_B$.
\item inverse element axiom of $i$ implies the antipode axiom of $\vs_B$. 
\end{enumerate}
Therefore $B$ has a cdg-Hopf algebra structure $(B,u_B,m_B,\ep_B,\cp_B,\vs_B,d_B)$.

Now we check the claimed isomorphisms in \eq{clniso}. 
Note that $\Bbbk$ is an initial object in the category $\category{cdgA}(\Bbbk)$
since  any morphism $\Bbbk \to A$  of cdg-algebras 
has to be the unit $u_A$ by the unit axiom. 
Let $A\otimes A^\pr$ is be the cdg-algebra obtained by the tensor product of
cdg-algebras $A$ and $A^\pr$. Then we have the following inclusion morphisms of cdg-algebras: 
\[
i_A:=
\xymatrix{
A \ar[r]^-{\jmath^{-1}_A}& A\otimes \Bbbk \ar[r]^-{\I_A\otimes u_{\!A^\pr}}& A\otimes A^\pr
,
}
\quad\quad
i_{A^\pr}:=
\xymatrix{
A^\pr \ar[r]^-{\imath^{-1}_{A^\pr}}& \Bbbk\otimes A^\pr \ar[r]^-{u_{\!A}\otimes \I_{A^\pr}}& A\otimes A^\pr
.
}
\]
For each cdg-algebra $T$, we consider the function
$$
\HOM_{\category{cdgA}(\Bbbk)}(A\otimes A^\pr,T)\rightarrow\HOM_{\category{cdgA}(\Bbbk)}(A,T)\times \HOM_{\category{cdgA}(\Bbbk)}(A^\pr,T)
$$
defined by $h\mapsto \big(h\circ i_A, h\circ i_{A^\pr}\big)$. 
We show the function is a bijection for every $T$ by constructing its inverse.
Given morphisms $f:A\to T$ and $g:A^\pr \to T$ of cdg-algebras, define 
$\langle f, g\rangle:=m_T\circ (f\otimes g):A\otimes A^\pr\to T$,
which is a morphism of cdg-algebras since $m_T$ is commutative.
It is obvious that $\langle h\circ i_A,h\circ i_{A^\pr}\rangle=h$, $\langle f, g\rangle\circ i_A=f$ and $\langle f, g\rangle\circ i_{A^\pr}=g$. 
Moreover, the above bijection is natural in $T\in\category{cdgA}(\Bbbk)$.
Therefore we have constructed the claimed isomorphisms in \eq{clniso}.

Let $x_1,x_2$ be the elements in $\grave{\bm{\CG}}(A)$ that correspond
to morphisms $g_1,g_2:B\to A$ of cdg-algebras via the bijection 
$\grave{\bm{\CG}}(A)\cong\grave{\bm{\CG}}^{_B}(A)$. 
Then, by the Yoneda lemma, $\m^A(x_1,x_2)\in\grave{\bm{\CG}}(A)$ corresponds to 
$\langle g_1,g_2\rangle\circ \cp_B=g_1\star_{B,A}g_2\in\grave{\bm{\CG}}^{_B}(A)$, and
$e^A\in\grave{\bm{\CG}}(A)$ corresponds to $u_A\circ \ep_B\in\grave{\bm{\CG}}^{_B}(A)$.
This shows that the bijection $\grave{\bm{\CG}}(A)\cong\grave{\bm{\CG}}^{_B}(A)$ is an isomorphism of groups
for every cdg-algebra $A$.
Thus we have a natural isomorphism $\bm{\CG}\cong\bm{\CG}^B$ of functors 
from $\category{cdgA}(\Bbbk)$ to $\category{Grp}$.
\qed
\end{proof}

\begin{lemma}\label{grdgschemethree}
For every morphism $\p:B \rightarrow B^\pr$ of cdg-Hopf algebras we have a natural
transformation $\sN_{\p}: \bm{\CG}^{B^\pr} \Longrightarrow \bm{\CG}^{B}: \category{cdgA}(\Bbbk) \rightsquigarrow \category{Grp}$ functors,
whose component $\sN^A_{\psi}: \bm{\CG}^{B^\pr}(A) \rightarrow \bm{\CG}^{B} (A)$ 
at each cdg-algebra $A$ is defined by $\sN_{\p}^A(g^\pr) :=g^\pr\circ \p$ for all $g^\pr \in \HOM_{\cdga}\big(B^\pr, A\big)$. 
We also have $\HOM_{\cdgh}\big(B, B^\pr\big)\cong \mathsf{Nat}\big(\bm{\CG}^{B^\pr},\bm{\CG}^{B}\big)$.
\end{lemma}

\begin{proof}
We show that $\sN_{\p}$ is a natural transformation for every $\p \in \HOM_{\cdgh}\big(B, B^\pr)$,
which is a linear map $\p: B\rightarrow B^\pr$  satisfying the following relations
$$
d_{B^\pr}\circ \p = \p \circ d_B
,\quad
\begin{cases}
\p\circ u_B=u_{B^\pr}
,\cr
\p\circ m_B = m_{B^\pr}\circ (\p\otimes \p)
\end{cases}
,\quad
\begin{cases}
\ep_B = \ep_{B^\pr} \circ \p
,\cr
\cp_{B^\pr}\circ \p=
(\p\otimes \p)\circ \cp_B
,
\end{cases}
$$
as follows.
\begin{itemize}
\item $\sN_{\p}^A(g^\pr) \in \HOM_{\cdga}\big(B, A\big)$ for all $g^\pr \in \HOM_{\cdga}\big(B^\pr, A\big)$:
\begin{align*}
\sN_{\p}^A(g^\pr)\circ u_B
=
&
g^\pr\circ \p\circ u_B
=
g^\pr\circ u_{B^\pr}
=
u_A
,\cr
{d}_{B\!,\,A}\left( \sN_{\p}^A(g^\pr)\right)=
&
g^\pr\circ \p \circ d_B
-d_A\circ g^\pr\circ \p
=
g^\pr\circ d_{B^\pr}\circ\p - g^\pr \circ d_{B^\pr}\circ \p
=
0
,\cr
\sN_{\p}^A(g^\pr)\circ m_B 
=
&
g^\pr\circ\p\circ m_B
=g^\pr\circ m_{B^\pr}\circ (\p\otimes \p)
=m_{A}\circ (g^\pr\otimes g^\pr)\circ (\p\otimes \p)
\cr
=
&
m_A\circ \left(\sN_{\p}^A(g^\pr)\otimes\sN_{\p}^A(g^\pr)\right).
\end{align*}

\item 
$\sN^A_{\psi}: \bm{\CG}^{B^\pr}(A) \longrightarrow \bm{\CG}^{B} (A)$ is a group homomorphism: 
$\forall g_1^\pr, g_2^\pr \in \HOM_{\cdga}\big(B^\pr, A\big)$,
\eqalign{
\sN_{\p}^A\big(u_A\circ \ep_{B^\pr}\big)=& u_A\circ \ep_{B^\pr}\circ \p=u_A\circ \ep_B
,\cr
\sN_{\p}^A\big(g_1^\pr \star_{B^\pr\!,A} g_2^\pr\big)
=
&
m_A\circ (g_1^\pr \otimes g_2^\pr)\circ \cp_{B^\pr}\circ\p
=
m_A\circ \big(g_1^\pr\circ \p \otimes g_2^\pr\circ \p\big)\circ\cp_B
\cr
=
&
\sN_{\p}^A\big(g_1^\pr\big) \star_{B\!,A} \sN_{\p}^A(g_2^\pr\big).
}

\item 
$\sN_{\p}^{A^\pr}\circ \bm{\CG}^B\big(f\big)= \bm{\CG}^{B^\pr}\big(f\big)\circ\sN_{\p}^A$
for  every $f\in \HOM_{\cdga}\big(A, A^\pr\big)$:
This is trivial, since for all $g^\pr \in \HOM_{\cdga}(B^\pr, A)$ we have
\eqalign{
\sN_{\p}^{A^\pr}\circ \bm{\CG}^B\big(f\big)(g^\pr)=    \big(f\circ g^\pr\big)\circ \p
=f\circ\big(g^\pr\circ \p\big)=
\bm{\CG}^{B^\pr}\big(f\big)\circ\sN_{\p}^A(g^\pr).
}
\end{itemize}
Therefore $\sN_{\psi} \in \mathsf{Nat}\big(\bm{\CG}^{B^\pr}, \bm{\CG}^{B}\big)$ whenever 
${\psi} \in \HOM_{\cdgh}\big(B, B^\pr\big)$. Combined with the Yoneda lemma, we 
conclude that
$\mathsf{Nat}\big(\bm{\CG}^{B^\pr}, \bm{\CG}^{B}\big)\cong  \HOM_{\cdgh}\big(B, B^\pr\big)$.
\qed
\end{proof}

\begin{lemma}\label{hprt}
If we have homotopy pairs
\begin{itemize}
\item $\big(f(t),\s(t)\big)$ on $\HOM_{\cdga}(A,A^\pr)$;

\item $\big({g}(t),\l(t)\big)$,  $\big({g}_1(t),\l_1(t)\big)$ and  $\big({g}_2(t),\l_2(t)\big)$
 on $\HOM_{\cdga}(B,A)$;

\item
$\big(\p(t),\xi(t)\big)$ on $\HOM_{\cdgh}(B,B^\pr)$;

\item $\big({g}^\pr(t),\l^\pr(t)\big)$  on $\HOM_{\cdga}(B^\pr,A)$.
\end{itemize}

Then we also have  homotopy pairs

\begin{enumerate}[label=({\alph*}),leftmargin=0.8cm]

\item
$\big({g}_1(t)\star_{B\!,A}{g}_2(t), \l_1(t)\star_{B\!,A}{g}_2(t) + {g}_1(t)\star_{B\!,A}\l_2(t) \big)$
 on $\HOM_{\cdga}(B,A)$;

\item
$\big({g}(t)\circ\vs_B, \l(t)\circ\vs_B\big)$ on $\HOM_{\cdga}(B,A)$;

\item
$\big(f(t)\circ{g}(t), f(t)\circ \l(t) +\s(t)\circ {g}(t)\big)$  on $\HOM_{\cdga}(B,A^\pr)$;

\item
$\big({g}^\pr(t)\circ \p(t), \l^\pr(t)\circ \p(t)+ {g}^\pr(t)\circ \xi(t)\big)$  on $\HOM_{\cdga}(B,A)$.
\end{enumerate}
\end{lemma}

\begin{proof}
These can be checked by routine computations, which are omitted for the sake of space.
\qed
\end{proof}

Now we are ready for the proof of  Theorem \ref{affinegrdgsch}.

\begin{proof}[Theorem \ref{affinegrdgsch}]
After Lemmas \ref{grdgschemeone}, \ref{grdgschemetwo}, \ref{grdgschemefour}, \ref{grdgschemethree}
and \ref{hprt},  we just need to check  few things to finish the proof.

1.  We check that the group $\bm{\mG}^B(A)$ is well-defined for every cdg-algebra $A$.
\begin{itemize}
\item We have $g_1\star_{B\!,A} g_2 \sim \tilde g_1\star_{B\!,A} \tilde g_2 \in  \HOM_{\cdga}\big(B, A\big)$
whenever $g_1\sim \tilde g_1, g_2\sim \tilde g_2\in \HOM_{\cdga}\big(B, A\big)$: this follows from
Lemma \ref{hprt}$(a)$.
 
\item  We have $g^{-1}\sim \tilde{g}^{-1} \in  \HOM_{\cdga}\big(B, A\big)$
whenever $g\sim \tilde g \in \HOM_{\cdga}\big(B, A\big)$: this follows from
Lemma \ref{hprt}$(b)$.
\end{itemize}
Moreover, the homotopy type
$[g_1\star_{B\!,A} g_2]$ of $g_1\star_{B\!,A} g_2$ depends only on the homotopy types
$[g_1],[g_2]  \in \HOM_{\hcdga}\big(B, A\big)$ of $g_1$, $g_2$.
Therefore the group  $\bm{\mG}^B(A)$ is well-defined.

2. We check that the homomorphism $\bm{\mG}^B([f]):\bm{\mG}^B(A)\rightarrow \bm{\mG}^B(A^\pr)$ of groups
is well-defined for every $[f]\in \HOM_\hcdga(A,A^\pr)$.  Let $f\sim \tilde f \in \HOM_\cdga(A,A^\pr)$ and
$g\sim \tilde g \in   \HOM_\cdga(B,A)$. Then, by Lemma \ref{hprt}$(b)$, we have
$f\circ g \sim f\circ \tilde g\sim \tilde f\circ g\sim \tilde f\circ \tilde g \in \HOM_\cdga(B, A^\pr)$
so that $\bm{\mG}^B([f])([g])=[f\circ g]$ depends only on the homotopy types
$[f]$ and $[g]$.  Therefore $\bm{\mG}^B([f])$ is  a well-defined group homomorphism.
It is obvious that $\bm{\mG}^B([f])$ is an isomorphism of groups
whenever $f: A \rightarrow A^\pr$ is a homotopy equivalence of cdg-algebras.

3.  We check that the natural transformation 
$\sN_{[\p]}: \bm{\mG}^{B^\pr}\Longrightarrow \bm{\mG}^B:
\mathit{ho}\category{cdgA}(\fieldk)\rightsquigarrow \category{Grp}$ is well-defined
for every  $[{\psi}]\in \HOM_{\hcdgh}\big(B, B^\pr\big)$.
Let $\p\sim \tilde \p \in \HOM_\cdgh(B,B^\pr)$ and
$g^\pr\sim\tilde g^\pr \in   \HOM_\cdga(B^\pr,A)$. Then, by Lemma \ref{hprt}$(d)$, we have
$g^\pr\circ \p \sim g^\pr\circ \tilde \p\sim \tilde g^\pr\circ \p\sim \tilde g^\pr\circ \tilde \p \in \HOM_\cdga(B, A)$
so that $\sN^A_{[\p]}([g^\pr])=[g^\pr\circ \p]$ for every cdg-algebra $A$ depends only on the homotopy types
$[\p]$ and $[g^\pr]$.
Therefore the natural transformation 
$\sN_{[\p]}: \bm{\mG}^{B^\pr}\Longrightarrow \bm{\mG}^B$ is well-defined such that
$\sN_{[{\psi}]}\in \mathsf{Nat}\big(\bm{\mG}^{B^\pr}, \bm{\mG}^B\big)$ whenever $[{\psi}]\in \HOM_{\hcdgh}\big(B, B^\pr\big)$.
Combined with the Yoneda lemma, we have 
$$
\mathsf{Nat}\big(\bm{\mG}^{B^\pr}, \bm{\mG}^B\big)\cong \HOM_{\hcdgh}\big(B, B^\pr\big)
.
$$
That is, the category of affine group dg-schemes over $\Bbbk$ is anti-equivalent to the homotopy category $\mathit{ho}\category{cdgH}(\Bbbk)$ 
of cdg-Hopf algebras over $\Bbbk$.
It is obvious that $\sN^A_{[{\psi}]}: \bm{\mG}^{B^\pr}(A)\rightarrow\bm{\mG}^B(A)$  is an isomorphism of groups 
for every cdg-algebra $A$ whenever $\p:B\rightarrow B^\pr$ is a homotopy equivalence of cdg-Hopf algebras. Therefore
$\sN_{[{\psi}]}$  is a natural isomorphism whenever $\p$ is  a homotopy equivalence of cdg-Hopf algebras.
\qed
\end{proof}

%

\subsection{$\bm{\mG}(\Bbbk)$ action on the affine dg-scheme 
$\grave{\bm{\mG}}:\mathit{ho}\category{cdgA}(\fieldk)\rightsquigarrow \category{Set}$.}

Fix an affine group dg-scheme $\bm{\mG}:\mathit{ho}\category{cdgA}(\fieldk)\rightsquigarrow \category{Grp}$ over $\Bbbk$.
Note that the ground field $\Bbbk$ is an initial object in $\mathit{ho}\category{cdgA}(\fieldk)$.
Any cdg-algebra $A=(A, u_A, m_A, d_A)$ comes with the cdg-algebra map $u_A\in \HOM_{\cdga}\big( \Bbbk,A\big)$,
which induces a canonical group homomorphism $\bm{\mG}([u_A]):\bm{\mG}(\Bbbk)\rightarrow \bm{\mG}(A)$.

\begin{lemma}
For every affine group dg-scheme ${\bm{\mG}}$
the underlying representable functor  
$$
\grave{\bm{\mG}}
=\HOM_{\hcdga}\big(B, -\big): \mathit{ho}\category{cdgA}(\fieldk)\rightsquigarrow \category{Set},
$$
is  $\bm{\mG}(\Bbbk)$-set valued,
such that
for each cdg-algebra $A$,  $\grave{\bm{\mG}}(A)$ is a $\bm{\mG}(\Bbbk)$-set with the action
$r^A:  \bm{\mG}(\Bbbk)\times\grave{\bm{\mG}}(A)\rightarrow \grave{\bm{\mG}}(A)$,
defined by, for all $[g] \in  \bm{\mG}(\Bbbk)$ and $[x]  \in \grave{\bm{\mG}}(A)$,
$$
r^A\big([g], [x]\big):=\bm{\mG}^B\big([u_A]\big)([g])\ast_{B\!,A}[x]
=[g.x],\qquad g.x = (u_A\circ g)\star_{B\!,A} x,
$$
where $g$ and $x$ are representatives of $[g]$ and $[x]$, respectively, 
and for every morphism $[f]\in \HOM_{\hcdga}\big(A, A^\pr\big)$ the following diagram commutes:
$$
\xymatrix{
\ar[d]_-{\I_{\bm{\mG}(\Bbbk)}\times \grave{\bm{\mG}}([f])}
\bm{\mG}(\Bbbk)\times\grave{\bm{\mG}}\big(A\big)\ar[rr]^-{r^A} & &\grave{\bm{\mG}}\big(A\big)\ar[d]^-{\grave{\bm{\mG}}([f])}
\cr
\bm{\mG}(\Bbbk)\times \grave{\bm{\mG}}\big(A^\pr\big)\ar[rr]^-{r^{A^\pr}} & &\grave{\bm{\mG}}\big(A^\pr\big)
.
}
$$
\end{lemma}

\begin{proof}This is a corollary of Theorem \ref{affinegrdgsch}, whose details are omitted.
\qed
\end{proof}

%
%
%

\subsection{Functor $\mb{T}\!\bm{\mG}: \mathit{ho}\category{cdgA}(\Bbbk) \rightsquigarrow \category{Lie}(\Bbbk)$}

For each cdg-Hopf algebra $B$ we construct a functor 
$\category{T}\bm{\mG}^B: \mathit{ho}\category{cdgA}(\Bbbk) \rightsquigarrow \category{Lie}(\Bbbk)$ 
 to the category $\category{Lie}(\Bbbk)$ 
of Lie algebras over $\Bbbk$, so that we have
a natural isomorphism $\category{T}\bm{\mG}^B\cong \category{T}\bm{\mG}^{B^\pr}$
of functors whenever $\bm{\mG}^B\cong \bm{\mG}^{B^\pr}$.
Moreover, the assignment  $\bm{\mG}^B \rightsquigarrow \category{T}\!\bm{\mG}^B$ 
for every cdg-Hopf algebra $B$  is functorial
-- there is a functor $\category{T}$ from the category of affine group dg-schemes
to the category of functors from  $\mathit{ho}\category{cdgA}(\Bbbk)$ to $\category{Lie}(\Bbbk)$.

For each cdg-algebra $A$, 
we define $\THOM_{\cdga}(B, A)$  as the set of all \emph{tangential} morphisms of cdg-algebras  about the identity
$e_{B\!,A}:=u_A\circ \ep_B$:
$$
\THOM_{\cdga}(B,A):=
\left\{\ups \in \Hom(B, A)^0\left|
\begin{aligned}
& {d}_{B\!,A}\ups =0
\cr
&\ups\circ u_B =0
\cr
& \ups\circ m_B = m_A\circ \big(e_{B\!,A}\otimes \ups +\ups\otimes e_{B\!,A}\big)
\end{aligned}
 \right.
\right\}.
$$
A \emph{homotopy pair} on $\THOM_{\cdga}(B, A)$
is a pair of families
$$
\big(\ups(t),\s(t)\big) \in \Hom(B, A)^0[t]\oplus \Hom(B, A)^{-1}[t]
,
$$
satisfying
the homotopy flow equation
$\Fr{d}{dt}\ups(t)= d_{B,A}\s(t)$ generated by $\s(t)$ subject to the following conditions:
\eqalign{
\ups(0)\in \THOM_{\cdga}(B, A),
\qquad
\left\{
\begin{aligned}
&\s(t)\circ u_B =0
\cr
& \s(t)\circ m_B = m_A\circ \big(e_{B\!,A}\otimes \s(t) +\s(t)\otimes e_{B\!,A}\big).
\end{aligned}\right.
}
Then, $\ups(t)$ is a family of tangential morphisms of cdg-algebras about the identity.
We say $\ups,\tilde \ups \in \THOM_{\cdga}(B, A)$  are \emph{homotopic}, $\ups\sim \tilde\ups$, 
or having the \emph{same homotopy type}, $[\ups]=[\tilde\ups]$, if there is a homotopy flow connecting them.  
The set of homotopy types of all tangential morphisms of  cdg-algebras about the identity
is denoted by
$\THOM_{\hcdga}(B, A)$.

\begin{theorem}\label{Lieversion}
For every cdg-Hopf algebra $B$
we have a functor
$\bm{T\!\mG}^{B}:{\mathit{ho}\category{cdgA}}(\fieldk)\rightsquigarrow \category{Lie}(\Bbbk)$,
sending

\begin{itemize}
\item
each  cdg-algebra $A$ to the Lie algebra 
$$
\bm{T\!\mG}^{B}(A):=\Big(\THOM_{\hcdga}(B,A), [-,-]_{\ast_{B\!,A}} \Big)
$$
with the Lie bracket defined by, $\forall [\ups_1],[\ups_2] \in \THOM_{\hcdga}(B, {A})$,
$$
\big[[{\ups}_1],[{\ups}_2]\big]_{\ast_{B\!,A}}
:=  \big[m_A\circ ({\ups}_1\otimes {\ups}_2-{\ups}_2\otimes {\ups}_1\big)\circ \cp_B\big]
,
$$
where ${\ups}_1,{\ups}_2 \in \THOM_{\cdga}\big(B,{A}\big)$ are arbitrary representatives
of the homotopy types $[{\ups}_1],[{\ups}_2]$, respectively.

\item 
each  morphism $[f] \in \HOM_{\hcdga}\big(A, A^\pr\big)$ to
the morphism
$\bm{T\mG}^{B}\big([f]\big):\bm{T\mG}^{B}\big(A\big)\rightarrow \bm{T\mG}^{B}\big(A^\pr\big)$ 
of Lie algebras
defined by, $\forall [{\ups}] \in   \THOM_{\hcdga}\big(B,{A}\big)$,
$$
\bm{T\mG}^{B}([f])\big([{\ups}]\big):=\big[f\circ {\ups}\big]
,$$
where  $f$
and ${\ups}$
are arbitrary representatives of the homotopy types $[f]$ and $[{\ups}]$, respectively.
\end{itemize}

For every  $[\p] \in \HOM_{\hcdgh}\big({B},{B}^\pr\big)$ 
we have a natural transformation 
$${T\!\!\sN}_{\![\p]}: \bm{T\!\mG}^{{B^\pr}}\Longrightarrow \bm{T\!\mG}^{{B}}:
{\mathit{ho}\category{cdgA}}(\fieldk)\rightsquigarrow \category{Lie}(\Bbbk),
$$
whose component
${T\!\!\sN}_{\![\p]}^{A}: \bm{T\!\mG}^{B^\pr}(A)\rightarrow  \bm{T\!\mG}^{B}(A)$
at  each cdg-algebra $A$ is defined by,   
for all $[\ups^\pr] \in \THOM_{\hcdga}\big(B^\pr, {A}\big)$,
$$
{T\!\!\sN}_{\![\p]}^{A}\big([\ups^\pr]\big):= \big[\ups^\pr\circ\psi\big],
$$
where 
 $\p \in \HOM_{\cdgh}\big(B, B^\pr\big)$ 
and $\ups^\pr \in \THOM_{\cdga}\big(B^\pr, {A})$ 
are  arbitrary representatives of the homotopy types $[\p]$
and $[\ups^\pr]$, respectively.

The natural transformation
${T\!\!\sN}_{\![\p]}: \bm{T\!\mG}^{B^\pr}\Longrightarrow \bm{T\!\mG}^{B}$ 
is an isomorphism whenever $\p:{B}\rightarrow {B}^\pr$ is a homotopy equivalence of cdg-Hopf algebras.
\end{theorem}

The basic idea of proof is to define a functor
$\bm{T\!\CG}^{B}:{\category{cdgA}}(\fieldk)\rightsquigarrow \category{Lie}(\Bbbk)$
from ${\category{cdgA}}(\fieldk)$  and show that  $\bm{T\!\CG}^{B}$ induces the functor
$\bm{T\!\mG}^{B}:{\mathit{ho}\category{cdgA}}(\fieldk)\rightsquigarrow \category{Lie}(\Bbbk)$
on  ${\mathit{ho}\category{cdgA}}(\fieldk)$.
Our proof shall be a consequence of three lemmas, which will be
stated without proofs.

\begin{lemma}\label{Lieone}
For every cdg-Hopf algebra ${B}$
we have  a functor
$\bm{T\!\CG}^{B}:{\category{cdgA}}(\fieldk)\rightsquigarrow \category{Lie}(\Bbbk)$,
sending
each  cdg-algebra $A$ to the Lie algebra 
$$
\bm{T\!\CG}^{B}(A):=\Big(\THOM_{\cdga}(B, {A}), [-,-]_{\star_{B\!,A}} \Big)
,
$$
where
$[\ups_1, \ups_2]_{\star_{B\!,A}}: =  \ups_1{\star_{B\!,A}}\ups_2 -\ups_2{\star_{B\!,A}}\ups_1$ 
for all $\ups_1, \ups_2 \in \THOM_{\cdga}(B, {A})$,
and each morphism $f:A\rightarrow A^\pr$ of cdg-algebras to
the Lie algebra homomorphism 
$$
\bm{T\!\CG}^{B}(f):\bm{T\!\CG}^{B}(A)\rightarrow \bm{T\!\CG}^{B}(A^\pr)
$$
defined by $\bm{T\!\CG}^{B}(f)(\ups): =f\circ \ups$ for all $\ups \in \THOM_{\cdga}(B, {A})$.
\end{lemma}

\begin{lemma}\label{Lietwo}
For each morphism  $\p:{B}\rightarrow {B}^\pr$ of cdg-Hopf algebras 
we have a natural transformation 
${T\!\!\sN}_{\!\p}: \bm{T\!\CG}^{B^\pr}\Longrightarrow 
\bm{T\!\CG}^B: {\category{cdgA}}(\fieldk)\rightsquigarrow \category{Lie}(\Bbbk)$,
whose component
${T\!\!\sN}_{\!\p}^{A}: \bm{T\!\CG}^{B^\pr}(A)\Longrightarrow 
\bm{T\!\CG}^B(A)$
at  each cdg-algebra $A$ is defined by 
${T\!\!\sN}_{\!\p}^{A}(\ups^\pr):=\ups^\pr\circ \psi$ for all  $\ups^\pr \in \THOM_{\cdga}\big(B^\pr, {A}\big)$.
\end{lemma}

\begin{lemma}\label{Liethree}
Assume that we have following homotopy pairs:
\begin{itemize}

\item $\big({\ups}(t),\s(t)\big)$ , $\big({\ups}_1(t),\s_1(t)\big)$ and  $\big({\ups}_2(t),\s_2(t)\big)$
on $\THOM_{\cdga}(B,{A})$;

\item $\big(f(t),\l(t)\big)$  on $\HOM_{\cdga}(A,A^\pr)$;

\item
$\big(\p(t),\xi(t)\big)$on $\HOM_{\cdgh}({B},{B}^\pr)$;

\item
$\big({\ups}^\pr(t),\s^\pr(t)\big)$  on $\THOM_{\cdga}(B^\pr,{A})$.
\end{itemize}

Then we have following homotopy pairs
\begin{enumerate}[label=({\alph*}),leftmargin=.8cm]

\item
$\Big(\left[\big({\ups}_1(t), {\ups}_2(t)\right]_{\star_{B\!,A}},  \left[\s_1(t), {\ups}_2(t)\right]_{\star_{B\!,A}} 
+ \left[{\ups}_1(t),\s_2(t)\right]_{\star_{B\!,A}} \Big)$
on $\THOM_{\cdga}(B,{A})$.

\item
$\Big(f(t)\circ\ups(t), f(t)\circ \s(t) +\l(t)\circ \ups(t)\Big)$ 
on $\THOM_{\cdga}\big(B,{A}^\pr\big)$.

\item
$\Big({\ups}^\pr(t)\circ \p(t), \ups^\pr(t)\circ \xi(t)+ {\s}^\pr(t)\circ \p(t)\Big)$ 
 on $\HOM_{\cdga}(B,{A})$.
\end{enumerate}
\end{lemma}

Now we are ready for the proof of  Theorem \ref{Lieversion}

\begin{proof}[Theorem \ref{Lieversion}]
From Lemmas \ref{Lieone}, \ref{Liethree}$(a)$ and  \ref{Liethree}$(b)$,
it is trivial to
check that the Lie algebra $\bm{T\mG}^{B}(A)$ is well-defined,
and $\bm{T\!\mG}^{B}([f]): \bm{T\!\mG}^{B}(A)\rightarrow \bm{T\!\mG}^{B}(A^\pr)$ 
is a well-defined Lie algebra homomorphism.
Therefore $\bm{T\!\mG}^{B}$ is a functor from the homotopy category $\mathit{ho}\category{cdgA}(\Bbbk)$ of cdg-algebras
as claimed,  where the functoriality of $\bm{T\!\mG}^{B}$ is obvious.
It is also obvious that $\bm{T\!\mG}^{B}([f])$ is an isomorphism of Lie algebras whenever $f:A\rightarrow A^\pr$ is
a homotopy equivalence of cdg-algebras. 

From Lemmas \ref{Lietwo} and  \ref{Liethree}$(c)$,
it is also trivial to check that the natural transformation 
${T\!\!\sN}_{\![\p]}: \bm{T\!\mG}^{B^\pr}\Longrightarrow \bm{T\!\mG}^{B}:
{\mathit{ho}\category{cdgA}}(\fieldk)\rightsquigarrow \category{Lie}(\Bbbk)$ is well-defined
for every $[{\psi}]\in \HOM_{\hcdgh}\big({B}, {B}^\pr\big)$. Finally it is obvious that
${T\!\!\sN}_{\![{\psi}]}$  is a natural isomorphism whenever $\p:B \rightarrow B^\pr$ is  a homotopy equivalence of cdg-Hopf algebras.
\qed
\end{proof}

Remind that
the category of affine group dg-schemes is anti-equivalent 
to the homotopy category $\mathit{ho}\category{cdgH}(\Bbbk)$
of cdg-Hopf algebras---Theorem \ref{affinegrdgsch}.  Therefore,
the assignments
$\bm{\mG}^{B}\mapsto \bm{T\!\mG}^{B}$
and $\sN_{\![{\psi}]}\mapsto  {T\!\!\sN}_{\![{\psi}]}$
define a contravariant functor $\category{T}$ from the category of affine group dg-schemes
to the category of functors from $\mathit{ho}\category{cdgA}(\Bbbk)$ to $\category{Lie}(\Bbbk)$,
where  ${T\!\!\sN}_{\![{\psi}]}$ is a natural isomorphism
whenever $\sN_{\![{\psi}]}$ is a natural isomorphism.

\subsection{Reduction to affine group schemes}

Let $H(B)$ be the cohomology of a cdg-Hopf algebra $B$. Then, $H(B)$ has a uniquely induced structure 
of cdg-Hopf algebra with the zero-differential, since every structure in $B$ is compatible with the differential $d_B$. 
In particular
the $0$-th cohomology $H^0(B)$ is a commutative Hopf algebra.  
In this subsection, we  compare the affine group dg-scheme
$\bm{\mG}^B:\mathit{ho}\category{cdgA}(\Bbbk)\rightsquigarrow \category{Grp}$ with the affine group scheme
$\bm{G}^{H^0(B)}:\category{cAlg}(\Bbbk)\rightsquigarrow \category{Grp}$ represented by $H^0(B)$, 
where $\category{cAlg}(\Bbbk)$ is the category of commutative algebras over $\Bbbk$.
Note that $\category{cAlg}(\Bbbk)$ is a full subcategory of $\mathit{ho}\category{cdgA}(\Bbbk)$.

The cochain complex $(B, d_B)$
is over a  field $\Bbbk$ so that it splits.  A choice of a splitting, which is an $\chi\in \Hom(B, B)^{-1}$ satisfying $d_B =d_B\circ \chi\circ d_B$,
provides us with a strong deformation retract:
$$
\xymatrix{\big(H(B),0\big) \ar@/^/[r]^{p} & \ar@/^/[l]^{q} \big(B, d_B\big)}
,\qquad
\left\{
\begin{aligned}
{q}\circ {p} &=\I_{H(B)},\cr
{p}\circ {q} &=\I_B - d_B\circ \chi- \chi\circ d_B,
\end{aligned}\right.
$$
where ${p}$ and ${q}$ are homotopy equivalences.  We can always choose $\chi$ such that $\chi\circ u_B =0$.
Then, $H(B)$ has a unique structure of $\Z$-graded commutative Hopf algebra, independent to the choice of splitting,  
with the unit $u_{H(B)}$, the counit $\ep_{H(B)}$,
the product $m_{H(B)}$, the coproduct $\cp_{H(B)}$ and the antipode $\vs_{H(B)}$ defined as follows:
$$
\begin{aligned}
u_{H(B)}&:={q}\circ u_B
,\cr
\ep_{H(B)}&:=\ep_B\circ {p}
,
\end{aligned}
\qquad
\begin{aligned}
m_{H(B)} &:= {q}\circ m_B\circ ({p}\otimes {p})
,\cr
\cp_{H(B)} &:= ({q}\otimes {q})\circ \cp_B\circ {p}
,
\end{aligned}\qquad
\vs_{H(B)} := {q}\circ \vs_B\circ {p}.
$$
On the other hand the cochain maps ${p}:H(B)\rightarrow B$
and ${q}:B\rightarrow H(B)$ are  morphisms of cdg-Hopf algebras
only up to homotopy; it may preserve unit and counit but is an algebra map only up to homotopy, coalgebra map
only up homotopy and commutes with antipodes only up to homotopy.

Let $A$ is a commutative algebra.
Note that $\a \in \Hom(B,A)^0$  is a zero-map on $B^k$ unless $k=0$ since the degree of $A$ is concentrated to zero.
We  compare $\HOM_{\hcdga}(B, A)$ with $\HOM_{\hcdga}(H(B), A)= \HOM_{\category{cAlg}(\Bbbk)}(H^0(B), A)$.
Consider $g \in \HOM_{\cdga}(B, A)$. It is apparent that $g$ induces a morphism of algebra
$H(g) \in \HOM_{\category{cAlg}(\Bbbk)}(H^0(B), A)$ and $H(\tilde{g})=H(g)$, 
whenever $g\sim \tilde g \in \HOM_{\cdga}(B, A)$.
However, we do not know if every element of $\HOM_{\category{cAlg}(\Bbbk)}\big(H^0(B), A\big)$ is obtained in the above manner:
for any $a \in \HOM_{\category{cAlg}(\Bbbk)}\big(H^0(B), A\big)$ we have  $\a = a\circ {q} \in \Hom(B, A)^0$
such that $H(\a)=a$, $d_{B\!,A}\a=0$, $\a\circ u_B = u_A$, but  $\a$ is an algebra map only up to homotopy
 since ${q}_0:B^0\rightarrow H^0(B)$ 
is an algebra map only up to homotopy. 
%
Nevertheless, the following limited statement is valid. 
\begin{lemma}
For every cdg-Hopf algebra $B$ whose  degrees are concentrated to non-negative integers,  the group
$\bm{\mG}^B(A)$   is isomorphic to the group  $\bm{G}^{H^0(B)}(A)$
 for every commutative algebra $A$.
\end{lemma}

\begin{proof} 
Due to the degree reason $\chi$ is the zero map on $B^0$, i.e, $\chi_0=0$. 
Therefore we have ${q}_0\circ {p}_0 =\I_{H^0(B)}$ and ${p}_0\circ {q}_0 =\I_{B^0} - \chi_1\circ d_{B^0}$,
which imply that both ${p}_0:H^0(B)\rightarrow B^0$ and ${q}_0:B^0\rightarrow H^0(B)$ are Hopf algebra maps.
\qed
\end{proof}

\begin{corollary}
For every cdg-Hopf algebra $B$ with the degrees concentrated to non-negative integers,  
the Lie algebra $\category{T}\bm{\mG}^B(A)$  is isomorphic to the Lie algebra $\category{T}\bm{\mG}^{H^0(B)}(A)$ 
 for every commutative algebra $A$.
\end{corollary}

\subsection{Pro-unipotent affine group dg-schemes}

We say an affine  group dg-scheme $\bm{\mG}^B$ \emph{pro-unipotent} if the cdg-Hopf algebra $B$ is conilpotent.
Then, we shall show that the underlying set-valued functors 
 $$
 \grave{\bm{\mG}}^B:\mathit{ho}\category{cdgA}(\fieldk)\rightsquigarrow \category{Set}
 \hbox{ and } \grave{\category{T}\!\bm{\mG}}^B:\mathit{ho}\category{cdgA}(\fieldk)\rightsquigarrow \category{Set}
 $$
are naturally isomorphic. Thus we can recover the group ${\bm{\mG}}^B(A)$ from the Lie algebra ${\category{T}\!\bm{\mG}}^B(A)$
for every cdg-algebra $A$.

Consider a cdg-Hopf algebra $B=(B, u_B, m_B,\ep_B, \cp_B, \vs_B, d_B)$. 
Let $\bar{B} :=\Ker \ep_B$ be the kernel of the counit $\ep_B:B \rightarrow \Bbbk$. 
 Then we have the splitting
$B =\Bbbk\cdot u_B(1)\oplus \bar B$ since  $\ep_B\circ u_B =\I_\Bbbk$.
Define the linear map $\bar\cp_B:{B}\rightarrow {B}\otimes {B}$ such that, $\forall x \in B$,
$$
\bar\cp_B (x) := \cp_B(x) - u_B(1)\otimes x - x\otimes u_B(1).
$$
From the properties of counit $\ep_B$ we have, $\forall x \in B$,
\eqalign{
(\ep_B\otimes \ep_B)\circ\bar\cp_B (x) := \cp_B\big(\ep_B(x)\big) - 1\otimes \ep_B(x) - \ep_B(x)\otimes 1
,\qquad
\ep_B\big(d_B(x)\big) = d_B\big(\ep_B(x)\big).
}
Therefore we have $\bar{\cp}_B\left(\bar{B}\right) \subset \bar{B}\otimes \bar{B}$ 
and $d_B\left(\bar{B}\right) \subset \bar{B}$ so that $\big(\bar{B}, \bar\cp_B, d_B\big)$ is a non-counital dg-coalgebra.

Consider the $n$-fold iterated reduced coproduct
$\bar\cp^{(n)}_B : \bar B \rightarrow \bar B^{\otimes n}$, $n\geq 1$,
generated by $\bar{\cp}_B$, where 
$\bar\cp^{(1)}_B:=\I_B$ and $\bar\cp^{(n+1)}_B 
:= (\bar\cp^{(n)}_B\otimes \I_B)\circ \bar\cp_B$.
Let $\bar B_n \subset \bar B$ be the kernel of $\bar \cp^{(n+1)}_B$. Then, we have the  filtration
$$
0=\bar B_0 \subset \bar B_1 \subset \bar B_2 \subset \bar B_3 \subset \cdots,
$$
satisfying
$\cp_B\big( \bar B_n \big)\subset \sum_{r=0}^n \bar B_r \otimes \bar B_{n-r}$ and $d_B\big( \bar B_n) \subset \bar B_n$.

\begin{definition}
A cdg-Hopf algebra $B$ is conilpotent  if $\bar{B}$ is the union $\cup_{n\geq 0} \bar{B}_n$, i.e.,  for any $x \in \bar{B}$
there is some positive integer $n$ such that $\bar{\cp}_{B}^{(n+1)} (x)=0$.
An affine group dg-scheme $\bm{\mG}^B$ is pro-unipotent if the cdg-Hopf algebra $B$ is conilpotent.
\end{definition}
The notion of conilpotent cdg-Hopf algebras is the straightforward dg-version of conilpotent commutative Hopf algebra \cite{Cartier}.

We introduce some new notations.
Let $\cp^{(n)}_B :  B \rightarrow  B^{\otimes n}$, $n\geq 1$, 
be the $n$-fold iterated  coproduct generated by ${\cp}_B$, where  $\cp^{(1)}_B:=\I_B$ and 
$\cp^{(n+1)}_B := \big(\cp^{(n)}_B\otimes \I_B\big)\circ\cp_B$.
Let $\cp^{(n)}_{B\otimes B} :  B\otimes B \rightarrow  (B\otimes B)^{\otimes n}$, $n\geq 1$, 
be the $n$-fold iterated  coproduct generated by ${\cp}_{B\otimes B} =(\I_B\otimes \t\otimes \I_B)\circ (\cp_B\otimes \cp_B):B\otimes B
\rightarrow (B\otimes B)^{\otimes 2}$, where  $\cp^{(1)}_{B\otimes B}:=\I_{B\otimes B}$ and 
$\cp^{(n+1)}_{B\otimes B} := \big(\cp^{(n)}_{B\otimes B}\otimes \I_{B\otimes B}\big)\circ\cp_{B\otimes B}$.
Then it is trivial to show that, $\forall n\geq 1$,
\eqn{hsqa}{
\cp^{(n)}_{B}\circ m_B
=
\overbrace{( m_B {\otimes} \ldots {\otimes} m_B)}^n\circ \cp^{(n)}_{B\otimes B}.
}

Let $B$ be a conilpotent cdg-Hopf algebra,  $\bm{\mG}^B:\mathit{ho}\category{cdgA}(\Bbbk)\rightsquigarrow \category{Grp}$
be the pro-unipotent affine group dg-scheme represented by $B$ and $\bm{T\!\mG}^B:\mathit{ho}\category{cdgA}(\Bbbk)\rightsquigarrow \category{Lie}(\Bbbk)$
be the associated Lie algebra valued functor.  
We shall show that one can recover $\bm{\mG}^B$ from $\bm{T\!\mG}^B$.

For the precise statement we need some notations.
Let $A=(A, u_A, m_A, d_A)$ be a cdg-algebra and
$m^{(n)}_A:A^{\otimes n}\rightarrow A$, $n\geq 1$, 
be
the $n$-fold iterated product generated by $m_A:A\otimes A \rightarrow A$
such that $m^{(1)}_A:=\I_A$ and 
$m^{(n+1)}_A:= m_A\circ (m^{(n)}_A\otimes \I_A)=  m^{(n)}_A\circ (\I_{A^{\otimes n-1}}\otimes m_A\big)$.
Consider the underlying set-valued functors of the functors ${\bm{\mG}}^B$ and ${\bm{T\!\mG}}^B$:
\eqalign{
\grave{\bm{\mG}}^B=\HOM_{\hcdga}(B, -)&:\mathit{ho}\category{cdgA}(\Bbbk)\rightsquigarrow \category{Set}
,\cr
\grave{\category{T}\!\bm{\mG}}^B=\THOM_{\hcdga}(B, -)
&:\mathit{ho}\category{cdgA}(\Bbbk)\rightsquigarrow \category{Set}.
}

\begin{theorem}\label{conilth}
For every conilpotent cdg-Hopf algebra $B$
we have a natural isomorphism 
$
\xymatrixrowsep{.5pc}
\xymatrixcolsep{3pc}
\xymatrix{ \grave{\bm{T\!\mG}}^B\ar@/^/@{=>}[r]^{\bm{\exp}}  & \ar@/^/@{=>}[l]^{\bm{\ln}} \grave{\bm{\mG}}^B}
$
of functors
whose components
$\xymatrix{ \grave{\bm{T\!\mG}}^B(A)\ar@/^/@{->}[r]^{\bm{\exp}_A}  & \ar@/^/@{->}[l]^{\bm{\ln}_A} \grave{\bm{\mG}}^B(A)}$
for each cdg-algebra $A$ is given by
\eqalign{
\bm{\exp}_{A}\big([\ups]\big) &:= [u_A\circ \ep_B]+\sum_{n=1}^\infty\Fr{1}{n!}\left[m^{(n)}_A\circ \big(\ups\otimes \ldots \otimes\ups\big)\circ \cp^{(n)}_B\right]
,\cr
\bm{\ln}_{A}\big([g]\big) &:= -\sum_{n=1}^\infty\Fr{(-1)^n}{n}\left[m^{(n)}_A\circ \big(\bar g\otimes \ldots \otimes \bar g\big)\circ \cp^{(n)}_B\right],
}
where $g\in  \HOM_{\cdga}(B, A)$ and $\ups\in  \THOM_{\cdga}(B, A)$ are arbitrary representatives of the homotopy types
$[g]$ and $[\ups]$, respectively; and   $\bar g := g -u_A\circ \ep_B$.
\end{theorem}

The above theorem implies that the affine group dg-scheme $\bm{\mG}^B$ can be recovered from the Lie algebra valued functor 
$\mb{T}\!\bm{\mG}^B$ by the Baker-Campbell-Hausdorf formula.
We divide the proof into pieces, and begin with two technical lemmas.

\begin{lemma}\label{gusl} 
For every cdg-algebra $A$ we have
\begin{enumerate}[label=({\alph*})]
\item
$\a_1\star_{B\!,A}\cdots\star_{B\!,A} \a_n 
= m^{(n)}_A\circ(\a_1{\otimes}\ldots{\otimes} \a_n)\circ \cp^{(n)}_B$,
$\forall \a_1,\ldots,\a_n \in \Hom(B,A)$ and $n\geq 1$;
\item 
$m_A\circ m^{(n)}_{A{\otimes}A} = m^{(n)}_{A}\circ \overbrace{(m_A{\otimes} \ldots {\otimes} m_A)}^n$
 for all $n\geq 1$.
\end{enumerate}

\end{lemma}

\begin{proof}
The property
(a)  is trivial for $n=1$. It is the definition of the convolution product $\star_{B\!,A}$ for $n=2$ 
and the rest can be checked easily by an induction using the associativity of $m_A$ and the coassociativity of $\cp_B$.
The property
(b) is trivial for $n=1$, since $m^{(1)}_{A{\otimes} A}:=\I_{A\otimes A}$ and $m^{(1)}_{A}:=\I_{A}$.
For $n=2$, from 
$m^{(2)}_{A\otimes A} := m_{A\otimes A}= (m_A\otimes m_A)\circ (\I_A\otimes \t \otimes \I_A)$ 
and $m^{(2)}_A :=m_A$
it becomes the identity
$$
m_A\circ (m_A\otimes m_A)\circ (\I_A\otimes \t \otimes \I_A)=m_A\circ (m_A\otimes m_A)
,
$$
which is due to the commutativity of $m_A$.
The rest can be easily checked by an induction.
\qed
\end{proof}

\begin{lemma}\label{gusla} 
For every $\b \in \Hom(B,A)^0$ satisfying $\b\circ u_B=0$ and for all $n\geq 1$, we have
$$
m^{(n)}_A\circ(\b{\otimes}\ldots{\otimes} \b)\circ \cp^{(n)}_B = m^{(n)}_A\circ(\b{\otimes}\ldots{\otimes} \b)\circ \bar\cp^{(n)}_B.
$$
\end{lemma}

\begin{proof}
This is an immediate consequence of the splitting
$B =\Bbbk\cdot u_B(1)\oplus \bar B$ and the identity  $m^{(n)}_A\circ(\b{\otimes}\ldots{\otimes} \b)\circ \cp^{(n)}_B\circ u_B=0$ for all $n\geq 1$,
which follows from the properties that $ \cp^{(n)}_B\circ u_B= (u_B\otimes \ldots\otimes u_B)\circ \cp^{(n)}_\Bbbk$ and $\b\circ u_B=0$.
\qed
\end{proof}

We introduce the main propositions for the proof of Theorem \ref{conilth}.

\begin{proposition}\label{quslx}
We have 
an isomorphism 
$\xymatrix{\THOM_{\cdga}(B, A)\ar@/^/[r]^-{\exp_A}&\ar@/^/[l]^-{\ln_A} \HOM_{\cdga}(B, A)}$
for every cdg-algebra $A$, where  $\forall \ups \in \THOM_{\cdga}(B, A)$ and  $\forall g \in \HOM_{\cdga}(B, A)$
\eqalign{
\exp_A(\ups)
:=&u_A\circ \ep_B
+ \sum_{n=1}^\infty\Fr{1}{n!}m^{(n)}_{A}\circ \big({\ups}{\otimes} \ldots {\otimes}{\ups}\big)\circ \cp^{(n)}_B
,\cr
\ln_A(g)
:=&
- \sum_{n=1}^\infty\Fr{(-1)^n}{n}
m^{(n)}_{A}\circ \big(\bar g {\otimes} \ldots {\otimes} \bar g \big)\circ \cp^{(n)}_B
}
such that $\exp_A(\ups)\sim \exp_A(\tilde\ups) \in \HOM_{\cdga}(B, A)$
whenever $\ups\sim \tilde\ups  \in \THOM_{\cdga}(B, A)$, and  
$\ln_A(g)\sim \ln_A(\tilde g) \in \THOM_{\cdga}(B, A)$
whenever $g\sim \tilde g  \in \HOM_{\cdga}(B, A)$.

\end{proposition}

\begin{proof}
We use some shorthand notations. 
We set $e =u_A\circ \ep_B$.
We also set $\star =\star_{B\!,A}$, $\a^{\star 0}=e$ and
$\a^{\star n} = \overbrace{\a\star\ldots\star \a}^n$, $n\geq 1$, for all $\a \in \Hom(B,A)^0$.
Then, by Lemma \ref{gusl}(a),  we have
\eqnalign{sexpln}{
\exp_A(\ups)= e + \sum_{n=1}^\infty\Fr{1}{n!}\ups^{\star n}= \sum_{n=0}^\infty\Fr{1}{n!}\ups^{\star n},
\qquad
\ln_A(g)=- \sum_{n=1}^\infty\Fr{(-1)^n}{n} \bar g^{\star n}.
}
Remind that $e \in  \HOM_{\cdga}(B, A)$, i.e.,
$d_{B\!, A}e=0$,  $e \circ u_B=u_A$ and
$ e \circ m_B =m_A\circ (e{\otimes} e\big)$.

1. We have to justify that the infinite sums in the definition of $\exp_A$ and $\ln_A$ make sense. 
Note the $\ups\circ u_B=0$ by definition. We also have $\bar{g}\circ u_B=0$ 
since   $g\circ u_B =u_A$ and $\bar{g}=g-e$. Then,
by Lemma \ref{gusla} and the conilpotency of $\cp_B$, 
both $\exp_A(\ups)$ and $\ln_A(\a)$ are finite sums.

2.  We check that $\exp_A (\ups)\in  \HOM_{\cdga}(B, A)$ for every $\ups \in \THOM_{\cdga}(B, A)$:
\eqalign{
{d}_{B\!,A}\exp_A (\ups)=0
,\quad
\exp_A(\ups)\circ u_B= u_A
,\quad
\exp_A(\ups)\circ m_B =m_A\circ \big(\exp_A(\ups){\otimes} \exp_A(\ups)\big).
}
The $1$st relation is  trivial since ${d}_{B\!,A}$ is a derivation of $\star$ and ${d}_{B\!,A}e={d}_{B\!,A}\ups=0$.
The $2$nd relation is also  trivial since $u_A \circ \ep_B \circ u_B = u_A$ 
while $\ups^{\star n}\circ u_B= (\ups\circ u_B)^{\star n}=0$ for all $n\geq 1$.

It remains to check the $3$rd relation, which is equivalent to the following relations; $\forall n\geq 0$,
\eqn{checthis}{
\ups^{\star n}\circ m_B= \sum_{k=0}^n \Fr{n!}{(n-k)!k!}  m_A\circ\big(\ups^{\star n-k}{\otimes} \ups^{\star k}\big).
}
For $n=0$ the above becomes $e\circ m_B =m_A\circ (e{\otimes} e\big)$, which is trivial.

For cases with $n \geq 1$, we adopt an additional notation.
Remind that $\big(B{\otimes} B, \ep_{B{\otimes} B}, \cp_{B{\otimes} B}\big)$ 
is a $\Z$-graded coassociative coalgebra
and  $\big(A{\otimes} A,  u_{A{\otimes} A}, m_{A\otimes A}\big)$ 
is a $\Z$-graded super-commutative associative algebra.
Therefore we have  a $\Z$-graded associative algebra 
$\big( \Hom(B\otimes B, A{\otimes}A), e{\otimes} e,  \cstar \big)$,
where $\chi_1{\cstar}\chi_2 := m_{A{\otimes} A}\circ (\chi_1{\otimes} \chi_2)\circ \cp_{B\otimes B}$,  
$\forall \chi_1,\chi_2\in \Hom(B\otimes B, A{\otimes}A)$.
We also have, for all $n\geq 1$ and $\chi_1,\ldots,\chi_n\in \Hom(B\otimes B, A{\otimes}A)$,
\eqn{quslc}{
\chi_1{\cstar}\ldots \cstar\chi_n 
=  m^{(n)}_{\O{\otimes} \O}\circ (\chi_1{\otimes}\ldots{\otimes} \chi_n)\circ \cp^{(n)}_{C\otimes C}.
}
For example, consider $\a_1,\a_2,\b_1,\b_2 \in \Hom(B,A)^0$ 
so that $\a_1{\otimes} \a_2, \b_1{\otimes} \b_2\in \Hom(B\otimes B, A{\otimes}A)^0$.
Then we have 
$(\a_1{\otimes} \b_1){\cstar}(\a_2{\otimes} \b_2) = \a_1\star\a_2 {\otimes} \b_1\star \b_2$.

It follows that 
$(\ups{\otimes} e) {\cstar}(e{\otimes} \ups) =(e{\otimes} \ups){\cstar}(\ups{\otimes} e)$ 
since the both terms are $\ups{\otimes} \ups$.
We also have $(\ups{\otimes} e)^{{\cstar}n} = \ups^{\star n}{\otimes} e$ 
and $(e{\otimes} \ups)^{{\cstar}n} = e{\otimes} \ups^{\star n}$.
Combined with the binomial identity, we obtain that, $\forall n\geq 1$,
$$
(\ups{\otimes} e + e{\otimes} \ups)^{{\cstar}n} 
=\sum_{k=0}^n \Fr{n!}{(n-k)!k!} \big(\ups^{\star n-k}{\otimes} \ups^{\star k}\big)
.
$$
Therefore the RHS of \eq{checthis} becomes
\eqnalign{elrhs}{
\mathit{RHS}:=
&
m_A\circ (\ups{\otimes} e + e{\otimes} \ups)^{{\cstar}n} 
= 
m_A\circ m^{(n)}_{A{\otimes}A}\circ \big(\ups{\otimes} e + e{\otimes} \ups\big)^{{\otimes}  n}
\circ \cp^{(n)}_{B\otimes B}
\cr
=
&m^{(n)}_{A}\circ (m_A{\otimes} \ldots {\otimes} m_A)\circ \big(\ups{\otimes} e + e{\otimes} \ups\big)^{{\otimes}  n}
\circ \cp^{(n)}_{B\otimes B},
}
where we  use Lemma \ref{gusl}(b) for the last equality.
Consider the LHS of \eq{checthis}:
\eqnalign{ellhs}{
\mathit{LHS}:=
&
\ups^{\star n}\circ m_B
= 
m^{(n)}_{A}\circ (\ups{\otimes}\ldots{\otimes} \ups)\circ \cp^{(n)}_{B}\circ m_B
\cr
=
&
m^{(n)}_{A}\circ (\ups\circ m_B {\otimes} \ldots {\otimes} \ups\circ m_B)\circ \cp^{(n)}_{B\otimes B}
\cr
=
&m^{(n)}_{A}\circ (m_A{\otimes} \ldots {\otimes} m_A)\circ \big(\ups{\otimes} e + e{\otimes} \ups\big)^{{\otimes}  n}
\circ \cp^{(n)}_{B\otimes B}
,
}
where we use \eq{hsqa} for the $3$rd equality and
the property $\ups \circ m_B=m_A\circ (\ups{\otimes} e+e{\otimes}\ups)$ for the last equality.
Therefore we have $\exp_A(\ups)\circ m_B =m_A\circ \big(\exp_A(\ups){\otimes} \exp_A(\ups)\big)$.

2.  We check that $\ln_A(g)\in   \THOM_{\cdga}(B, A)$ for every  $g\in\HOM_{\cdga}(B, A)$:
\eqalign{
{d}_{B\!,A}\ln_A (g)=0
,\quad
\ln_A(g)\circ u_B= 0
,\quad
\ln_A (g)\circ m_B=m_A\circ \big(\ln_A(g){\otimes} e+e{\otimes} \ln_A(g)\big).
}
The $1$st relation is obvious
since ${d}_{B\!,A}(\bar{g})={d}_{B\!,A}g -{d}_{B\!,A}e=0$ and ${d}_{B\!,A}$ is a derivation of $\star$. 
The $2$nd relation is also obvious  
since $\bar g^{\star n}\circ u_B=m^{(n)}_A\circ(\bar g\circ u_B)^{\otimes n}\circ \cp^{(n)}_\Bbbk=0$ for all $n\geq 1$.
Therefore it remains to check the $3$rd relation.

Define 
$\ln^{\cstar}_A(\chi) :=-\sum_{n=1}^\infty\Fr{(-1)^n}{n} (\chi -e{\otimes} e)^{{\cstar}n}$ 
for all $\chi \in \Hom(B{\otimes}B, A{\otimes} A)$ satisfying
$ \chi \circ (u_B{\otimes} u_B)=u_A{\otimes} u_A$.  Then, we have
\eqalign{
\ln_A^{\cstar}(g{\otimes} e) =-\sum_{n=1}^\infty\Fr{(-1)^n}{n} ( g{\otimes} e -e{\otimes} e)^{{\cstar}n}
=-\sum_{n=1}^\infty\Fr{(-1)^n}{n} (\bar g{\otimes} e)^{{\cstar}n}=\ln_A (g){\otimes} e,
}
and, similarly,  $e{\otimes} \ln_A(g)=\ln_A^{\cstar}( g{\otimes} e)$.
Therefore, we have 
\eqalign{
\ln_A (g){\otimes} e+e{\otimes} \ln_A(g)
=
& \ln^{\cstar}_A(g{\otimes} e) +\ln^{\cstar}_A(e{\otimes} g)
=  
\ln^{\cstar}_A \big((g {\otimes} e) {\cstar}(e{\otimes} g)\big)
=\ln^{\cstar}_A(g{\otimes} g).
}
On the other hand,
we have
\eqalign{
\bar g^{\star n}\circ m_B
=
&m_A^{(n)}\circ \bar g^{{\otimes} n}\circ \cp^{(n)}_B\circ m_B
\cr
= 
&
m^{(n)}_{A}
\circ  \big(\bar{g}\circ m_B{\otimes} \ldots{\otimes}\bar{g}\circ m_B\big)\circ\cp^{(n)}_{B\otimes B}
\cr
=
&
m^{(n)}_{A}
\circ (m_A\otimes \ldots\otimes m_A)\circ(g{\otimes} g -e{\otimes} e)^{{\otimes} n} 
\circ\cp^{(n)}_{B\otimes B}
\cr
=
& 
m_A\circ m^{(n)}_{A{\otimes}A}
\circ(g{\otimes} g -e{\otimes} e)^{{\otimes} n} 
\circ\cp^{(n)}_{B\otimes B} 
\cr
=
&m_A\circ  (g {\otimes} g -e{\otimes} e)^{{\cstar}n}
,
}
where we have used \eq{hsqa} for the $2$nd equality,
the condition
$\bar g\circ m_B =  g \circ m_B  -  e\circ m_B=m_A\circ (g{\otimes} g -e{\otimes} e)$ 
for the $3$rd equality, Lemma \ref{gusl}(b) for the $4$th equality, and \eq{quslc} for the last equality.
Therefore, we obtain that
\eqalign{
\ln_A (g)\circ m_B = 
&
- \sum_{n=1}^\infty\Fr{(-1)^n}{n} m_A\circ (g {\otimes} g -e{\otimes} e)^{{\cstar}n} 
=
m_A\circ \ln^{\cstar}_A(g{\otimes} g) 
\cr
= 
&
m_A\circ \Big(\ln_A (g){\otimes} e+e{\otimes} \ln_A(g)\Big). 
}

3.  It is obvious now that 
$\ln_A \big(\exp_A(\ups)\big) =\ups$ and $\exp_A \big(\ln_A(g)\big) =g$. Hence $(\exp_A ,\ln_A)$ is
an isomorphism.

4. 
Let $\ups\sim \tilde\ups \in \THOM_{\cdga}(B, A)$. 
Then we have a corresponding  homotopy pair $\big(\ups(t),\s(t)\big)$  on $\THOM_{\cdga}(B, A)$
such that $\ups(0)=\ups$ and $\ups(1)=\tilde\ups$. Let
\eqalign{
g(t) :=& \exp_A \big(\ups(t)\big)
,\cr
\l(t) :=& 
\sum_{n=1}^\infty\sum_{j=1}^n\Fr{1}{n!} \ups(t)^{\star j-1}\star \s(t) \star \ups(t)^{\star n-j}
.
}
Then it is trivial to check that $\big(g(t),\l(t)\big)$ is a homotopy pair on $\HOM_{\cdga}(B, A)$,
so that $\exp_A (\ups)=g(0)\sim g(1)= \exp_A (\tilde\ups) \in \HOM_{\cdga}(B, A)$.

5. 
Let $g\sim \tilde g \in \HOM_{\cdga}(B, A)$ 
and $\big(g(t),\l(t)\big)$  be the corresponding  homotopy pair on $\HOM_{\cdga}(B, A)$
such that $g(0)=g$ and $g(1)=\tilde g$. Let
\eqalign{
v(t) :=& \ln_A \big(g(t)\big)
,\cr
\s(t) :=& - \sum_{n=1}^\infty\sum_{j=1}^n\Fr{(-1)^n}{n} \bar g(t)^{\star j-1}\star \l(t) \star \bar g(t)^{\star n-j}
.
}
Then it is trivial to check that $\big(\ups(t),\s(t)\big)$ is a homotopy pair on $\THOM_{\cdga}(B, A)$,
so that $ \ln_A (g)=\ups(0) \sim \ups(1)= \ln_A (\tilde g) \in \THOM_{\cdga}(B, A)$.
\qed
\end{proof}

\begin{proposition}\label{quslz}
Let $B$ be a conilpotent cdg-Hopf algebra.  Then
we have a natural isomorphism
$
\xymatrixcolsep{3pc}
\xymatrix{ \grave{\bm{T\!\CG}}^{B}\ar@/^/@{=>}[r]^{{\exp}}  
& \ar@/^/@{=>}[l]^{{\ln}} \grave{\bm{\CG}}^{B}}
:{\category{cdgA}}(\Bbbk)\rightsquigarrow \category{Set}
$ 
of functors on ${\category{cdgA}}(\Bbbk)$,
whose  component at each cdg-algebra $A$ is $(\exp_A, \ln_A)$
defined in Proposition \ref{quslx}.
\end{proposition}

\begin{proof}
Remind that $\grave{\bm{T\!\CG}}^B=\THOM_{\cdga}(B,-)$ and $\grave{\bm{\CG}}^B=\HOM_{\cdga}(B,-)$.
In Proposition \ref{quslx},
we have shown that  
$\xymatrix{ \grave{\bm{T\!\CG}}^B(A) \ar@/^/[r]^-{\exp_A}&\ar@/^/[l]^-{\ln_A}  \grave{\bm{\CG}}^B(A)}$ 
is an isomorphism for every cdg-algebra $A$. It remains to check
the naturalness of $\exp$ and $\ln$  
that for every morphism $f:A \rightarrow A^\pr$ of cdg-algebras
the following diagrams are commutative:
\eqalign{
\xymatrixrowsep{1.5pc}
\xymatrixcolsep{3pc}
\xymatrix{
\ar[d]_-{\exp_A}
\grave{\bm{T\!\CG}^{B}}(A)\ar[r]^-{\grave{\bm{T\!\CG}}^{B}(f)} & \grave{\bm{T\!\CG}}^{B}(A^\pr) 
\ar[d]^-{\exp_{A^\pr}}
\cr
\grave{\bm{\CG}}^{B}(A)\ar[r]^-{\grave{\bm{\CG}}^{B}(f)} & \grave{\bm{\CG}}^{B}(A^\pr) 
}
,\qquad
\xymatrix{
\ar[d]_-{\ln_{A}}
\grave{\bm{\CG}}^{B}(A) \ar[r]^-{\grave{\bm{\CG}}^{B}(f)} & \grave{\bm{\CG}}^{B}(A^\pr) 
\ar[d]^-{\ln_{A^\pr}}
\cr
\grave{\bm{T\!\CG}}^{B}(A)\ar[r]^-{\grave{\bm{T\!\CG}}^{B}(f)} & \grave{\bm{T\!\CG}}^{B}(A^\pr) 
}.
}
That   is,  
$\grave{\bm{\CG}}^{B}(f)\circ\exp_A =\exp_{A^\pr}\circ \grave{\bm{T\!\CG}}^{B}(f)$ and
$\grave{\bm{T\!\CG}}^{B}(f)\circ\ln_A =\ln_{A^\pr}\circ \grave{\bm{\CG}}^{B}(f)$.  These are
straightforward since for every $\ups \in \grave{\bm{T\!\CG}}^{B}(A)$ we have
\eqalign{
\grave{\bm{\CG}}^{B}(f)\left(\exp_A (\ups)\right) 
&=f\circ \exp_A(\ups) 
=f\circ u_A\circ \ep_{B} 
+ \sum_{n=1}^\infty\Fr{1}{n!}
f\circ m^{(n)}_{A}\circ \big({\ups}{\otimes} \ldots {\otimes}{\ups}\big)\circ \cp^{(n)}_{B}
\cr
&=u_{A^\pr}\circ \ep_{B} 
+ \sum_{n=1}^\infty\Fr{1}{n!}
m^{(n)}_{A^\pr}\circ \big(f\circ \ups{\otimes} \ldots {\otimes} f\circ \ups\big)\circ \cp^{(n)}_B
=  \exp_{A^\pr}( f\circ\ups) 
\cr
&= \exp_{A^\pr}\left( \grave{\bm{T\!\CG}}^{B}(f)(\ups)\right).
}
The naturalness of $\ln$ can be checked similarly.
\qed
\end{proof}

Now we can finish our proof of Theorem \ref{conilth}.

\begin{proof}[Theorem \ref{conilth}]
We note that
the components 
$\xymatrixcolsep{1.5pc}
\xymatrix{ \grave{\bm{T\mG}}^{B}(A)\ar@/^0.5pc/[r]^{\bm{\exp}_A}  
& \ar@/^0.5pc/[l]^{\bm{\ln}_A} \grave{\bm{\mG}}^{B}(A)}$
of  $\bm{\exp}$ and $\bm{\ln}$ at every cdg-algebra $A$ 
are defined such that $\bm{\exp}_{A}\big([\ups]\big) =\left[ \exp_A(\ups) \right]$
and $\bm{\ln}_{A}\big([g]\big) =\left[ \ln_A(g) \right]$.  Due to Proposition \ref{quslx},
they are well-defined, depending only on the homotopy types of $\ups$ and $g$,
and are
isomorphisms for every cdg-algebra $A$.

It remains to check the naturalness of $\bm{\exp}$ and $\bm{\ln}$.
For every $[f] \in \HOM_{\hcdga}(A, A^\pr)$, we show that the following diagrams commute
\eqalign{
\xymatrixrowsep{1.5pc}
\xymatrixcolsep{3pc}
\xymatrix{
\ar[d]_-{\bm{\exp}_{A}}
\grave{\bm{T\!\mG}^{B}}(A)\ar[r]^-{\grave{\bm{T\!\mG}}^{B}([f])} & \grave{\bm{T\!\mG}}^{B}(A^\pr) 
\ar[d]^-{\bm{\exp}_{A^\pr}}
\cr
\grave{\bm{\mG}}^{B}(A)\ar[r]^-{\grave{\bm{\mG}}^{B}([f])} & \grave{\bm{\mG}}^{B}(A^\pr) 
}
,\qquad
\xymatrix{
\ar[d]_-{\bm{\ln}_{A}}
\grave{\bm{\mG}}^{B}(A) \ar[r]^-{\grave{\bm{\CG}}^{B}([f])} & \grave{\bm{\mG}}^{B}(A^\pr) 
\ar[d]^-{\bm{\ln}_{A^\pr}}
\cr
\grave{\bm{T\!\mG}}^{B}(A)\ar[r]^-{\grave{\bm{T\!\mG}}^{B}([f])} & \grave{\bm{T\!\mG}}^{B}(A^\pr) 
}.
}
We will only check the naturalness of $\bm{\exp}$, since the proof is similar in the case of $\bm{\ln}$.

Let $f \in \HOM_{\cdga}(A, A^\pr)$ be an arbitrary representative of $[f]$. 
Consider any $[\ups] \in \THOM_{\hcdga}(B,A)$ and
let  $\ups\in \THOM_{\cdga}(B,A)$ be  an arbitrary representative of $[\ups]$.
Then it is straightforward to check that the homotopy type 
$[f\circ\ups]$ of $f\circ\ups \in \THOM_{\cdga}(B,A^\pr)$ 
depends only on $[f]$ and $[\ups]$. From Proposition \ref{quslz}, it also follow that 
 the homotopy type $\big[\exp_{A}(f\circ\ups)\big]$ of $\exp_{A}(f\circ \ups) \in \HOM_{\cdga}(B,A^\pr)$
depends only on $[f]$ and $[\ups^\pr]$.  It is also obvious that the homotopy type
$\left[f\circ\exp_A(\ups) \right]$ of  $f\circ \exp_{A}(\ups)\in \HOM_{\cdga}(B,A^\pr)$
depends only on $[f]$ and $[\ups]$. 
Combined with the identity
$f\circ\exp_A(\ups)= \exp_{A^\pr}(f\circ\ups)$ in the proof of  
Proposition \ref{quslz}, we have
\eqalign{
\grave{\bm{\mG}}^{B}([f])\left(\bm{\exp}_A ([\ups])\right) 
&= \left[f\circ\exp_{A}(\ups) \right]
=\left[ \exp_{A^\pr}(f\circ \ups) \right]
=
\bm{\exp}_A\left( \grave{\bm{T\!\mG}}^{B}([f])([\ups])\right).
}
Hence 
$\bm{\exp}:\grave{\bm{T\mG}}^B \Rightarrow \grave{\bm{\mG}}^B
: {\mathit{ho}\category{cdgA}}(\Bbbk)\rightsquigarrow \category{Set}$
is a natural isomorphism. 
\qed
\end{proof}

\section{Representations of an affine group dg-scheme}

Throughout this section we fix a  cdg-Hopf algebra $B=(B, u_B, m_B, \ep_B, \cp_B, \vs_B, d_B)$.
We define a linear representation of 
an affine group dg-scheme
$\bm{\mG}^B:\mathit{ho}\category{cdgA}(\Bbbk)\rightsquigarrow \category{Grp}$
via a linear representation of
the associated functor
$\bm{\CG}^B: \category{cdgA}(\Bbbk)\rightsquigarrow \category{Grp}$
on the category $\category{cdgA}(\Bbbk)$ of cdg-algebras.
Note that $\bm{\CG}^B$ is represented by $B$
and induces $\bm{\mG}^B$ on the homotopy category $\mathit{ho}\category{cdgA}(\fieldk)$.

The linear representations of $\bm{\CG}^B$ form a dg-tensor category
$\dgcat{Rep}(\bm{\CG}^B)$, which is isomorphic 
to the dg-tensor category $\dgcat{dgComod}_R(B)$ of 
right dg-comodules over $B$.  Working with the linear representations of 
$\bm{\CG}^B$ instead of the linear representations of $\bm{\mG}^B$ 
will be a crucial step for a Tannakian reconstruction of  $\bm{\mG}^B$.

\subsection{Preliminary}

In this subsection, we introduce the dg-tensor category formed by free right dg-modules
over a  cdg-algebra $A=(A,u_A, m_A, d_A)$.
We begin with the following basic lemma.

\begin{lemma}\label{basicl} 
For every pair $(M,N)$ of cochain complexes we have an exact sequence of cochain complexes
\eqalign{
\xymatrixcolsep{2pc}
\xymatrixrowsep{0pc}
\xymatrix{
0\ar[r]  & \Hom\big(M, N\otimes A\big)\ar[r]^-{\mp} & \ar@/^1pc/@{..>}[l]^{\mq} 
\Hom\big(M\otimes A, N\otimes A\big)\ar[r]^-{\mr} &   \ar@/^1pc/@{..>}[l]^{\ms} 
\Hom\big(M\otimes A\otimes A, N\otimes A\big)
,
}}
where $\forall \a_{i+1} \in \Hom\big(M\!\otimes\! A^{\otimes i}, N\otimes A\big)$, $i=0,1,2$,
\eqn{coalgext}{
\begin{aligned}
{\mp}(\a_1)&:=(\I_N\otimes m_A)\circ (\a_1\otimes \I_A)
,\cr
{\mq}(\a_2)&:=\a_2\circ (\I_M\otimes u_A)\circ \jmath^{-1}_M
,
\end{aligned}
\quad
\begin{aligned}
\mr(\a_2)&:=
\a_2\circ (\I_M\otimes m_A) - (\I_N\otimes m_{A}) \circ (\a_2\otimes \I_A)
,\cr
{\ms}(\a_3)&:=\a_3\circ (\I_M\otimes u_A\otimes\I_A)\circ \jmath^{-1}_M
.
\end{aligned}
}
such that 
\eqn{coalgextx}{
\mr\circ\mp=0
,\qquad
\mq\circ \mp =\I_{\Hom(M, N\otimes A)}
,\qquad
\mp\circ\mq+\ms\circ\mr =\I_{\Hom(M\otimes A,N\otimes A)}
.
}
\end{lemma}

A \emph{right dg-module over $A$} is a tuple $(M,\g_M)$, where
$M=(M,d_M)$ is a cochain complex and  $\g_M:M\otimes A\to M$ is a cochain map,
called the \emph{action}, satisfying the axioms:
$\g_M\circ(\I_M\otimes m_A)=\g_M\circ(\g_M\otimes \I_A)$ and $\g_M\circ(\I_M\otimes u_A)=\jmath_M$ that the following
diagrams commutes.
\[
\xymatrixcolsep{3.5pc}
\xymatrix{
M\otimes A\otimes A \ar[r]^-{\I_M\otimes m_A} \ar[d]_-{\g_M\otimes \I_A}&
M\otimes A \ar[d]^-{\g_M}\\
M\otimes A \ar[r]^-{\g_M}&
M
}\qquad\qquad
\xymatrix{
M\otimes \Bbbk \ar[r]^-{\jmath_M} \ar[d]_-{\I_M\otimes u_A}&
M\\
M\otimes A\ar[ur]_-{\g_M}
}
\]
For every cochain complex $M$ we have 
a \emph{free} right dg-module  $(M\otimes A, \I_M\otimes m_A)$ over $A$ with 
the free action $\xymatrixcolsep{3pc}\xymatrix{M\otimes A \otimes A\ar[r]^-{\I_M\otimes m_A}& M\otimes A}$.

We 
can form a dg-category   $\dgcat{dgMod}^{\mathit{fr}}_R(A)$ of free right dg-modules over $A$
whose  set of morphisms from $(M\!\otimes\! A,\I_M\!\otimes\! m_A)$ to   $(N\!\otimes\! A,\I_{N}\!\otimes\! m_A)$
is  $\Hom_{m_A\!}\big(M\!\otimes\! A, N\!\otimes\! A\big)$ with the differential $d_{M\!\otimes\! A, N\!\otimes\! A}$,
where 
$\Hom_{m_A\!}\big(M\!\otimes\! A, N\!\otimes\! A\big)$
denotes the set of $\Bbbk$-linear maps $\w:M\otimes A \rightarrow N\otimes A$ making the following diagram commutative:
$$
\xymatrixcolsep{3pc}
\xymatrix{
\ar[d]_-{\w\otimes \I_A}
M\otimes A\otimes A \ar[r]^-{\I_M\otimes m_A}& M\otimes A
\ar[d]^-{\w}
\cr
N\otimes A\otimes A \ar[r]_-{\I_{N}\otimes m_A}& N\otimes A
}
,\quad\hbox{i.e.,}\quad
\w\circ (\I_{M}\otimes m_A)= (\I_{N}\otimes m_A)\circ (\w\otimes \I_A) \Longleftrightarrow \mr(\w)=0.
$$
\begin{corollary} 
We have a bijection 
$\xymatrix{
{\mq}:
\Hom_{m_A}(M\otimes A, N\otimes A) \ar@<3pt>[r]&
\ar@<3pt>[l] \Hom(M, N\otimes A) :{\mp}  
.}$
\end{corollary}

Using Lemma \ref{basicl},  it is straightforward to check that $\dgcat{dgMod}^{\mathit{fr}}_R(A)$ is indeed a dg-category.

\begin{lemma}[Definition]\label{tensorfrmod}
The dg-category  $\dgcat{dgMod}^{\mathit{fr}}_R(A)$
is a dg-tensor category with the following tensor structure.

\begin{enumerate}
\item
The tensor product of free right dg-modules 
$(M\otimes A,\I_M\otimes m_A)$ 
and  $(M^\pr\otimes A,\I_{M^\pr}\otimes m_A)$ 
over $A$ is 
$$
(M\otimes A,\I_M\otimes m_A)\otimes_{m_A}(M^\pr\otimes A,\I_{M^\pr}\otimes m_A):=(M\otimes M^\pr\otimes A,\I_{M\otimes M^\pr}\otimes m_A).
$$
The unit object for the tensor product is the free right dg-module $(\Bbbk\otimes A, \I_\Bbbk\otimes m_A)$ over $A$

\item
If $\w \in \Hom_{m_A}(M\otimes A, N\otimes A)$ and  $\w^\pr \in \Hom_{m_A}(M^\pr\otimes A, N^\pr\otimes A)$,
then we have the tensor product
$\w\otimes_{m_A}\!\w^\pr \in \Hom_{m_A}\big(M\otimes M^\pr\otimes A,  N\otimes N^\pr\otimes A\big)$,
where
$$
\w\otimes_{m_A}\w^\pr
:=(\I_{N\otimes N^\pr}\otimes m_A)\circ(\I_{N}\otimes\t\otimes\I_A)\circ\big({\mq}(\w)\otimes \w^\pr\big).
$$
Equivalently, $\w\otimes_{m_A} \w^\pr$ is determined by the following equality:
$$
{\mq}(\w\otimes_{m_A} \w^\pr)
=(\I_{N\otimes N^\pr}\otimes m_A)\circ(\I_{N}\otimes\t\otimes\I_A)\circ\big({\mq}(\w)\otimes{\mq}(\w^\pr)\big).
$$

\item
$d_{M\otimes M^\pr\otimes A,N\otimes N^\pr\otimes A}
\big(\w\otimes_{m_A}\!\w^\pr\big)
=d_{M\otimes A\!, N\otimes A} \w\otimes_{m_A}\!\w^\pr 
+(-1)^{|\w|} \w\otimes_{m_A}\!d_{M^\pr\otimes A\!, N^\pr\otimes A}\w^\pr$.
\end{enumerate}

\end{lemma}

\begin{proof} Property $1$ is obvious. For property $2$, 
we can check that ${\mr}(\w\otimes_{m_A} \!\w^\pr)=0$ 
whenever ${\mr}(\w)={\mr}(\w^\pr)=0$.
Property $3$ can be checked by a straightforward computation.
\qed
\end{proof}

\begin{lemma}[Definition]
\label{atensoring}
 We have  a tensor dg-functor 
$\otimes A:\dgcat{CoCh}(\Bbbk) \rightsquigarrow \dgcat{dgMod}^{\mathit{fr}}_R(A)$
for  every cdg-algebra $A$, sending
\begin{itemize}
\item each cochain complex $M$ to the free right dg-module $(M\otimes A,\I_M\otimes m_A)$, and
\item each $\Bbbk$-linear map $M\xrightarrow{\p} M^\pr$ to the morphism
$$
\xymatrixcolsep{3pc}
\xymatrix{
(M\otimes A,\I_M\otimes m_A)\ar[r]^-{\p\otimes\I_A} &(M^\pr\otimes A,\I_{M^\pr}\otimes m_A)
}
$$
of right dg-modules over $A$.
\end{itemize}
\end{lemma}
\begin{proof}
Trivial.
\end{proof}

For every cdg-algebra  $A$, we have the following constructions.
\begin{enumerate}
\item
Let
$\End_{m_A\!}(M\otimes A):=\Hom_{m_A\!}(M\otimes A,M\otimes A)$, 
which is the $\Z$-graded vector space of linear maps 
$\w:M\otimes A \rightarrow M\otimes A$ satisfying 
$\w\circ (\I_{M}\otimes m_A)= (\I_{M}\otimes m_A)\circ (\w\otimes \I_A)$.
Then we have a dg-algebra 
\eqn{EMC}{
\bm{\CE}^M(A)= \big(\End_{m_A\!}(M\otimes A), \I_{M\otimes A}, \circ , d_{M\otimes A, M\otimes A}\big).
}

\item
Let $\operatorname{Z^0Aut}_{m_A\!}(M\otimes A)$ be the subset of $\End_{m_A\!}(M\otimes A)$ 
consisting of every degree zero element $\w$ which has a composition inverse 
$\w^{-1}$ and satisfies $d_{M\otimes A,M\otimes A}\w=0$. 
Then we have a group 
\eqn{GLMC}{
\bm{\CG\!\ell}^M(A) :=\big(\operatorname{Z^0Aut}_{m_A\!}(M\otimes A), \I_{M\otimes A}, \circ\big).
}

\item 
Let $\operatorname{H^0Aut}_{m_A\!}(M\otimes A)$ be the set of  cohomology classes of elements in 
$\operatorname{Z^0Aut}_{m_A\!}(M\otimes A)$.
Recall that $\w, \tilde \w \in \operatorname{Z^0Aut}_{m_A\!}(M\otimes A)$
belongs to the same cohomology class 
$\w\sim \tilde \w$, i.e., $[\w]=[\tilde \w]\in \operatorname{H^0Aut}_{m_A\!}(M\otimes A)$, 
if $\tilde \w - \w =d_{M\otimes A, M\otimes A}\l$ for some $\l \in \End_{m_A\!}(M\otimes A)^{-1}$. 
We can check that $\w_1\circ \w_2\sim \tilde \w_1\circ \tilde \w_2\in\operatorname{Z^0Aut}_{m_A\!}(M\otimes A)$
whenever $\w_1\sim \tilde \w_1,  \w_2\sim \tilde \w_2\in\operatorname{Z^0Aut}_{m_A\!}(M\otimes A)$,
and
the $\w^{-1}\sim \tilde \w^{-1}\in\operatorname{Z^0Aut}_{m_A\!}(M\otimes A)$
whenever $\w\sim \tilde \w \in\operatorname{Z^0Aut}_{m_A\!}(M\otimes A)$. 
Let $[\w_1]\diamond [\w_2]:=[\w_1\circ \w_2]$ and $[\w]^{-1}:= [\w^{-1}]$. Then we have a group
\eqn{MGLMC}{
\bm{\mG\!\ml}^M(A) :=\big(\operatorname{H^0Aut}_{m_A\!}(M\otimes A), [\I_{M\otimes A}], \diamond\big).
}
\end{enumerate}

The above  three constructions are functorial as described in the forthcoming
three lemmas. We shall omit the proofs.

\begin{lemma}\label{dullpain}
For every cochain complex $M$
we have a functor
$\bm{\CE}^M:{\category{cdgA}}(\Bbbk) \rightsquigarrow \category{dgA}(\Bbbk)$, sending
\begin{itemize}
\item
each cdg-algebra $A$ to
the dg-algebra $\bm{\CE}^M(A)= \big(\End_{m_A\!}(M\otimes A), \I_{M\otimes A}, \circ , d_{M\otimes A, M\otimes A}\big)$;
\item
each $f \in \HOM_{\cdga}(A, A^\pr)$  to
a morphism  $\bm{\CE}^M(f):\bm{\CE}^M(A)\rightarrow \bm{\CE}^M(A^\pr)$ of dg-algebras
defined by, $\forall \w \in\End_{m_A\!}(M\otimes A)$,
\eqalign{
\bm{\CE}^{M}(f)(\w):= &\mp\Big((\I_M\otimes f)\circ \mq(\w)\Big)
\cr
=
&\big(\I_M\otimes m_{A^\pr}\big)\circ \big(\I_M\otimes f\otimes \I_{A^\pr}\big)\circ \big( \mq(\w)\otimes \I_{A^\pr}\big)
\cr
=
&\big(\I_M\otimes m_{A^\pr}\big)\circ \big(\I_M\otimes f\otimes \I_{A^\pr}\big)\circ \big(\w\circ (\I_M\otimes u_A)\circ \jmath^{-1}_M\otimes \I_{A^\pr}\big).
}
(See 
$\xymatrixcolsep{3pc}
\bm{\CE}^{M}(f)(\w):\xymatrix{
M\otimes A^\pr \ar[r]^-{\mq(\w)\otimes \I_{A^\pr}} & M\otimes A\otimes A^\pr\ar[r]^{\I_M\otimes f\otimes \I_{A^\pr}}& M\otimes A^\pr\otimes A^\pr\ar[r]^-{\I_M\otimes m_{A^\pr}} & M\otimes A^\pr
.}$)
\end{itemize}
That is, we have $\bm{\CE}^M(f)(\w) \in  \End_{m_{A^\pr}\!}(M\otimes A^\pr)$, and
\begin{enumerate}[label=({\alph*}),leftmargin=0.8cm]

\item $\bm{\CE}^M(f)(\I_{M\otimes A})= \I_{M\otimes A}$;

\item $\bm{\CE}^M(f)\big(\w_1\circ \w_2\big)
=\bm{\CE}^M(f)\big(\w_1\big)\circ \bm{\CE}^M(f)\big(\w_2\big)$;

\item $\bm{\CE}^M(f)\circ d_{M\otimes A, M\otimes A} 
= d_{M\otimes A, M\otimes A}\circ \bm{\CE}^M(f)$;

\item $\bm{\CE}^M(f^\pr)\circ \bm{\CE}^M(f)=\bm{\CE}^M(f^\pr\circ f)$
holds for another morphism $f^\pr:A^\pr\to A^{\pr\pr}$ of cdg-algebras.

\end{enumerate}
\end{lemma}

\begin{lemma}\label{dullpaina}
For every cochain complex $M$ we have 
a  functor $\bm{\CG\!\ell}^M:{\category{cdgA}}(\Bbbk) \rightsquigarrow \category{Grp}$,
sending 
\begin{itemize}
\item
each cdg-algebra $A$ to the  group 
$\bm{\CG\!\ell}^M(A)=\big(\operatorname{Z^0Aut}_{m_A\!}(M\otimes A), \I_{M\otimes A}, \circ\big)$;
\item
each $f \in \HOM_{\cdga}(A, A^\pr)$ 
to a homomorphism
$\bm{\CG\!\ell}^M(f):\bm{\CG\!\ell}^M(A)\rightarrow \bm{\CG\!\ell}^M(A^\pr)$ of groups
defined by $\bm{\CG\!\ell}^M(f) := \bm{\CE}^M(f)$, 
\end{itemize}
such that
\begin{enumerate}[label=({\alph*}),leftmargin=0.8cm]
\item
$\bm{\CG\!\ell}^M(\tilde f)(\w) \sim \bm{\CG\!\ell}^M(f)(\w)  
\in \operatorname{Z^0Aut}_{m_{A^\pr}\!}(M\otimes A^\pr)$
for all $\w \in \operatorname{Z^0Aut}_{m_{A}\!}(M\otimes A)$ 
whenever $f \sim \tilde f\in \HOM_{\cdga}(A, A^\pr)$, and

\item $\bm{\CG\!\ell}^M( f)(\tilde\w) \sim \bm{\CG\!\ell}^M( f)(\w)
\in \operatorname{Z^0Aut}_{m_{A^\pr}\!}(M\otimes A^\pr)$
for all $f \in \HOM_{\cdga}(A, A^\pr)$ 
whenever $\w \sim \tilde\w \in \operatorname{Z^0Aut}_{m_{A}\!}(M\otimes A)$.
\end{enumerate}

\end{lemma}

\begin{lemma}\label{dullpainax}
For every cochain complex $M$ we have a  functor
$\bm{\mG\!\ml}^{M}:{\mathit{ho}\category{cdgA}}(\Bbbk) \rightsquigarrow \category{Grp}$,
sending
\begin{itemize}
\item
each cdg-algebra $A$ to the  group
$\bm{\mG\!\ml}^{M}(A)=\big(\operatorname{H^0Aut}_{m_A\!}(M\otimes A), [\I_{M\otimes A}], \diamond\big)$;

\item
each $[f]\in \HOM_{\mathit{ho}\cdga}(A, A^\pr)$ to
a homomorphism
$\bm{\mG\!\ml}^{M}([f]) :\bm{\mG\!\ml}^{M}(A)\rightarrow \bm{\mG\!\ml}^{M}(A^\pr)$ of groups
defined by, for all 
$[\w]\in \operatorname{H^0Aut}_{m_{A}\!}(M\otimes A)$,
$$
\bm{\mG\!\ml}^{M}([f])([\w])
:=\big[\bm{\CG\!\ell}^M(f)(\w)\big],
$$
where $f \in  \HOM_{\cdga}(A, A^\pr)$ and $\w\in \operatorname{Z^0Aut}_{m_{A}\!}(M\otimes A)$ 
are arbitrary representatives of $[f]$
and $[\w]$, respectively.
\end{itemize}
\end{lemma}

As variants of the above three lemmas, we also have the following associated corollaries.

Remind that a dg-Lie algebra is both a cochain complex and a $\Z$-graded Lie algebra,
whose differential is a derivation with respect to the bracket.   For any dg-algebra $(A, u_A, m_A, d_A)$ we have 
a dg-Lie algebra $(A, [-,-]_{m_A}, d_A)$ where $[x, y]_{m_A} := m_A(x\otimes y) -(-1)^{|x||y|}m_A(y\otimes x)$
for all $x,y \in A$.   A morphism of dg-Lie algebras is simultaneously a cochain map and a Lie algebra map.

\begin{corollary}\label{ldullpain}
For every cochain complex $M$
we have a functor
$\bm{\CL}^M:{\category{cdgA}}(\Bbbk) \rightsquigarrow \category{dgL}(\Bbbk)$, sending
\begin{itemize}
\item
each cdg-algebra $A$ to
the dg-Lie algebra $\bm{\CL}^M(A):= \big(\End_{m_A\!}(M\otimes A), [-,-]_\circ , d_{M\otimes A, M\otimes A}\big)$;
\item
each $f \in \HOM_{\cdga}(A, A^\pr)$  to
a morphism  $\bm{\CL}^M(f):\bm{\CL}^M(A)\rightarrow \bm{\CL}^M(A^\pr)$ of dg-Lie algebra
defined by $\bm{\CL}^{M}(f):=\bm{\CE}^{M}(f)$.
\end{itemize}

\end{corollary}

\begin{corollary}\label{ldullpaina}
For every cochain complex $M$ we have 
a  functor $\bm{g\!\ell}^M:{\category{cdgA}}(\Bbbk) \rightsquigarrow  \category{Lie}(\Bbbk)$,
sending 
\begin{itemize}
\item
each cdg-algebra $A$ to the  Lie algebra 
$\bm{g\!\ell}^M(A):=\big(\operatorname{Z^0End}_{m_A\!}(M\otimes A), [-,-]_\circ\big)$;
\item
each $f \in \HOM_{\cdga}(A, A^\pr)$ 
to a morphism
$\bm{g\!\ell}^M(f):\bm{g\!\ell}^M(A)\rightarrow \bm{g\!\ell}^M(A^\pr)$ of Lie algebras
defined by $\bm{g\!\ell}^M(f) := \bm{\CL}^M(f)$, 
\end{itemize}
such that
\begin{enumerate}[label=({\alph*}),leftmargin=0.8cm]
\item
$\bm{g\!\ell}^M(\tilde f)(\w) \sim\bm{g\!\ell}^M(f)(\w)  
\in \operatorname{Z^0End}_{m_{A^\pr}\!}(M\otimes A^\pr)$
for all $\w \in \operatorname{Z^0End}_{m_{A}\!}(M\otimes A)$ 
whenever $f \sim \tilde f\in \HOM_{\cdga}(A, A^\pr)$, and

\item $\bm{g\!\ell}^M( f)(\tilde\w) \sim \bm{g\!\ell}^M( f)(\w)
\in \operatorname{Z^0End}_{m_{A^\pr}\!}(M\otimes A^\pr)$
for all $f \in \HOM_{\cdga}(A, A^\pr)$ 
whenever $\w \sim \tilde\w \in \operatorname{Z^0End}_{m_{A}\!}(M\otimes A)$.
\end{enumerate}

\end{corollary}

\begin{corollary}\label{ldullpainax}
For every cochain complex $M$ we have a  functor
$\bm{\mg\!\ml}^{M}:{\mathit{ho}\category{cdgA}}(\Bbbk) \rightsquigarrow \category{Lie}(\Bbbk)$,
sending
\begin{itemize}
\item
each cdg-algebra $A$ to the  group $\bm{\mg\!\ml}^{M}(A)=\big(\operatorname{H^0End}_{m_A\!}(M\otimes A), [-,-]_\diamond\big)$;

\item
each $[f]\in \HOM_{\mathit{ho}\cdga}(A, A^\pr)$ to
a morphism
$\bm{\mg\!\ml}^{M}([f]) :\bm{\mg\!\ml}^{M}(A)\rightarrow \bm{\mg\!\ml}^{M}(A^\pr)$ of Lie algebras
defined by
$\bm{\mg\!\ml}^{M}([f])([\w]):=\big[\bm{g\!\ell}^M(f)(\w)\big]$,
for all 
$[\w]\in \operatorname{H^0End}_{m_{A}\!}(M\otimes A)$,
where $f \in  \HOM_{\cdga}(A, A^\pr)$ and $\w\in \operatorname{Z^0End}_{m_{A}\!}(M\otimes A)$ 
are arbitrary representatives of $[f]$
and $[\w]$, respectively.
\end{itemize}
\end{corollary}

\subsection{Linear representations of affine group dg-scheme}


We are ready to define  a linear representation of 
the functor $\bm{\CG}^{B}:\category{cdgA}(\Bbbk)\rightsquigarrow \category{Grp}$ represented by $B$.

\begin{definition}
A linear presentation of the functor
$\bm{\CG}^{B}:\category{cdgA}(\Bbbk)\rightsquigarrow \category{Grp}$
is a pair $\big(M,\bm{\r}^M\big)$, where $M$ is a cochain complex and  
${\bm{\r}^M: \bm{\CG}^{B}\Rightarrow\bm{\CG\!\ell}^{M}
:\category{cdgA}(\Bbbk)\rightsquigarrow \category{Grp}}$
is a natural transformation of the functors.
\end{definition}

\begin{remark}\label{repyoneda}
Since a linear representation $\bm{\r}^M: \bm{\CG}^{B}\Rightarrow\bm{\CG\!\ell}^{M}$ 
is a natural transformation of covariant functors,
\begin{itemize}
\item  the component $\bm{\r}^M_{\!A}:\bm{\CG}^{B}(A)\rightarrow \bm{\CG\!\ell}^{M}(A)$ of $\bm{\r}^M$ at each cdg-algebra $A$
is a homomorphism of groups, 
and  
\item
the following diagram commutes for every morphism $f:A\rightarrow A^\pr$ of cdg-algebras;
\eqnalign{repnatural}{
\xymatrixrowsep{1.5pc}
\xymatrixcolsep{4pc}
\xymatrix{
\ar[d]_-{\bm{\r}^M_{\!A}}
\bm{\CG}^{B}(A)\ar[r]^-{\bm{\CG}^{B}(f)}&\bm{\CG}^{B}(A^\pr)
\ar[d]^-{\bm{\r}^M_{\!A^\pr}}
\cr
\bm{\CG\!\ell}^{M}(A)\ar[r]^-{\bm{\CG\!\ell}^{M}(f)}&\bm{\CG\!\ell}^{M}(A^\pr)
},\quad{i.e.,}\quad
\bm{\r}^M_{\!A^\pr}\circ\bm{\CG}^{B}(f) =\bm{\CG\!\ell}^{M}(f)\circ \bm{\r}^M_{\!A}
.
}
\end{itemize}
Since $\bm{\CG}^{B}$ is representable, the Yoneda lemma implies that a natural transformation
$\bm{\r}^M:\bm{\CG}^{B}\Rightarrow \bm{\CG\!\ell}^{M}$
is completely determined  by the universal element 
$$
\bm{\r}^M_B(\I_{B}):M\otimes B\to M\otimes B.
$$
Indeed,  the naturalness  \eq{repnatural} of $\bm{\r}^M$ imposes that for every  morphism ${B\xrightarrow{g}  A}$ of cdg-algebras
we have
$\bm{\r}^M_{\!A}\big(g\big)=\bm{\r}^M_{\!A}\left(\bm{\CG}^{B}(g)(\I_{B})\right)
=\bm{\CG\!\ell}^{M}(g)\left( \bm{\r}^M_{B}(\I_{B})\right)$.
Explicitly, we obtain that
\eqnalign{univnat}{
\bm{\r}^M_{\!A}\big(g\big)
&={\mp}\Big((\I_M\otimes g)\circ {\mq}\big(\bm{\r}^M_{B}(\I_{B})\big)  \Big)
\quad\Longleftrightarrow \quad
{\mq}\Big(\bm{\r}^M_{\!A}\big(g\big)\Big) =(\I_M\otimes g)\circ {\mq}\big(\bm{\r}^M_{B}(\I_{B})\big),
}
where ${\mp}$ and ${\mq}$ are defined in Lemma \ref{basicl}.
\end{remark}

\begin{lemma}
A linear representation 
$\bm{\r}^{M}: \bm{\CG}^{B}\Rightarrow\bm{\CG\!\ell}^{M}$
of $\bm{\CG}^{B}$ induces a natural transformation
$[\bm{\r}]^{M}: \bm{\mG}^{B}\Rightarrow\bm{\mG\!\ell}^{M}
:{\mathit{ho}\category{cdgA}}(\Bbbk)\rightsquigarrow \category{Grp}$,
whose component $[\bm{\r}]^{M}_{\!A}$ at each cdg-algebra
$A$ is the  homomorphism  $[\bm{\r}]^{M}_{\!A}:\bm{\mG}^{B}(A)\rightarrow \bm{\mG\!\ell}^{M}(A)$ of groups 
defined by 
$[\bm{\r}]^{M}_{\!A}([g]):=\left[\bm{\r}^{M}_{\!A}(g)\right] \in \operatorname{H_0Aut}_{m_A\!}(M\otimes A)$ 
for all  $[g]\in \HOM_{\hcdga}(B,A)$,
where $g\in \HOM_{\cdga}(B,A)$ is an arbitrary representative of $[g]$.
\end{lemma}

\begin{proof}
Let  $g \sim \tilde g \in \HOM_{\cdga}(B,A)$.
Then 
we have a homotopy pair $\big(g(t),\chi(t)\big)$ on $\HOM_{\cdga}(B,A)$ 
where $g(t) = g + d_{B\!, A} \int^t_0 \chi(s) d\!s$  is a family of morphism of cdg-algebras 
such that  $g(0) = g$ and $g(1)=\tilde g$.
From \eq{univnat} it is straightforward to check that 
$\bm{\r}^{M}_{\!A}(\tilde g) \sim \bm{\r}^{M}_{\!A}(g)  \in \operatorname{Z_0Aut}_{m_A\!}(M\otimes A)$,
since both ${\mp}$ and ${\mq}$ are cochain maps 
and $\bm{\r}^{M}_{B}(\I_B) \in \operatorname{Z_0Aut}_{m_B\!}(M\otimes B)$.
Therefore we have 
$\big[\bm{\r}^{M}_{\!A}(\tilde g)\big] 
=\big[\bm{\r}^{M}_{\!A}(g)\big]  \in \operatorname{H_0Aut}_{m_A\!}(M\otimes A)$,
so that $[\bm{\r}]^{M}_{\!A}:\bm{\mG}^{B}(A)\rightarrow \bm{\mG\!\ell}^{M}(A)$  
is well-defined homomorphism of groups for every $A$.
The naturalness  of $[\bm{\r}]^{M}_{\!A}$, i.e.,
for every $[f] \in \HOM_{\hcdga}(A, A^\pr)$ we have 
$[\bm{\r}]^M_{\!A^\pr}\circ\bm{\mG}^{B}([f]) =\bm{\mG\!\ell}^{M}([f])\circ [\bm{\r}]^M_A$
follows from the naturalness  \eq{repnatural} of $\bm{\r}^{M}_{\!A}$ and the definitions of 
$[\bm{\r}]^{M}_{\!A}$,  $\bm{\mG}^{B}([f])$ and $\bm{\mG\!\ell}^{M}([f])$.
\qed
\end{proof}

\begin{definition}
A linear representation of an affine group dg-scheme $\bm{\mG}^{B}$ 
is a pair $\big(M, [\bm{\r}]^{M}\big)$ of cochain complex $M$ and 
a natural transformation 
${[\bm{\r}]^{M}: \bm{\mG}^{B}\Rightarrow\bm{\mG\!\ell}^{M}}$,
which is
\emph{induced} from a linear representation 
$\bm{\r}^{M}: \bm{\CG}^{B}\Rightarrow\bm{\CG\!\ell}^{M}$
of  $\bm{\CG}^{B}$. 
\end{definition}

\begin{remark}
Despite of the above definition we will  work with  linear representations of 
$\bm{\CG}^{B}$ rather than those of $\bm{\mG}^{B}$.
 Working with the dg-tensor category of linear representations of 
$\bm{\CG}^{B}$  will be a crucial step for  our Tannakian reconstructions 
of both $\bm{\CG}^{B}$ and $\bm{\mG}^{B}$
in the next section.
The linear representations   of $\bm{\CG}^{B}$ shall form a dg-tensor category
$\dgcat{Rep}(\bm{\CG}^{B})$.
We regard $\dgcat{Rep}(\bm{\CG}^{B})$ as the  dg-tensor category
of "linear representations of $\bm{\mG}^{B}$".
\end{remark}

Here are two basic examples of linear representations of $\bm{\CG}^{B}$.

\begin{example}[The trivial representation] \label{trivalrepsh}
The ground field $\Bbbk$ as a cochain complex $\Bbbk=(\Bbbk, 0)$ with zero differential
defines the trivial representation  $\big(\Bbbk, \bm{\r}^{\Bbbk}\big)$, where the component  
$\bm{\r}^{\Bbbk}_{\!A}$ 
of $\bm{\r}^{\Bbbk}:\bm{\CG}^{B}\Rightarrow \bm{\CG\!\ell}^{\Bbbk}$
at every cdg-algebra $A$ is the 
trivial homomorphism $\bm{\r}^{\Bbbk}_{\!A}:\bm{\CG}^{B}(A)\rightarrow\bm{\CG\!\ell}^{\Bbbk}(A)$ of groups:
$\forall g \in \HOM_{\cdga}(B, A)$,
\eqn{trivialrep}{
\bm{\r}^{\Bbbk}_{\!A}(g) := \I_\Bbbk \otimes \I_A: \Bbbk\otimes A\rightarrow \Bbbk\otimes A
.
}
\end{example}

\begin{example}[The regular representation] \label{regularrepsh}
Associated to the cdg-Hopf algebra $B$ as a cochain complex
we have the regular representation  $\big(B, \bm{\r}^{B}\big)$, where the component  
$\bm{\r}^{B}_{\!A}$ 
of $\bm{\r}^{B}:\bm{\CG}^{B}\Rightarrow \bm{\CG\!\ell}^{B}$
at each cdg-algebra $A$ is the 
homomorphism $\bm{\r}^{B}_{\!A}:\bm{\CG}^{B}(A)\rightarrow\bm{\CG\!\ell}^{A}$ of groups
defined by, $\forall g \in \HOM_{\cdga}(B,A)$,
\eqnalign{regularrep}{
\bm{\r}^B_{\!A} (g)
=
&
(\I_B\otimes m_A)\circ (\I_B\otimes g\otimes \I_A)\circ (\cp_B\otimes \I_A)
={\mp}\big((g\otimes\I_A)\circ \cp_B\big)
\cr
&
\xymatrixcolsep{3pc}
\xymatrix{
B\otimes A \ar[r]^-{\cp_B\otimes \I_A} & B\otimes B \otimes A \ar[r]^-{\I_B\otimes g\otimes \I_A} 
& B\otimes A\otimes A \ar[r]^-{\I_B\otimes m_A}& B\otimes A
}
}
We can check that $\big(B, \bm{\r}^{B}\big)$ is a linear representation as follows
\begin{itemize}
\item 
We have $d_{M\otimes A,M\otimes A}\bm{\r}^B_A (g)
= {\mp}\big((d_{B\!,A}g\otimes\I_A)\circ \cp_B\big)=0$
for  all $g \in \HOM_{\cdga}(B,A)$,
since both ${\mp}$ and $\cp_B$ are cochain maps;

\item
We have $\bm{\r}^B_A (u_A\circ \ep_B)
=  (\I_B\otimes m_A)\circ \big(\I_B\otimes( u_A\circ \ep_B) \otimes \I_A\big)\circ (\cp_B\otimes \I_A)
=\I_{B\otimes A}
.
$

\item
For all $g_1,g_2 \in \HOM_\cdga(B,A)$:
\eqalign{
\bm{\r}^B_A \big(g_1 &\star_{B\!,A} g_2\big)
:=
(\I_B\otimes m_A)\circ \Big(\I_B\otimes \big(m_A\circ (g_1\otimes g_2)
\circ \cp_B\big)\otimes \I_A\Big)\circ (\cp_B\otimes \I_A)
\cr
=
&
(\I_B\otimes m_A)\circ (\I_B\otimes g_1\otimes \I_A)\circ (\cp_B\otimes \I_A)
\circ (\I_B\otimes m_A)\circ (\I_B\otimes g_2\otimes \I_A)\circ (\cp_B\otimes \I_A)
\cr
=
&\bm{\r}^B_A (g_1) \circ \bm{\r}^B_A (g_2) 
}
The 2nd equality is due to coassociativity of $\cp_B$ and the associativity of $m_A$.
\qed
\end{itemize}

\end{example}

\begin{definition}[Lemma]
\label{repofpshgdcat}
The dg-tensor category 
$\Big(\dgcat{{Rep}}\big(\bm{\CG}^B\big), \bm{\otimes},  \big(\Bbbk, \bm{\r}^\Bbbk\big)\Big)$
of linear representations
of $\bm{\CG}^B$ 
is
defined as follows.

\begin{enumerate}[label=({\alph*})]
\item
An object is a linear representation $\big(M,\bm{\r}^{M}\big)$ of 
$\bm{\CG}^B$.

\item
A morphism $\p: \big(M,\bm{\r}^{M}\big)\rightarrow \big(M^\pr,\bm{\r}^{M^\pr}\big)$ 
of linear representations  of
$\bm{\CG}^B$ is a linear map $\p:M \rightarrow M^\pr$
making the following diagram commutative
for every cdg-algebra $A$
and every $g \in \bm{\CG}^B(A)$  
$$
\xymatrixcolsep{3pc}
\xymatrix{
 \ar[d]_-{\p\otimes \I_A}
M\otimes A \ar[r]^-{\bm{\r}^{M}_{\!A}(g)} & M\otimes A
 \ar[d]^-{\p\otimes \I_A}
\cr
M^\pr\otimes A \ar[r]^-{\bm{\r}^{M^\pr}_{A}(g)} & M^\pr\otimes A
,}
\quad\hbox{i.e.,}\quad
(\p\otimes \I_A)\circ\bm{\r}^{M}_{\!A}(g)
=\bm{\r}^{M^\pr}_{\!A}(g)\circ (\p\otimes \I_A)
.
$$

\item 
The differential of a morphism 
$\p: \big(M,\bm{\r}^M\big)\rightarrow \big(M^\pr,\bm{\r}^{M^\pr}\big)$ of linear representations 
is the morphism  
$d_{M\!,M^\pr}\p: \big(M,\bm{\r}^M\big)\rightarrow \big(M^\pr,\bm{\r}^{M^\pr}\big)$ of linear representations.

\item The tensor product $\big(M,\bm{\r}^{M}\big)\bm{\otimes} \big(M^\pr,\bm{\r}^{M^\pr}\big)$ of two objects
is the linear representation $(M\otimes M^\pr, \bm{\r}^{M\otimes M^\pr})$, where
$M\otimes M^\pr=(M\otimes M^\pr,  d_{M\otimes M^\pr})$ is the tensor product of cochain complexes
and  $ \bm{\r}^{M\otimes M^\pr} : \bm{\CG}^B\Rightarrow\bm{\CG\!\ell}^{M\otimes M^\pr}$ is the natural transformation
whose component $\bm{\r}^{M\otimes M^\pr}_A: \bm{\CG}^B(A)\rightarrow\bm{\CG\!\ell}^{M\otimes M^\pr}(A)$ at each cdg-algebra $A$ 
is the group homomorphism defined by, $\forall g \in \HOM_{\cdga}(B,A)$,
\eqalign{
\bm{\r}^{M\otimes M^\pr}_A(g)
:=& \bm{\r}^{M}_{\!A}(g)\otimes_{m_A}\!\bm{\r}^{M^\pr}_A(g)
}
The unit object for the tensor product is the trivial representation 
$(\Bbbk, \bm{\r}^\Bbbk)$ in Example \ref{trivalrepsh}.
\end{enumerate}
\end{definition}
\begin{proof}
We check that $\dgcat{{Rep}}(\bm{\CG}^B)$ is a dg-category as follows.
Let $\psi: \big(M, \bm{\r}^{M}\big) \rightarrow \big(M^\pr, \bm{\r}^{M^\pr}\big)$ be  a morphism of linear representations.
Then,  for every cdg-algebra $A$ we have 
\eqn{borea}{
(\p\otimes \I_A)\circ\bm{\r}^{M}_{\!A}(g)
= \bm{\r}^{M^\pr}_{\!A}(g)\circ (\p\otimes \I_A)
,\qquad
\begin{cases}
d_{M\otimes A}\circ \bm{\r}^M_A(g) =  \bm{\r}_M^A(g) \circ d_{M\otimes A}
,\cr
d_{M^\pr\otimes A}\circ \bm{\r}^{M^\pr}_A(g) =  \bm{\r}^{M^\pr}_A(g) \circ d_{M^\pr\otimes A}
,
\end{cases}
}
where  the $2$nd set of relations is due to the conditions 
that
$\bm{\r}^{M}_{\!A}(g) \in \operatorname{Z^0Aut}_{m_A}(M\otimes A)$ and 
$\bm{\r}^{M^\pr}_{\!A}(g) \in \operatorname{Z^0Aut}_{m_A}(M^\pr\otimes A)$.
Combining the relations in \eq{borea} we can deduce that
$$
(d_{M\!,\, M^\pr}\p\otimes \I_A)\circ\bm{\r}^{M}_{\!A}(g)
= \bm{\r}^{M^\pr}_{\!A}(g)\circ (d_{M\!,\, M^\pr}\p\otimes \I_A)
.
$$
Therefore $d_{M\!,\, M^\pr}\p:\big(M, \bm{\r}^{M}\big) \rightarrow \big(M^\pr, \bm{\r}^{\!M^\pr}\big)$
is also a morphism of linear representations. It is obvious that $d_{M\!,\, M^\pr}\circ d_{M\!,\, M^\pr}=0$.
Hence $\HOM_{\dgcat{{Rep}}(\bm{\CG}^B)}\big((M, \bm{\r}^{M}), (M^\pr, \bm{\r}^{\!M^\pr})\big)$
is a cochain complex with the differential $d_{M\!,\, M^\pr}$.
For another morphism
$\p^\pr: \big(M^\pr, \bm{\r}^{M^\pr}\big)\rightarrow \big(M^\ppr, \bm{\r}^{M^\ppr}\big)$ 
of linear representations, the composition
$\p^\pr\circ \p: \big(M, \bm{\r}^M\big)  \rightarrow \big(M^\ppr,\bm{\r}^{M^\ppr}\big)$ is
also a morphism of linear representations and we have
$d_{M\!,\,M^\ppr}\big(\p^\pr\circ \p\big)
=d_{M^\pr\!,\,M^\ppr}\p^\pr\circ \p+(-1)^{|\p^\pr|}\p^\pr\circ d_{M\!,\,M^\pr} \p$.
Moreover, one can easily check that the tensor product
and the unit object in $(d)$ endow 
$\dgcat{{Rep}}\big(\bm{\CG}^B\big)$ with a structure of dg-tensor category. 
\qed
\end{proof}

\begin{definition}
A linear presentation of the functor
$\bm{T\!\CG}^{B}:\category{cdgA}(\Bbbk)\rightsquigarrow \category{Lie}(\Bbbk)$
is a pair $\big(M,\bm{\e}^M\big)$, where $M$ is a cochain complex and  
$\bm{\e}^M: \bm{T\!\CG}^{B}\Rightarrow\bm{g\!\ell}^{M}
:\category{cdgA}(\Bbbk)\rightsquigarrow \category{Lie}(\Bbbk)$
is a natural transformation of the functors.
\end{definition}

We can check that
a linear representation 
$\bm{\ep}^{M}: \bm{T\!\CG}^{B}\Rightarrow\bm{g\!\ell}^{M}$
of $\bm{T\!\CG}^{B}$ induces a natural transformation
$[\bm{\e}]^{M}: \bm{T\!\mG}^{B}\Rightarrow\bm{\mg\!\ml}^{M}
:{\mathit{ho}\category{cdgA}}(\Bbbk)\rightsquigarrow \category{Lie}(\Bbbk)$,
whose component $[\bm{\e}]^{M}_{\!A}$ at each cdg-algebra
$A$ is the  homomorphism  $[\bm{\e}]^{M}_{\!A}:\bm{T\!\mG}^{B}(A)\rightarrow \bm{\mg\!\ml}^{M}(A)$ of groups 
defined by 
$[\bm{\e}]^{M}_{\!A}([\ups])=\left[\bm{\e}^{M}_{\!A}(\ups)\right] \in \operatorname{H_0End}_{m_A\!}(M\otimes A)$ 
for all  $[\ups]\in \THOM_{\hcdga}(B,A)$,
where $\ups\in \THOM_{\cdga}(B,A)$ is an arbitrary representative of $[\ups]$.

\begin{definition}
A linear presentation of the functor
$\bm{T\!\mG}^{B}$
is a pair $\big(M,[\bm{\e}]^M\big)$, where $M$ is a cochain complex and  
$[\bm{\e}]^M: \bm{T\!\mG}^{B}\Rightarrow\bm{\mg\!\ml}^{M}$ is
a natural transformation induced from a linear representation 
$\bm{\e}^M: \bm{T\!\CG}^{B}\Rightarrow\bm{g\!\ell}^{M}$ of $\bm{T\!\CG}^{B}$.
\end{definition}

\begin{remark}
We can also form a dg-tensor category of linear representations of 
$\bm{T\!\CG}^{B}$.  It is not difficult to show that the dg-tensor categories
formed by linear presentations of $\bm{T\!\CG}^{B}$ and linear representations of 
$\bm{\CG}^{B}$ are isomorphic  if $\bm{\CG}^B$ is pro-unipotent, or equivalently, $B$ is conilpotent.
\end{remark}

\subsection{The dg-tensor category of dg-comodules over a cdg-Hopf algebra}

A \emph{right dg-comodule} over a cdg-Hopf algebra $B$
is a pair $\big(M,\g^M\big)$ of a cochain complex $M=\big(M, d_M\big)$
and  a right coaction $\g^M:M\rightarrow M\otimes B$, which is a cochain map
making the following diagrams commutative:
$$
\xymatrix{
\ar[dr]_-{\jmath_M^{-1}}
M\ar[r]^-{\g^M} & M\otimes B
\ar[d]^-{\I_M\otimes \ep_B}
\cr
& M\otimes \Bbbk
}\qquad
\xymatrix{
\ar[d]_{\g^M}
M\ar[r]^{\g^M} & M\otimes B \ar[d]^-{\I_M\otimes \cp_B}
\cr
M\otimes B \ar[r]^-{\g^M\otimes \I_B}& M\otimes B\otimes B
.}
$$
That is, $\g^M \in \Hom(M, M\otimes B)^0$ and satisfies
\eqnalign{raction}{
\g^M\circ d_M =d_{M\otimes B}\circ \g^M
,\quad
\begin{cases}
(\g^M\otimes \I_B)\circ \g^M = (\I_M\otimes \cp_B)\circ \g^M
,\cr
(\I_M\otimes \ep_B)\circ \g^M = \jmath^{-1}_M
.
\end{cases}
}
\begin{example} Here are some standard examples.
\begin{enumerate}

\item
The ground field $\Bbbk$ as a cochain complex $(\Bbbk,0)$ is
a right dg-comodule $\big(\Bbbk, \g^{\Bbbk}\big)$ over $B$ with
the coaction $\g^\Bbbk:=\imath^{-1}_B\circ u_B:\xymatrix{\Bbbk \ar[r]^-{u_B}& B \ar[r]^-{\imath^{-1}_B}& \Bbbk\otimes B}$.

\item
The cdg-Hopf algebra $B$ as a cochain complex $(B,d_B)$ is
a right dg-comodule $\big(B, \cp_B\big)$ over $B$ with
the coaction $\cp_B: B \rightarrow B\otimes B$. Note that $\big(B, \cp_B\big)$ is also a left dg-comodule over $B$.

\item
For every cochain complex $M=(M, d_M)$ we have 
a right dg-comodule  $\big(M\otimes B, \I_M\otimes \cp_B\big)$ with the coaction $\I_M\otimes \cp_B: M\otimes B \rightarrow M\otimes B\otimes B$,
called the \emph{cofree} right dg-comodule over $B$ cogenerated by $M$.

\end{enumerate}
\end{example}

A morphism $\xymatrix{(M,\g^M)\ar[r]^{\p}& (M^\pr,\g^{M^\pr})}$ 
of right dg-comodules over $B$ is a  linear map $\p:M\rightarrow M^\pr$ 
making the following diagram commutes:
\eqn{cdgomor}{
\xymatrixcolsep{3pc}
\xymatrix{
\ar[d]_-{\p}M\ar[r]^-{\g^M} & M\otimes B \ar[d]^-{\p\otimes \I_B}
\cr
M^\pr \ar[r]^-{\g^{M^\pr}} &M^\pr\otimes B
}
,\qquad  \hbox{i.e.,}\qquad \g^{M^\pr}\circ \p = (\p\otimes \I_B)\otimes \g^M.
}
We can check that 
$d_{M\!,\, M^\pr}\p: M \rightarrow M^\pr$ is a morphism
of right dg-comodules over $B$ whenever $\p: M \rightarrow M^\pr$ is a morphism of right dg-comodules over $B$:
\eqalign{
\g^{M^\pr}\circ \p = (\p\otimes \I_B)\otimes \g^M
\quad\Longrightarrow\quad
\g^{M^\pr}\circ d_{M\!,\, M^\pr}\p = \big(d_{M\!,\, M^\pr}\p\otimes \I_B\big)\otimes \g^M
 .
 }

The tensor product  of right dg-comodules $\big(M,\g^M\big)$ and $\big(N,\g^N\big)$ 
over $B$ is the  right dg-comodule $\big(M,\g^M\big)\otimes_{m_B}\!\big(N,\g^N\big):=\big(M\otimes N, \g^{M\otimes_{m_B}\! N}\big)$ over $B$,
where $M\otimes N=(M\otimes N, d_{M\otimes N})$ is the tensor product of
the underlying cochain complexes and the  coaction $\g^{M\otimes_{m_B}\! N}: M\otimes N \rightarrow M\otimes N\otimes B$ is
defined by
\eqalign{
\g^{M\otimes_{m_B}\! N} &:= 
( \I_{M}\otimes \I_{N}\otimes m_B)\circ (\I_B\otimes \t \otimes \I_N)\circ (\g^M\otimes \g^N)
,\cr
&\xymatrixcolsep{3pc}
\xymatrix{ M\otimes N \ar[r]^-{\g^M\otimes \g^N} & M\otimes B\otimes N \otimes B
\ar[r]^-{\I_M\otimes \t\otimes \I_B}&M \otimes N\otimes B\otimes B \ar[r]^-{ \I_{M}\otimes \I_{N}\otimes m_B}
&
M\otimes N\otimes B.
}
}
Let
$\xymatrix{(M, \g^{M})\ar[r]^-{\p} & (M^\pr, \g^{M^\pr})}$ and $\xymatrix{(N, \g^{N})\ar[r]^-{\phi} & (N^\pr, \g^{N^\pr})}$
be morphisms of  right dg-comodules over $B$. Then 
the linear map
$\p\otimes\phi: M\otimes N\rightarrow M^\pr\otimes N^\pr$ is
a morphism 
$\xymatrix{(M\otimes N, \g_{M\otimes N})\ar[r]^-{\p\otimes \phi} & (M^\pr\otimes N^\pr, \g_{M^\pr\otimes N^\pr})}$
of right dg-comodules  over $B$, and we have
$$
d_{M\otimes N, M^\pr\otimes N^\pr}\big(\p\otimes\phi\big)
= d_{M, M^\pr}\p\otimes\phi +(-1)^{|\p|}\p\otimes d_{N, N^\pr} \phi.
$$

\begin{definition}[Lemma]
The right dg-comodules over $B$ form a dg-tensor category
$\Big({\dgcat{{dgComod}}_R(B)},\otimes_{m_B}, \big(\Bbbk, \g^{\Bbbk}\big)\Big)$.
\end{definition}
\begin{proof}
Exercise.
\qed
\end{proof}

Consider the functor $\bm{\CG}^B: \category{cdgA}(\Bbbk)\rightsquigarrow \category{Grp}$
represented by cdg-Hopf algebra $B$.

\begin{theorem}\label{repmod}
The dg-category $\dgcat{Rep}\big(\bm{\CG}^B\big)$ of linear representations
of $\bm{\CG}^B$ is isomorphic to the dg-category
of $\dgcat{dgComod}_R(B)$ of right dg-comodules over $B$
as dg-tensor categories.
Explicitly, 
we have an isomorphism of dg-tensor categories 
$$\xymatrix{{\functor{X}}:\dgcat{Rep}\big(\bm{\CG}^B\big)\ar@{~>}@/^/[r]
& \ar@{~>}@/^/[l] \dgcat{dgComod}_R(B): {\functor{Y}}}$$
defined as follows.
\begin{itemize}

\item The functor $\functor{X}$ sends each linear representation $\big(M,\bm{\r}^M\big)$ to the right dg-comodule 
$\big(M,\widebreve\g^{M}\big)$ over $B$, where
\eqnalign{funcx}{
\widebreve\g^{M}:=&{\mq}\big(\bm{\r}^{M}_B(\I_B)\big)= \bm{\r}^M_B(\I_{B})\circ (\I_M\otimes u_B)\circ \jmath^{-1}_M: M\rightarrow M\otimes B
}
and each morphism $\p:\big(M,\bm{\r}^{M}\big)\rightarrow \big(M^\pr,\bm{\r}^{M^\pr}\big)$ of linear representations
to the morphism 
$\p:\big(M,\widebreve{\g}^{M}\big)\rightarrow \big(M^\pr,\widebreve{\g}^{M^\pr}\big)$ of right dg-comodules over $B$.

\item 
The functor $\functor{Y}$ sends each right dg-comodule 
$\big(M,\g^{M}\big)$ over $B$ to the linear representation $\big(M,\widebreve{\bm{\r}}^{M}\big)$,
where the component  $\widebreve{\bm{\r}}^M_A$ of $\widebreve{\bm{\r}}^M$ at a cdg-algebra
$A$ is defined by, $\forall g \in \HOM_{\cdga}(B, A)$,
\eqnalign{funcy}{
\widebreve{\bm{\r}}^{M}_{\!A}(g):=
&{\mp}\Big((\I_M\otimes g)\circ \g^{M}\Big)
\cr
=
& (\I_M\otimes m_A)\circ  (\I_M\otimes g\otimes \I_A)\circ (\g^{M}\otimes \I_A)
:M\otimes A\rightarrow M\otimes A,
}
and each morphism $\p:\big(M,\g^{M}\big)\rightarrow \big(M^\pr,\g^{M^\pr}\big)$ of right dg-comodules over $B$
to the morphism  $\p:\big(M,\widebreve{\bm{\r}}^{M}\big)\rightarrow \big(M^\pr,\widebreve{\bm{\r}}^{\!M^\pr}\big)$ 
of linear representations.
\end{itemize}

\end{theorem}

\begin{proof}
We need to check that both $\functor{X}$
and $\functor{Y}$ are dg-tensor functors and show that they are inverse to each other.

1. We check that  
$\big(M,\widebreve\g^{\!M}\big)=\functor{X}\big(M,\bm{\r}^{\!M}\big)$ is a right dg-comodule over $B$ as follows.

From \eq{univnat}  in Remark  \ref{repyoneda} and the definition \eq{funcx},
we have the following relation for every morphism $g:B\to A$ of cdg-algebras:
\eqnalign{Yoneda lemma implies}{
\check{\mq}\Big(\bm{\r}_A^{\!M}\big(g\big)\Big)=(\I_M\otimes g)\circ\widebreve\g^{\!M}:M\to M\otimes A.
}
The component $\bm{\r}_A^{\!M}$ of  $\bm{\r}^{\!M}$ at  every cdg-algebra $A$, by definition,
is a morphism
$\bm{\r}_A^{\!M}:\bm{\CG}^{\!\!B}(A)\rightarrow \bm{\CG\!\ell}^{\!\!M}(A)$ of groups, i.e.,
for every pair of morphisms $g_1,g_2:B\to A$ of cdg-algebras, we have
\eqnalign{group homomorphism condition}{
\bm{\r}_A^{\!M}\big(u_A\circ\ep_B\big)=\I_{M\otimes A},
\qquad
\bm{\r}_A^{\!M}\big(g_1\star_{B,A}g_2\big)=\bm{\r}_A^{\!M}(g_1)\circ \bm{\r}_A^{\!M}(g_2).
}
\begin{itemize}
\item
Applying $\mq$ on the $1$st equality of \eq{group homomorphism condition}
and using  \eq{Yoneda lemma implies}, we have
\eqnalign{unital action actiom}{
(\I_M\otimes u_A)\circ(\I_M\otimes \ep_B)\circ\widebreve\g^{\!M}=\I_M\otimes u_A.
}
By putting $A=\Bbbk$, we obtain that  $(\I_M\otimes \ep_B)\circ\widebreve\g^{\!M}=\jmath^{-1}_M$.

\item
Applying $\mq$ on the $2$nd equality of \eq{group homomorphism condition}, we have
\eqnalign{showing the action axiom}{
(\I_M\otimes m_A)\circ&(\I_M\otimes g_1\otimes g_2)\circ(\I_M\otimes \cp_B)\circ\widebreve\g^{\!M}
\cr
&
=(\I_M\otimes m_A)\circ(\I_M\otimes g_1\otimes g_2)\circ(\widebreve\g^{\!M}\otimes\I_B)\circ\widebreve\g^{\!M}.
}
Consider the cdg-algebra $B\otimes B$ and the inclusion maps $i_1,i_2:B\to B\otimes B$
\[
i_1:=
\xymatrix{
B \ar[r]^-{\jmath^{-1}_B}&B\otimes \Bbbk\ar[r]^-{\I_M\otimes u_B}& B\otimes B
,
}
\quad\quad
i_2:=
\xymatrix{
B \ar[r]^-{\imath^{-1}_B}&\Bbbk\otimes B\ar[r]^-{u_B\otimes \I_B}& B\otimes B
,
}
\]
which are morphisms of cdg-algebras. 
We can check that $m_{B\otimes B}\circ(i_1\otimes i_2)=\I_{B\otimes B}$ from an elementary calculation.
By substituting $A=B\otimes B$, $g_1=i_1$ and $g_2=i_2$ in \eq{showing the action axiom} we obtain that
$(\I_M\otimes \cp_B)\circ\widebreve\g^{\!M}=(\widebreve\g^{\!M}\otimes\I_B)\circ\widebreve\g^{\!M}$.

\item
Finally we check that  $\widebreve\g^{\!M}:M\to M\otimes B$ is a cochain map:
\eqalign{
d_{M,M\otimes\!B}\widebreve\g^{\!M}=
d_{M,M\otimes B}\mq\big(\bm{\r}^{\!M}_B(\I_B)\big)
=\mq\big(d_{M\otimes B,M\otimes B}\bm{\r}^{\!M}_B(\I_B)\big)
=0,
}
where we have used the facts that $\mq$ is a cochain map defined in Lemma \ref{basicl}
and  $\bm{\r}^{\!M}_{\!B}(\I_B) \in \operatorname{Z^0Aut}_{m_B}(M\otimes B)$.

\end{itemize}

2. We show that $\functor{X}$ is a dg-tensor functor.
Let  $\p:\big(M,\bm{\r}^{\!M}\big)\to \big(M',\bm{\r}^{M'}\big)$ be a morphism of representations of $\bm{\CG}^{\!\!B}$.
Then
$\functor{X}(\p)
=\p:\big(M,\widebreve{\g}^{\!M}\big)\rightarrow \big(M^\pr,\widebreve{\g}^{\!M^\pr}\big)$ 
is a morphism of right dg-comodules over $B$,
since the following diagram commutes
\[
\xymatrixrowsep{2.4pc}
\xymatrixcolsep{5pc}
\xymatrix{
\ar@/^1pc/[rrr]^-{\widebreve\g^M}
M \ar@{..>}[r]_-{\jmath^{-1}_M} \ar[d]_-{\p}&
M\otimes \Bbbk \ar@{..>}[r]_-{\I_M\otimes u_B} \ar@{..>}[d]^-{\p\otimes\I_\Bbbk}&
M\otimes B \ar@{..>}[r]_-{\bm{\r}^{M}_{\!B}(\I_B)} \ar@{..>}[d]^-{\p\otimes \I_B}&
M\otimes B \ar[d]^-{\p\otimes\I_B}
\\
\ar@/_1pc/[rrr]_-{\widebreve\g^{M^\pr}}
M^\pr \ar@{..>}[r]^-{\jmath^{-1}_{M^\pr}}&
M^\pr\otimes \Bbbk \ar@{..>}[r]^-{\I_{M^\pr}\otimes u_B}&
M^\pr\otimes B \ar@{..>}[r]^-{\bm{\r}^{M^\pr}_{\!B}(\I_B)}&
M^\pr\otimes B
}.
\]
Note that the very right square commutes since $\p$ is a morphism of representations
and the commutativity of the other squares are obvious.
It is obvious that 
$\functor{X}(d_{M,M^\pr}\p)=d_{M,M^\pr}\p=d_{M,M^\pr}\big(\functor{X}(\p)\big)$.
Therefore $\functor{X}$ is a dg-functor.
The tensor property of $\functor{X}$ is checked as follows:
\begin{itemize}
\item
From Example \ref{trivalrepsh} we have $\functor{X}\big(\Bbbk,\bm{\r}^\Bbbk\big)=\big(\Bbbk,\g^\Bbbk\big)$.

\item
For two representations $\big(M,\bm{\r}^{\!M}\big)$ and $\big(M^\pr,\bm{\r}^{M^\pr}\big)$ of $\bm{\CG}^{\!\!B}$, we have
\eqalign{
\g^{\functor{X}(M,\bm{\r}^{\!M})\otimes_{m_B}\functor{X}(M^\pr,\bm{\r}^{M^\pr})}
&=
(\I_{M\otimes M^\pr}\otimes m_B)\circ(\I_M\otimes\t\otimes\I_B)\circ
\big(
\g^{\functor{X}(M,\bm{\r}^{\!M})}\otimes\g^{\functor{X}(M^\pr,\bm{\r}^{M^\pr})}
\big)\\
&=
\g^{\functor{X}\left(
(M,\bm{\r}^{\!M})\otimes(M^\pr,\bm{\r}^{M^\pr})
\right)}
.
}
The $1$st equality is from the definition of tensor product of right dg-comodules over $B$, and the $2$nd equality is from the definition of tensor products of representations of $\bm{\CG}^B$.
We conclude that 
$\functor{X}\Big(\big(M,\bm{\r}^{\!M}\big)\otimes\big(M^\pr,\bm{\r}^{M^\pr}\big)\Big)
=\functor{X}\big(M,\bm{\r}^{\!M}\big)\otimes_{m_B}\! \functor{X}\big(M^\pr,\bm{\r}^{M^\pr}\big).$
\end{itemize}

3. 
We show  that 
$\big(M,\widebreve{\bm{\r}}^{\!M}\big)=\functor{Y}\big(M,\g^{\!M}\big)$ 
is a representation of $\bm{\CG}^{\!\!B}$ as follows.
We first show that $\widebreve{\bm{\r}}^{\!M}_A:\bm{\CG}^{\!\!B}(A)\rightarrow \bm{\CG\!\ell}^{\!\!M}(A)$ 
is a morphism of groups for every cdg-algebra $A$:
\begin{itemize}

\item
We have $\widebreve{\bm{\r}}^{\!M}_A(u_A\circ \ep_B)
=  
(\I_M\otimes m_A)\circ\big(\I_M\otimes(u_A\circ\ep_B)\otimes\I_A\big)\circ(\g^M\otimes \I_A)
=
\I_{M\otimes A}$
where we have used the counity of $m_A$ and  the property $(\I_M\otimes\ep_B)\circ{\g}^M=\jmath^{-1}_M$.

\item
For all $g_1,g_2 \in \HOM_{\category{cdgA}(\Bbbk)}(B,A)$:
\eqalign{
\widebreve{\bm{\r}}_A^{\!M}\big(g_1 &\star_{B,A} g_2\big)
:=
(\I_M\otimes m_A)\circ
\Big(
\I_M\otimes\big(m_A\circ(g_1\otimes g_2)\circ\cp_B\big)\otimes \I_A
\Big)
\circ(\g^M\otimes\I_A)
\cr
=
&
(\I_M\otimes m_A)\circ (\I_M\otimes g_1\otimes \I_A)\circ (\g^M\otimes \I_A)
\circ (\I_M\otimes m_A)\circ (\I_M\otimes g_2\otimes \I_A)\circ (\g^M\otimes \I_A)
\cr
=
&\widebreve{\bm{\r}}_A^{\!M}(g_1) \circ \widebreve{\bm{\r}}_A^{\!M}(g_2),
}
where
the $2$nd equality is due to associativity of $m_A$ and the property 
$(\I_M\otimes \cp_B)\circ{\g}^M=(\g^M\otimes \I_B)\circ{\g}^M$.

\item We have  $d_{M\otimes A,M\otimes A}\widebreve{\bm{\r}}_A^{\!M}(g)
=\mp\big((\I_M\otimes d_{B,A}g)\circ\g^{\!M}\big)=0$
for  all $g \in \HOM_{\category{cdgA}(\Bbbk)}(B,A)$,
since both $\mp$ and $\g^{\!M}$ are cochain maps.
\end{itemize}
Combining all the above, we conclude that
$\widebreve{\bm{\r}}_A^M\big(g\big)\in Z^0\Aut_{m_A}(M\otimes A)$ for all $g \in \HOM_{\category{cdgA}(\Bbbk)}(B,A)$
and $\widebreve{\bm{\r}}_A^{\!M}$ is a morphism of groups.
We check the naturalness of $\widebreve{\bm{\r}}^{\!M}$ 
that for every morphism 
$f:A\rightarrow A^\pr$ of cdg-algebras we have  
$\widebreve{\bm{\r}}_{A^\pr}^M\circ\bm{\CG}^{\!\!B}(f) =\bm{\CG\!\ell}^{\!\!M}(f)\circ \widebreve{\bm{\r}}_A^M$
as follows: for all $g \in \HOM_{\category{cdgA}(\Bbbk)}(B,A)$ we have
\eqalign{
\widebreve{\bm{\r}}_{A^\pr}^M\left(\bm{\CG}^{\!\!B}(f)(g)\right)
=&\widebreve{\bm{\r}}_{A^\pr}^{\!M}(f\circ g)
=
\mp\Big((\I_M\otimes f\circ g)\circ\g^{\!M}\Big)
,\cr
\bm{\CG\!\ell}^{\!\!M}(f)\left(\bm{\r}_A^{\!M}(g)\right)
= &
\mp\Big(
(\I_M\otimes f) \circ \mq\left( \bm{\r}_A^{\!M}(g)\right)
\Big)
=
\mp\Big(
(\I_M\otimes f)\circ
\mq\left(\mp\Big(
(\I_M\otimes g)\circ\g^{\!M}
\Big)\right)
\Big)
\cr
=
&
\mp\Big(
(\I_M\otimes f)\circ
(\I_M\otimes g)\circ\g^{\!M}
\Big)
=
\mp\Big((\I_M\otimes f\circ g)\circ\g^{\!M}\Big)
,
}
where we have used  $\mq\circ\mp =\I_{\Hom(M,M\otimes A)}$.

4. We show that $\functor{Y}$ is a dg-tensor functor.
Given a morphism $\p:\big(M,\g^{\!M}\big)\to\big(M^\pr,\g^{\!M^\pr}\big)$ of right dg-comodules over $B$, 
$\functor{Y}(\p)=\p:\big(M,\widebreve{\bm{\r}}^{\!M}\big)\to\big(M^\pr,\widebreve{\bm{\r}}^{\!M^\pr}\big)$
is a morphism of representations, since the following diagram commutes  for every morphism $g:B\to A$ of cdg-algebras:
\[
\xymatrixcolsep{4pc}
\xymatrix{
\ar@/^1pc/[rrr]^-{\widebreve{\bm{\r}}^{M}_{\!A}(g)}
M\otimes A \ar@{..>}[r]_-{\g^M\otimes \I_A} \ar[d]_-{\p\otimes \I_A}&
M\otimes B\otimes A \ar@{..>}[r]_-{\I_M\otimes g\otimes \I_A} \ar@{..>}[d]^-{\p\otimes \I_{B\otimes A}}&
M\otimes A\otimes A \ar@{..>}[r]_-{\I_M\otimes m_A}\ar@{..>}[d]^-{\p\otimes \I_{A\otimes A}}&
M\otimes A\ar[d]^-{\p\otimes \I_A}
\\
\ar@/_1pc/[rrr]_-{\widebreve{\bm{\r}}^{M^\pr}_{\!A}(g)}
M^\pr\otimes A \ar@{..>}[r]^-{\g^{M^\pr}\otimes \I_A}&
M^\pr\otimes B\otimes A \ar@{..>}[r]^-{\I_{M^\pr}\otimes g\otimes \I_A}&
M^\pr\otimes A\otimes A \ar@{..>}[r]^-{\I_{M^\pr}\otimes m_A}&
M^\pr\otimes A
}
\]
It is obvious that $\functor{Y}(d_{M,M^\pr}\p)=d_{M,M^\pr}\p=d_{M,M^\pr}\big(\functor{Y}(\p)\big)$.
Therefore $\functor{Y}$ is a dg-functor.
The tensor property of $\functor{Y}$ is checked as follows:
\begin{itemize}
\item From Example \ref{trivalrepsh}, we have $\functor{Y}\big(\Bbbk,\g^\Bbbk\big)=\big(\Bbbk,\bm{\r}^{\!\Bbbk}\big)$.

\item
Let $\big(M,\g^M\big)$ and $\big(M^\pr,\g^{M^\pr}\big)$ be right dg-comodules over 
$B$, and $g:B\to A$ be a morphism of cdg-algebras. Then by the definition of tensor product of representations, we have
\eqalign{
\mq\left(
\bm{\r}_A^{\!\functor{Y}(M,\g^M)\otimes\functor{Y}(M^\pr,\g^{M^\pr})}(g)
\right)
&=
(\I_{M\!\otimes\! M^\pr}\!\otimes\! m_A)\circ(\I_{M\!\otimes\! M^\pr}\!\otimes\! g\!\otimes\! g)\circ
(\I_M\!\otimes\! \t\!\otimes\! \I_A)\circ\big(\g^M\!\otimes\! \g^{M^\pr}\big)\\
&=
(\I_{M\otimes M^\pr}\otimes g)\circ(\I_{M\otimes M^\pr}\otimes m_B)\circ
(\I_M\otimes \t\otimes \I_A)\circ\big(\g^M\!\otimes\! \g^{M^\pr}\big)\\
&=
(\I_{M\otimes M^\pr}\otimes g)\circ\g^{M\otimes_{m_B}M^\pr}
=
\mq\left(
\bm{\r}_A^{\!\functor{Y}(M\otimes M^\pr,\g^{M\otimes_{m_B}M^\pr})}(g)
\right).
}
Thus we conclude that
$\functor{Y}\left(\big(M,\g^M\big)\otimes_{m_B}\big(M^\pr,\g^{M^\pr}\big)\right)
=\functor{Y}(M,\g^M)\otimes \functor{Y}(M^\pr,\g^{M^\pr}\big)$.
\end{itemize}

5. It is immediate from the constructions that $\functor{X}$ and $\functor{Y}$ are inverse to each other.
\qed
\end{proof}

In the dg-tensor category
$\dgcat{dgComod}_R(B)\cong \dgcat{Rep}(\bm{\CG}^{\!\!B})$,
the structures of a cdg-Hopf algebra $B$ are reflected as the existance of
certain morphisms of right dg-comodules over $B$.
This phenomena will make our Tannakian reconstruction possible.
To begin with, the following lemma reflects the cdg-bialgebra properties of $B$.
Remind that $(B,\cp_B)$ is a right dg-comodule over $B$.

\begin{lemma} \label{First three comodule maps}
We have the following morphisms of right dg-comodules over $B$:
\begin{enumerate}[label=({\alph*})]
\item
The product $m_B:B\otimes B\rightarrow B$ induces a morphism 
$(B\otimes B,\g^{B\otimes_{m_B}B})\xrightarrow{m_B}(B,\cp_B)$.

\item
The unit  $u_B:\Bbbk\rightarrow B$ induces a morphism $(\Bbbk,\g^\Bbbk)\xrightarrow{u_B}(B,\cp_B)$.

\item 
The coaction $\g^M:M\to M\otimes B$ of each right dg-comodule $(M, \g^M)$ over $B$
induces a morphism
$(M,\g^M)\xrightarrow{\g^M}(M\otimes B,\I_M\otimes \cp_B)$.
\end{enumerate}
\end{lemma}

\begin{proof}
\emph{(a)} follows from the property of $m_B$ being a morphism of dg-coalgebras.
Indeed, we have
$(m_B\otimes\I_B)\circ\g^{B\otimes_{m_B}B}
=(m_B\otimes m_B)\circ(\I_B\otimes\tau\otimes\I_B)\circ(\cp_B\otimes\cp_B)
=\cp_B\circ m_B$.
\emph{(b)} follows from the property of $u_B$ being a morphism of dg-coalgebras.
Indeed, we have
$(u_B\otimes\I_B)\circ\g^\Bbbk=(u_B\otimes u_B)\circ\cp_\Bbbk=\cp_B\circ u_B$.
Finally,
\emph{(c)} follows from the coaction axiom that $\g^M$ satisfies.
Indeed, we have
$(\g^M\otimes\I_B)\circ\g^M=(\I_M\otimes\cp_B)\circ\g^M$.
\qed
\end{proof}

Now we examine the roles of the antipode $\vs_B:B\rightarrow B$ of $B$.
Since the antipode is an anti-morphism of dg-coalgebras,
we have a right dg-comodule $(B^*,\g^{B^*})$ over $B$
where $B^*=B$ as a cochain complex but with the alternative coaction
\[
\g^{B^*}:=\tau\circ(\vs_B\otimes\I_B)\circ\cp_B:B^*\to B^*\otimes B.
\]
This is indeed a right $B$-comodule, since
\[
\begin{aligned}
(\g^{B^*}\otimes\I_B)\circ\g^{B^*}
&=\sigma\circ(\tau\otimes\I_B)\circ(\vs_B\otimes\vs_B\otimes\I_B)\circ(\I_B\otimes\cp_B)\circ\cp_B\\
&=\sigma\circ(\tau\otimes\I_B)\circ(\vs_B\otimes\vs_B\otimes\I_B)\circ(\cp_B\otimes\I_B)\circ\cp_B\\
&=\sigma\circ(\cp_B\otimes\I_B)\circ(\vs_B\otimes \I_B)\circ\cp_B\\
&=(\I_B\otimes\cp_B)\circ\g^{B^*}.
\end{aligned}
\]
Here, $\sigma:B^{\otimes3}\to B^{\otimes3}$ is the permutation $\sigma(b_1\otimes b_2\otimes b_3):=(-1)^{|b_1||b_3|+|b_2||b_3|}b_3\otimes b_1\otimes b_2$. We used $\vs_B$ being an anti-morphism of dg-coalgebras on the $3$rd equality.

Using the unit $u_B:\Bbbk\rightarrow B$, every cochain complex $M$
becomes a right dg-comodule $(M_*,\g^{M_*})$ over $B$,
where $M_*=M$ as a cochain complex
and with the coaction $\g^{M_*}:=(\I_M\otimes u_B)\circ\jmath^{-1}_M:M_*\to M_*\otimes B$.

\begin{lemma} \label{Three comodule maps about B*}
We have the following morphisms of right dg-comodules over $B$:
\begin{enumerate}[label=({\alph*})]

\item
The product $m_B:B\otimes B\rightarrow B$ induces a morphism $(B^*\otimes B^*,\g^{B^*\otimes_{m_B}B^*})\xrightarrow{m_B}(B^*,\g^{B^*})$.

\item
The unit  $u_B:\Bbbk\rightarrow B$ induces a morphism   $u_B:(\Bbbk,\g^\Bbbk)\xrightarrow{u_B}(B^*,\g^{B^*})$;

\item 
The coaction $\g^M:M\to M\otimes B$ of each right dg-comodule $(M, \g^M)$ over $B$
induces a morphism
$(M_*,\g^{M_*})\xrightarrow{\g^M}(M\otimes B^*,\g^{M\otimes_{m_B}B^*})$.
\end{enumerate}
\end{lemma}

\begin{proof}
We shall see that \emph{(a)} and \emph{(b)} follows from the property of $\vs_B$ being a morphism of cdg-algebras,
and \emph{(c)} follows from the antipode axiom of $\vs_B$.

For \emph{(a)}, we check that $(m_B\otimes\I_B)\circ\g^{B^*\otimes_{m_B}B^*}
=\g^{B^*}\circ m_B$:
\eqalign{
(m_B\otimes\I_B)\circ\g^{B^*\otimes_{m_B}B^*}
=&(m_B\otimes m_B)\circ(\I_B\otimes\tau\otimes\I_B)\circ(\tau\otimes\tau)\circ(\vs_B\otimes\I_B\otimes\vs_B\otimes\I_B)\circ(\cp_B\otimes\cp_B)\\
=&(m_B\otimes m_B)\circ(\I_B\otimes\I_B\otimes\vs_B\otimes\vs_B)\circ\sigma^\pr\circ(\cp_B\otimes\cp_B),\\
\g^{B^*}\circ m_B
=&\tau\circ(\vs_B\otimes\I_B)\circ\cp_B\circ m_B\\
=&\tau\circ(\vs_B\otimes\I_B)\circ(m_B\otimes m_B)\circ(\I_B\otimes\tau\otimes\I_B)\circ(\cp_B\otimes\cp_B)\\
=&(\I_B\otimes\vs_B)\circ(m_B\otimes m_B)\circ\sigma^\pr\circ(\cp_B\otimes\cp_B),
}
where $\sigma^\pr:=(\I_B\otimes\t\otimes\I_B)\circ(\t\otimes\t):B^{\otimes4}\to B^{\otimes4}$.
From the property $m_B\circ(\vs_B\otimes\vs_B)=\vs_B\circ m_B$,
we have $(m_B\otimes\I_B)\circ\g^{B^*\otimes_{m_B}B^*}=\g^{B^*}\circ m_B$.

For \emph{(b)}, we check that $\g^{B^*}\circ u_B=(u_B\otimes\I_B)\circ\g^\Bbbk$:
\[
\begin{aligned}
\g^{B^*}\circ u_B=
&\tau\circ(\vs_B\otimes\I_B)\circ\cp_B\circ u_B=
\tau\circ(\vs_B\otimes\I_B)\circ(u_B\otimes u_B)\circ\cp_\Bbbk=
(u_B\otimes u_B)\circ\cp_\Bbbk
\cr
=
&(u_B\otimes\I_B)\circ\g^\Bbbk,
\end{aligned}
\]
where we used the property $\vs_B\circ u_B=u_B$ and the cocommutativity of $\cp_\Bbbk$.

For \emph{(c)}, we check that
$\g^{M\otimes_{m_B}B^*}\circ\g^M=(\g^M\otimes\I_B)\circ\g^{M_*}$:
\eqalign{
\g^{M\otimes_{m_B}B^*}&\circ\g^M\\
:=&
(\I_{M\otimes B}\otimes m_B)\circ(\I_M\otimes\tau\otimes\I_B)\circ(\I_{M\otimes B}\otimes \g^{B^*})\circ(\g^M\otimes\I_B)\circ\g^M\\
=&
(\I_M\otimes\tau)\circ(\I_M\otimes m_B\otimes\I_B)\circ(\I_{M\otimes B}\otimes\vs_B\otimes\I_B)\circ(\I_{M\otimes B}\otimes\cp_B)\circ(\g^M\otimes\I_B)\circ\g^M\\
=&
(\I_M\otimes\tau)\circ(\I_M\otimes m_B\otimes\I_B)\circ(\I_{M\otimes B}\otimes\vs_B\otimes\I_B)\circ(\I_{M\otimes B}\otimes\cp_B)\circ(\I_M\otimes \cp_B)\circ\g^M\\
=&
(\I_M\otimes\tau)\circ(\I_M\otimes m_B\otimes\I_B)\circ(\I_{M\otimes B}\otimes\vs_B\otimes\I_B)\circ
(\I_M\otimes\cp_B\otimes\I_B)\circ(\I_M\otimes \cp_B)\circ\g^M\\
=&
(\I_M\otimes\tau)\circ\big(\I_M\otimes(u_B\circ\ep_B)\otimes\I_B\big)\circ(\I_M\otimes\cp_B)\circ\g^M\\
=&
(\I_M\otimes\tau)\circ(\I_M\otimes u_B\otimes\I_B)\circ(\I_M\otimes\ep_B\otimes\I_B)\circ(\g^M\otimes\I_B)\circ\g^M\\
=&
(\I_M\otimes\tau)\circ(\I_M\otimes u_B\otimes \I_B)\circ(\jmath^{-1}_M\otimes\I_B)\circ\g^M\\
=&
(\g^M\otimes\I_B)\circ(\I_M\otimes u_B)\circ\jmath^{-1}_M\\
=&
(\g^M\otimes\I_B)\circ\g^{M_*}.
}
In the above, we used
$(\g^M\otimes\I_B)\circ\g^M=(\I_M\otimes\cp_B)\circ\g^M$ on the $3$rd and $6$th equality,
the coassociativity of $\cp_B$ on the $4$th equality and the antipode axiom
$m_B\circ(\I_B\otimes \vs_B)\circ\cp_B=u_B\circ\ep_B$
on the $5$th equality.
The rest equalities are straightforward.
\qed
\end{proof}

The existence of those  morphisms of right dg-comodules over $B$ in Lemma \ref{First three comodule maps}
and Lemma \ref{Three comodule maps about B*} will play the crucial roles in our Tannaka type reconstruction
theorem.

\section{Tannakian reconstruction theorem for affine group dg-schemes}
Throughout this section we fix a  cdg-Hopf algebra $B=(B, u_B, m_B, \ep_B, \cp_B, \vs_B,d_B)$.
Let   $\bm{\CG}^{\!\!B}:\category{cdgA}(\Bbbk) \rightsquigarrow \category{Grp}$ 
be the functor represented by the cdg-algebra $B$, which induces the functor  $\bm{\mG}^{\!\!B}:\mathit{ho}\category{cdgA}(\Bbbk) \rightsquigarrow \category{Grp}$
represented by $B$
on the homotopy category $\mathit{ho}\category{cdgA}(\Bbbk)$ ---an affine group dg-scheme.

In Sect.\ $5.1$,  we
consider the forgetful  functor $\bm{\o}: \dgcat{dgComod}_R(B)\rightsquigarrow \dgcat{CoCh}(\Bbbk)$
from the dg-category of right dg-comodules over $B$ to the dg-category of cochain complexes over $\Bbbk$
and, out of $\bm{\o}$,  construct two functors 
$$\bm{\CG}^{\bm{\o}}_{\!\otimes}:\category{cdgA}(\Bbbk) \rightsquigarrow \category{Grp}\hbox{ and }
\bm{\mG}^{\bm{\o}}_{\!\otimes}:\mathit{ho}\category{cdgA}(\Bbbk) \rightsquigarrow \category{Grp}
.
$$
 Then we shall establishes 
natural isomorphisms of functors
$\bm{\CG}^{\bm{\o}}_{\!\otimes}\cong \bm{\CG}^{\!\!B}$ and
$\bm{\mG}^{\bm{\o}}_{\!\otimes}\cong \bm{\mG}^{\!B}$,
which is our reconstruction theorem of  an affine group dg-scheme 
from the dg-tensor category of linear representations.

In Sect.\ $5.2$, we consider 
the dg-category $\dgcat{dgComod}_R(B)_{\!f}$ of finite dimensional right dg-comodules over $B$,
which is isomorphic as dg-tensor categories to the dg-category $\dgcat{Rep}(\bm{\CG}^B)_{\!f}$ of finite dimensional linear representations
of $\bm{\CG}^B$. 
From the forgetful  functor $\bm{\o}_{\!f}: \dgcat{dgComod}_R(B)_{\!f}\rightsquigarrow \dgcat{CoCh}(\Bbbk)_{\!f}$
to the dg-category of  finite dimensional cochain complexes, we  construct two functors
$$
\bm{\CG}^{\bm{\o}_f}_{\!\otimes}:\category{cdgA}(\Bbbk) \rightsquigarrow \category{Grp}
\hbox{ and }
\bm{\mG}^{\bm{\o}_f}_{\!\otimes}:\mathit{ho}\category{cdgA}(\Bbbk) \rightsquigarrow \category{Grp}
,
$$
and 
establishes 
natural isomorphisms of functors
$\bm{\CG}^{\bm{\o}_{\!f}}_{\!\otimes}\cong \bm{\CG}^{\!\!B}$ and
$\bm{\mG}^{\bm{\o}_{\!f}}_{\!\otimes}\cong \bm{\mG}^{\!B}$,
which is our reconstruction theorem of  an affine group dg-scheme 
from the dg-tensor category of finite dimensional linear representations.
Our proofs in Sect.\ $5.2$ are based on the reductions of the constructions in  Sect.\ $5.1$ to
the finite dimensional cases, using the fact that every right dg-comodule over $B$ is a filtered colimit of
its finite dimensional subcomodules over $B$.

In Sect.\ $5.3$, we give an independent proof of our $2$nd reconstruction theorem stated in Sect.\ $5.2$,
using the dg-version of rigidity of the dg-tensor category $\dgcat{dgComod}_R(B)_{\!f}$.

\subsection{Reconstruction from the dg-category of linear representations}

Consider the forgetful dg-functor $\bm{\o}: \dgcat{dgComod}_R(B)\rightsquigarrow \dgcat{CoCh}(\Bbbk)$,
which sends a right dg-comodule $(M, \g^M)$ over $B$ to its underlying cochain complex $M$,
and a morphism $\p: (M, \g^M)\rightarrow (M^\pr, \g^{M^\pr})$ of right dg-comodules over $B$
to its underlying $\Bbbk$-linear map $\p: M\rightarrow M^\pr$.
Out of $\bm{\o}$,  we shall construct three functors
$\bm{\CE}^{\bm{\o}}:\cdga\rightsquigarrow \category{dgA}(\Bbbk)$, $\bm{\CG}^{\bm{\o}}_{\!\otimes}:\category{cdgA}(\Bbbk) \rightsquigarrow \category{Grp}$ 
and $\bm{\mG}^{\bm{\o}}_{\!\otimes}:\mathit{ho}\category{cdgA}(\Bbbk) \rightsquigarrow \category{Grp}$ in turns.
Then we shall  establish natural isomorphisms $\bm{\CE}^{\bm{\o}}\cong \bm{\CE}^{B}$, $\bm{\CG}^{\bm{\o}}_{\!\otimes}\cong \bm{\CG}^{\!\!B}$
and $\bm{\mG}^{\bm{\o}}_{\!\otimes}\cong \bm{\mG}^{\!B}$, which constitute our reconstruction theorem.

Note that $\bm{\o}$ is a dg-tensor functor, since it sends
\begin{itemize}
\item
the unit object $(\Bbbk,\g^\Bbbk)$ to  the unit object $\Bbbk$;

\item
the tensor product 
$(M,\g^M)\otimes_{m_B}(M^\pr,\g^{M^\pr})$ of right dg-comodules over $B$
to the tensor product $M\otimes M^\pr$ of the underlying cochain complexes;

\item
the following isomorphisms
\eqalign{ 
&\big((M,\g^M)\otimes_{m_B}(M^\pr,\g^{M^\pr})\big)
\otimes_{m_B}(M^{\pr\pr},\g^{M^{\pr\pr}})
\cong (M,\g^M)\otimes_{m_B}\big((M^\pr,\g^{M^\pr})
\otimes_{m_B}(M^{\pr\pr},\g^{M^{\pr\pr}})\big)
,\cr
&(M,\g^M)\otimes_{m_B}(\Bbbk,\g^\Bbbk)
\cong(M,\g^M)\cong(\Bbbk,\g^\Bbbk)\otimes_{m_B}(M,\g^M)
}
of right dg-comodules over $B$ to the corresponding isomorphisms 
$(M\otimes M^\pr)\otimes M^{\pr\pr}\cong M\otimes(M^\pr\otimes M^{\pr\pr})$
and
$M\otimes \Bbbk\cong M\cong \Bbbk\otimes M$
of the underlying cochain complexes.
\end{itemize}

In Lemma \ref{atensoring}, we have defined a dg-tensor functor
$\otimes A:\dgcat{CoCh}(\Bbbk) \rightsquigarrow \dgcat{dgMod}^{\mathit{fr}}_R(A)$ for each cdg-algebra $A$.
By composing it with $\bm{\o}$, we get a dg-tensor functor
\[
\bm{\o}{\otimes}A: \dgcat{dgComod}_R(B)\rightsquigarrow \dgcat{CoCh}(\Bbbk) \rightsquigarrow \dgcat{dgMod}^{\mathit{fr}}_R(A)
\]
sending
\begin{itemize}
\item
a right dg-comodule $(M, \g^M)$ over $B$ to a free right dg-module $(M\otimes A, \I_M\otimes m_A)$ over $A$, and
\item
a morphism $\p:(M, \g^M)\rightarrow (M^\pr, \g^{M^\pr})$ of right dg-comodules over $B$ to a morphism
$\p\otimes\I_A:(M\otimes A,\I_M\otimes m_A)\to(M'\otimes A,\I_{M'}\otimes m_A)$ of right dg-modules over $A$.
\end{itemize}

Let $\mathsf{End}(\bm{\o}{\otimes}A):=\mathsf{Nat}(\bm{\o}{\otimes}A,\bm{\o}{\otimes}A)$
be the set of natural endomorphisms of the dg-functor $\bm{\o}{\otimes}A$.
We write an element in $\mathsf{End}(\bm{\o}{\otimes}A)$ as $\eta_A$, 
and denote $\eta_A^M$ as its component at a right dg-comodule $(M,\g^M)$ over $B$.
The component of $\eta_A$ at the tensor product $(M,\g^M)\otimes_{m_B\!}(M^\pr,\g^{M^\pr})$ 
is denoted by $\eta_A^{M\otimes_{m_B}M'}$.
Be aware that  the component of $\eta_A$ at the cofree right dg-comodule
$(M\otimes B,\I_M\otimes \cp_B)$ over $B$ is denoted by $\eta_A^{M\otimes B}$.
We have the following structure of dg-algebra on $\mathsf{End}(\bm{\o}{\otimes}A)$:
\eqn{costpd}{
\bm{\CE}^{\bm{\o}}(A):=
\big(\mathsf{End}(\bm{\o}{\otimes}A), \I_{\bm{\o}{\otimes}A}, \circ , \d_A\big)
}
where $\I_{\bm{\o}{\otimes}A}$ is the identity natural transformation, 
$\circ$ is the composition and $\d_A$ is 
the differential given by $(\d_A \eta_A)^M :=d_{M\otimes A, M\otimes A}\eta_A^M$.

\begin{lemma}\label{costpda}
We have a functor
$\bm{\CE}^{\bm{\o}}:\category{cdgA}(\Bbbk) \rightsquigarrow \category{dgA}(\Bbbk)$, sending
\begin{itemize}
\item each cdg-algebra $A$ to the dg-algebra 
$\bm{\CE}^{\bm{\o}}(A)$, and
\item
each morphism $f:A\rightarrow A^\pr$ of cdg-algebras
to a morphism
$\bm{\CE}^{\bm{\o}}(f):\bm{\CE}^{\bm{\o}}(A)\rightarrow \bm{\CE}^{\bm{\o}}(A^\pr)$ of dg-algebras,
where for each
$\eta_A\in\mathsf{End}(\bm{\o}\otimes A)$
the component of $\bm{\CE}^{\bm{\o}}(f)(\eta_A)$
at each right  dg-comodule $(M,\g^M)$ over $B$ is defined by
\eqalign{
\bm{\CE}^{\bm{\o}}(f)(\eta_A)^M:=
&
\mp\Big((f\otimes\I_{A^\pr})\circ\mq(\eta_A^M)\Big)
\cr
=
&
(\I_M\otimes m_{A^\pr})\circ(\I_M\otimes f\otimes\I_{A^\pr})\circ\big(\mq(\eta_A^M)\otimes\I_{A'}\big)
\cr
=
&
(\I_M\otimes m_{A^\pr})\circ(\I_M\otimes f\otimes\I_{A^\pr})\circ
\Big(\big(\eta_A^M\circ(\I_M\otimes A)\circ\jmath^{-1}_M\big)\otimes\I_{A'}\Big).
}
\end{itemize}
\end{lemma}
\begin{proof} 
We first show that
$\bm{\CE}^{\bm{\o}}(f)(\eta_A)$ is an element in $\mathsf{End}(\bm{\o}{\otimes}A^\pr)$ of degree $|\eta_A|$. 
Let $\p:\big(M,\g^M\big)\to \big(M^\pr,\g^{M^\pr}\big)$ be a morphism of right dg-comodules over $B$.
Since $\eta_A$ is a natural transformation, the following diagram commutes:
\[
\xymatrixrowsep{1.3pc}
\xymatrixcolsep{3pc}
\xymatrix{
M\otimes A \ar[r]^-{\p\otimes \I_A} \ar[d]_{\eta_A^M}&
M^\pr\otimes A \ar[d]^{\eta_A^{M^\pr}}\\
M\otimes A \ar[r]^-{\p\otimes \I_A}&
M^\pr\otimes A
}
,\qquad\text{i.e.},\qquad
(\p\otimes\I_A)\circ\eta_A^M=(-1)^{|\p||\eta_A|}\eta_A^{M^\pr}\circ(\p\otimes\I_A).
\]
It follows that
\eqalign{
(\p\otimes\I_{A^\pr})\circ\bm{\CE}^{\bm{\o}}(f)(\eta_A)^M
&=
(\I_{M^\pr}\otimes m_{A^\pr})\circ
\Big(
\big((\p\otimes f)\circ\eta_A^M\circ(\I_M\otimes u_A)\circ \jmath^{-1}_M\big)
\otimes \I_A
\Big)
\cr
&=
(-1)^{|\p||\eta_A|}(\I_{M^\pr}\otimes m_{A^\pr})\circ
\Big(
\big((\I_{M^\pr}\otimes f)\circ\eta_A^{M^\pr}\circ(\p\otimes u_A)\circ \jmath^{-1}_M\big)
\otimes \I_A
\Big)
\cr
&=(-1)^{|\p||\eta_A|}\bm{\CE}^{\bm{\o}}(f)(\eta_A)^{M^\pr}\circ(\p\otimes\I_A).
}
Therefore $\bm{\CE}^{\bm{\o}}(f)(\eta_A)$ is an element in $\mathsf{End}(\bm{\o}{\otimes}A^\pr)^{|\eta_A|}$. 
It remains to show that
\begin{itemize}
\item $\bm{\CE}^{\bm{\o}}(f)$ is a morphism of dg-algebras;
\item $\bm{\CE}^{\bm{\o}}(\I_A)=\I_{\bm{\o}\otimes A}$;
\item $\bm{\CE}^{\bm{\o}}(f^\pr\circ f)
=\bm{\CE}^{\bm{\o}}(f^\pr)\circ\bm{\CE}^{\bm{\o}}(f)$
for another morphism $f^\pr:A^\pr\to A^{\pr\pr}$ of cdg-algebras.
\end{itemize}
These are immediate from the analogous properties of $\bm{\CE}^{M}$
for cochain complexes $M=\bm{\o}(M,\g^M)$, as stated in Lemma \ref{dullpain}.
\qed
\end{proof}

Later in this section we shall construct an isomorphism
$\bm{\CE}^{\bm{\o}}\cong \bm{\CE}^B :\category{cdgA}(\Bbbk) \rightsquigarrow \category{dgA}(\Bbbk)$ of functors,
where  $\bm{\CE}^B$ is the functor defined in  Lemma \ref{grdgschemeone}. 

Now we construct the functor $\bm{\CG}^{\bm{\o}}_{\!\otimes} :\category{cdgA}(\Bbbk) \rightsquigarrow \category{Grp}$,
after some preparations.

\begin{definition}
We consider the following subsets  of $\mathsf{End}(\bm{\o}{\otimes}A)$:

\begin{enumerate}[label=$({\alph*})$,leftmargin=.6cm]

\item $Z^0\mathsf{End}(\bm{\o}{\otimes}A)$
consisting of every element  $\eta_A \in \mathsf{End}(\bm{\o}{\otimes}A)^0$ satisfying $\d_A\eta_A=0$.

\item  $\mathsf{End}_\otimes(\bm{\o}{\otimes}A)$  
consisting of every element  $\eta_A \in \mathsf{End}(\bm{\o}{\otimes}A)^0$ satisfying the conditions
\eqn{tensorial}{ 
\eta_A^{\Bbbk}=\I_{\Bbbk\otimes A}
,\qquad
\eta_A^{M\otimes_{m_B}M'}=\eta_A^M\otimes_{m_A} \eta_A^{M^\pr}
}
for all right dg-comodules $(M,\g^M)$ and $(M^\pr,\g^{M^\pr})$ over $B$.

\item $Z^0\mathsf{End}_\otimes(\bm{\o}{\otimes}A):=Z^0\mathsf{End}(\bm{\o}{\otimes}A)\cap\mathsf{End}_\otimes(\bm{\o}{\otimes}A)$.
\end{enumerate}
\end{definition}

%
%
%
We say an element $\eta_A$ in $\mathsf{End}_\otimes(\bm{\o}{\otimes}A)$ a \emph{tensor} natural transformation, and
an element $\eta_A$ in $Z^0\mathsf{End}_\otimes(\bm{\o}{\otimes}A)$ a \emph{dg-tensor} natural transformation.
We remind that 
\eqn{ctensorial}{
\eta_A^M\otimes_{m_A} \eta_A^{M^\pr}
=(\I_{M\otimes M'}\otimes m_A)\circ(\I_M\otimes\tau\otimes\I_A)\circ\big(\mq(\eta_A^M)\otimes\eta_A^{M'}\big).
}

\begin{lemma}\label{ctensorderi}
If $\eta_A \in \mathsf{End}_\otimes(\bm{\o}{\otimes}A)$ then for right dg-comodules $(M,\g^M)$ and $(M^\pr,\g^{M^\pr})$ over $B$, we have
$
(\d_A\eta_A)^{M\otimes_{m_B\!} M^\pr}=(\d_A\eta_A)^M\otimes_{m_A} \eta_A^{M^\pr}+\eta_A^M\otimes_{m_A} (\d_A\eta_A)^{M^\pr}
$.
\end{lemma}
\begin{proof}
Since $\eta_A$ is a tensor natural transformation, we have
\[
\begin{aligned}
(\d_A\eta_A)^{M\otimes_{m_B\!}M'}
&:=d_{M\otimes M'\otimes A,M\otimes M'\otimes A}\eta_A^{M\otimes_{m_B\!}M'}
=d_{M\otimes M'\otimes A,M\otimes M'\otimes A}\left(\eta_A^M\otimes_{m_A\!}\eta_A^{M'}\right)\\
&=\big(d_{M\otimes A,M\otimes A}\eta_A^M\big)\otimes\eta_A^{M'}+(-1)^{|\eta_A|}\eta_A^M\otimes\big(d_{M'\otimes A,M'\otimes A}\eta_A^{M'}\big)\\
&=(\d_A\eta_A)^M\otimes_{m_A} \eta_A^{M^\pr}+\eta_A^M\otimes_{m_A} (\d_A\eta_A)^{M^\pr}.
\end{aligned}
\]
\qed
\end{proof}
Clearly, the set $Z^0\mathsf{End}_\otimes(\bm{\o}{\otimes}A)$ is closed under composition and contains $\I_{\bm{\o}\otimes A}$.
Thus we have a monoid
\eqn{costgrp}{
\bm{\CG}^{\bm{\o}}_{\!\otimes}(A):=\big(Z^0\mathsf{End}_{\otimes}(\bm{\o}{\otimes}A), \I_{\bm{\o}{\otimes}A}, \circ\big)
}
for every cdg-algebra $A$.
We shall show that this is actually a group. We begin with a technical lemma.

\begin{lemma} \label{cofree comodules and eta}
For every $\eta_A\in\mathsf{End}(\bm{\o}{\otimes}A)$, its component $\eta_A^{M\otimes B}$ 
at the cofree right dg-comodule $(M\otimes B,\I_M\otimes \cp_B)$ over $B$ cogenerated 
by a cochain complex $M$ is $\eta_A^{M\otimes B}=\I_M\otimes\eta_A^B$.
\end{lemma}
\begin{proof}
For each $z\in M$, define a linear map $f_z:B\to M\otimes B$ of degree $|z|$ 
by $f_z(b):=z\otimes b$ for all $b\in B$. 
Then $f_z:(B,\cp_B)\to(M\otimes B,\I_M\otimes \cp_B)$ is a morphism of right dg-comodules over $B$.
Since $\eta_A$ is a natural transformation, the following diagram commutes
\[
\xymatrixrowsep{1.3pc}
\xymatrix{
B\otimes A \ar[r]^-{f_z\otimes \I_A} \ar[d]_{\eta_A^B}&
M\otimes B\otimes A \ar[d]^{\eta_A^{M\otimes B}}\\
B\otimes A \ar[r]^-{f_z\otimes \I_A}&
M\otimes B\otimes A
}
,\qquad\text{i.e.},\qquad
\eta_A^{M\otimes B}\circ(f_z\otimes \I_A)=(-1)^{|\eta_A||z|}(f_z\otimes\I_A)\circ\eta_A^B.
\]
Then for every $b\in B$ and $a\in A$, we have
$\eta_A^{M\otimes B}\big(z\otimes b\otimes a\big)
=\eta_A^{M\otimes B}\circ(f_z\otimes \I_A)\big(b\otimes a\big)
=(-1)^{|\eta_A||z|}(f_z\otimes \I_A)\circ\eta_A^B\big(b\otimes a\big)
=(-1)^{|\eta_A||z|}z\otimes \eta_A^B\big(b\otimes a\big)
=(\I_M\otimes\eta_A^B)\big(z\otimes b\otimes a\big)$.
Since this equality holds for all $a$, $b$ and $z$,
we conclude that $\eta_A^{M\otimes B}=\I_M\otimes \eta_A^B$.
\qed
\end{proof}

\begin{proposition} \label{group valued}
For every $\eta_A\in\mathsf{End}(\bm{\o}{\otimes}A)$ 
we have another natural transformation  $\vs(\eta_A)\in\mathsf{End}(\bm{\o}{\otimes}A)$,
whose component $\vs(\eta_A)^M$ at each right dg-comodule $(M,\g^M)$ over $B$ is defined by
\eqalign{
\vs(\eta_A)^M
:=&(\jmath_M\otimes\I_A)\circ(\I_M\otimes\ep_B\otimes\I_A)\circ\big(\I_M\otimes\eta_A^{B^*}\big)\circ(\g^M\otimes\I_A)
:M\otimes A\to M\otimes A,
}
such that $\vs(\eta_A)\in\mathsf{Z^0End}_\otimes (\bm{\o}{\otimes}A)$ whenever  $\eta_A\in\mathsf{Z^0End}_\otimes (\bm{\o}{\otimes}A)$.
The monoid
$\bm{\CG}^{\bm{\o}}_{\!\otimes}(A)=\big(Z^0\mathsf{End}_{\otimes}(\bm{\o}{\otimes}A), \I_{\bm{\o}\otimes A}, \circ\big)$
is actually a group, where the inverse of $\eta_A\in Z^0\mathsf{End}_{\otimes}(\bm{\o}{\otimes}A)$ is $\vs(\eta_A)$.
\end{proposition}

\begin{proof}
1.
We first check that $\vs(\eta_A)\in\mathsf{End}(\bm{\o}{\otimes}A)$ whenever $\eta_A\in\mathsf{End}(\bm{\o}{\otimes}A)$.
For each morphism $\p: (M, \g^M)\to (M^\pr, \g^{M^\pr})$ of right dg-comodules over $B$ the following diagram commutes:
\[
\xymatrixrowsep{3pc}
\xymatrixcolsep{3.6pc}
\xymatrix{
\ar@/^1pc/[rrrr]^-{\vs(\eta_A)^M}
\ar[d]_{\p\otimes\I_A}
M{\!\otimes\!} A \ar@{..>}[r]_{\g^M\otimes\I_A} &
M{\!\otimes} B {\otimes\!} A \ar@{..>}[r]_-{\I_M\otimes\eta_A^{B^*}} \ar@{..>}[d]^{\p\otimes\I_{B\otimes A}}&
M{\!\otimes\!} B {\!\otimes\!} A \ar@{..>}[r]_-{\I_M\otimes\ep_B\otimes\I_A} \ar@{..>}[d]^{\p\otimes\I_{B\otimes A}}&
M{\!\otimes\!} \Bbbk {\!\otimes\!} A\ar@{..>}[r]_-{\jmath_M\otimes\I_A} \ar@{..>}[d]^{\p\otimes\I_{\Bbbk\otimes A}}&
M{\!\otimes\!} A 
\ar[d]^{\p\otimes\I_A}
\cr
\ar@/_1pc/[rrrr]_-{\vs(\eta_{A})^{M^\pr}}
M'{\!\otimes\!} A \ar@{..>}[r]^{\g^{M^\pr}\otimes\I_A}&
M'{\!\otimes} B {\otimes\!} A \ar@{..>}[r]^-{\I_{M'}\otimes\eta_A^{B^*}}&
M'{\!\otimes\!} B {\!\otimes\!} A \ar@{..>}[r]^-{\I_{M'}\otimes\ep_B\otimes\I_A}&
M'{\!\otimes\!} \Bbbk {\!\otimes\!} A\ar@{..>}[r]^-{\jmath_{M'}\otimes\I_A}&
M'{\!\otimes\!} A.
}
\]
This shows that $\vs(\eta_A)\in\mathsf{End}(\bm{\o}{\otimes}A)$.

2.  We check that $\vs(\eta_A)\in Z^0\mathsf{End}_{\otimes}(\bm{\o}{\otimes}A)$
whenever $\eta_A\in Z^0\mathsf{End}_{\otimes}(\bm{\o}{\otimes}A)$.

It is obvious that $\vs(\eta_A)$ is in $Z^0\mathsf{End}(\bm{\o}{\otimes}A)$
since for every right dg-comodule $(M,\g^M)$ over $B$, all the maps 
$\jmath_M\otimes\I_A$, $\I_M\otimes\ep_B\otimes\I_A$, $\I_M\otimes\eta_A^{B^*}$ and $\g^M\otimes\I_A$ 
are of degree $0$ and are in the kernels of differentials. 

It remains to show that $\vs(\eta_A)\in \mathsf{End}_{\otimes}(\bm{\o}{\otimes}A)$, i.e.,  we have
$\vs(\eta_A)^{M\otimes_{m_B}M'}=\vs(\eta_A)^M\otimes_{m_A}\vs(\eta_A)^{M^\pr}$ 
for all right dg-comodules $(M,\g^M)$ and $(M^\pr,\g^{M^\pr})$ over $B$
and $\vs(\eta_A)^\Bbbk=\I_{\Bbbk\otimes A}$.
From Lemma \ref{Three comodule maps about B*}\emph{(a)}, 
the product $m_B:B\otimes B \rightarrow B$ of $B$
is a morphism $m_B:(B^*\otimes B^*,\g^{B^*\otimes_{m_B}B^*})\rightarrow (B^*,\g^{B^*})$ 
of right dg-comodules over $B$.
Since $\eta_A$ is a tensor natural transformation, the following diagram commutes:
\[
\xymatrixrowsep{1.3pc}
\xymatrixcolsep{2.5pc}
\xymatrix{
B\otimes B\otimes A \ar[r]^-{m_B\otimes\I_A} \ar[d]_{\eta_A^{B^*\otimes_{m_B}B^*}
=\eta_A^{B^*}\otimes_{m_A}\eta_A^{B^*}}&
B\otimes A \ar[d]^{\eta_A^{B^*}}\\
B\otimes B\otimes A \ar[r]^-{m_B\otimes\I_A}&
B\otimes A
}
\quad\hbox{i.e.,}\quad
(m_B\otimes\I_A)\circ\big(\eta_A^{B^*}\otimes_{m_A}\eta_A^{B^*}\big)
=\eta_A^{B^*}\circ(m_B\otimes\I_A).
\]
Using the above property we obtain that
\[
\begin{aligned}
\vs(\eta_A)^M\otimes_{m_A}\vs(\eta_A)^{M^\pr}
=&(\I_M\otimes\jmath_{M^\pr}\otimes\I_A)\circ
\Big(
\I_{M\otimes M^\pr}\otimes
\big((\ep_B\otimes\I_A)\circ(m_B\otimes\I_A)\circ(\eta_A^{B^*}\otimes_{m_A}\eta_A^{B^*})\big)
\Big)\\
&\circ(\I_M\otimes\tau\otimes\I_{B\otimes A})\circ(\g^M\otimes\g^{M^\pr}\otimes\I_A)
\\
=&(\I_M\otimes\jmath_{M^\pr}\otimes\I_A)\circ
\Big(
\I_{M\otimes M^\pr}\otimes
\big((\ep_B\otimes\I_A)\circ\eta_A^{B^*}\circ(m_B\otimes\I_A)\big)
\Big)\\
&\circ(\I_M\otimes\tau\otimes\I_{B\otimes A})\circ(\g^M\otimes\g^{M^\pr}\otimes\I_A)
\\
=&(\I_M\otimes \jmath_{M^\pr}\otimes \I_A)\circ\Big(\I_{M\otimes M^\pr}\otimes
\big((\ep_B\otimes\I_A)\circ\eta_A^{B^*}\big)\Big)\circ(\g^{M\otimes_{m_B}M^\pr}\otimes\I_A)\\
=&\vs(\eta_A)^{M\otimes_{m_B}M'}.
\end{aligned}
\]
From Lemma \ref{Three comodule maps about B*}\emph{(b)}, the unit $u_B:\Bbbk\rightarrow B$ of $B$
 is also a morphism $u_B:(\Bbbk,\g^\Bbbk)\to(B^*,\g^{B^*})$ of right dg-comodules over $B$.
Thus the following diagram also commutes:
\[
\xymatrixrowsep{1.3pc}
\xymatrixcolsep{3pc}
\xymatrix{
\Bbbk\otimes A\ar[r]^{u_B\otimes\I_A} \ar[d]_{\eta_A^\Bbbk=\I_{\Bbbk\otimes A}}&
B\otimes A \ar[d]^{\eta_A^{B^*}}\\
\Bbbk\otimes A \ar[r]^{u_B\otimes\I_A}&
B\otimes A
}
\qquad\hbox{i.e.,}\qquad
\eta_A^{B^*}\circ(u_B\otimes\I_A)=u_B\otimes\I_A.
\]
It follows that $\vs(\eta_A)^\Bbbk=(\ep_B\otimes\I_A)\circ\eta_A^{B^*}\circ(u_B\otimes\I_A)
=(\ep_B\otimes\I_A)\circ(u_B\otimes\I_A)=\I_{\Bbbk\otimes A}$. 

Therefore we have shown that  $\vs(\eta_A)\in Z^0\mathsf{End}_{\otimes}(\bm{\o}{\otimes}A)$. 

3. Finally, we show that $\vs(\eta_A)$ is the inverse of $\eta_A$. 
It suffices to show that $\vs(\eta_A)$ is the right inverse 
of $\eta_A$---we have $\eta_A\circ\vs(\eta_A)=\I_{\bm{\o}\otimes A}$---since every monoid with all right inverses is a group.
Equivalently, we should have $\eta_A^M\circ\vs(\eta_A)^M=\I_{M\otimes A}$ for every  right dg-comodule $(M, \g^M)$ over $B$.
By the definition of $\vs(\eta_A)^M$ and  \eq{ctensorial}, we obtain that
\eqnalign{hlja}{
\eta_A^M\circ\vs(\eta_A)^M
=&\mp(\mq(\eta_A^M))\circ\vs(\eta_A)^M
=(\I_M\otimes m_A)\circ\big(\mq(\eta_A^M)\otimes\I_A\big)\circ\vs(\eta_A)^M
\cr
=&(\jmath_M\otimes\I_A)\circ(\I_M\otimes\ep_B\otimes\I_A)\circ
{\color{blue}\big(\eta_A^M\otimes_{m_A}\eta_A^{B^*}\big)\circ(\g^M\otimes\I_A)}.
}
We claim that 
\eqn{hljb}{
\big(\eta_A^M\otimes_{m_A}\eta_A^{B^*}\big)\circ(\g^M\otimes\I_A)= (\g^M\otimes\I_A)
}
which immediately implies that $\eta_A^M\circ\vs(\eta_A)^M=\I_{M\otimes A}$  since 
we have
$(\jmath_M\otimes\I_A)\circ(\I_M\otimes\ep_B\otimes\I_A)\circ(\g^M\otimes\I_A)=\I_{M\otimes A}$
by the counit property of coaction $\g^M$.
It remains to check the claim.
Lemma \ref{Three comodule maps about B*}\emph{(c)} states that for each right dg-comodule $(M, \g^M)$ over $B$, 
the coaction $\g^M$
is a morphism $\g^M:(M_*,\g^{M_*})\rightarrow (M\otimes B^*,\g^{M\otimes_{m_B} B^*})$ of right dg-comodules over $B$.
Since $\eta^A$ is a tensor natural transformation, the following diagram commutes
$$
\xymatrixrowsep{1.3pc}
\xymatrixcolsep{3.5pc}
\xymatrix{
M\otimes A \ar[r]^-{\g^M\otimes \I_A}
\ar[d]_-{\eta_A^{M_*}}&
M\otimes B\otimes A
\ar[d]^-{\eta_A^{M}\otimes_{m_A}\eta_A^{B^*}}
\cr
M\otimes A \ar[r]^-{\g^M\otimes\I_A} &
M\otimes B\otimes A
}
\quad\hbox{i.e.,}\quad
\big(\eta_A^{M}\otimes_{m_A}\eta_A^{B^*}\big)\circ(\g^M\otimes\I_A)
=(\g^M\otimes\I_A)\circ\eta_A^{M_*}.
$$
Then the claimed identity \eq{hljb} amounts to $\eta_A^{M_*}=\I_{M\otimes A}$. 
Indeed, by Lemma \ref{First three comodule maps}\emph{(b)} and \emph{(c)},
the unit $u_B:(\Bbbk,\g^\Bbbk)\to(B,m_B)$
and
the coaction $\g^{M_*}:(M_*,\g^{M_*})\to(M\otimes B,\I_M\otimes\cp_B)$ 
are morphisms of right dg-comodules over $B$.
Since $\eta_A$ is a tensor natural transformation, the following diagrams commute:
\[
\xymatrixrowsep{1.3pc}
\xymatrixcolsep{3.5pc}
\xymatrix{
\Bbbk\otimes A \ar[r]^-{u_B\otimes \I_A} \ar[d]_{\eta_A^\Bbbk=\I_{\Bbbk\otimes A}}&
B\otimes A \ar[d]^{\eta_A^B}\\
\Bbbk\otimes A \ar[r]^-{u_B\otimes\I_A}&
B\otimes A
}
,
\qquad\qquad
\xymatrix{
M\otimes A \ar[r]^-{\g^{M_*}\otimes\I_A} \ar[d]_{\eta_A^{M_*}}&
M\otimes B\otimes A \ar[d]^{\eta_A^{M\otimes B}=\I_M\otimes\eta_A^{B}}\\
M\otimes A \ar[r]^-{\g^{M_*}\otimes\I_A}&
M\otimes B\otimes A
}
.
\]
Equality on the right diagram is due to Lemma \ref{cofree comodules and eta}. Therefore, we have
\[
\begin{aligned}
\eta_A^{M_*}
=&\Big(\big(\jmath_M\circ(\I_M\otimes\ep_B)\circ\g^{M_*}\big)\otimes\I_A\Big)\circ\eta_A^{M_*}\\
=&(\jmath_M\otimes\I_A)\circ(\I_M\otimes\ep_B\otimes\I_A)\circ(\g^{M_*}\otimes \I_A)\circ\eta_A^{M_*}\\
=&(\jmath_M\otimes\I_A)\circ(\I_M\otimes\ep_B\otimes\I_A)\circ(\I_M\otimes \eta_A^B)\circ(\g^{M_*}\otimes \I_A)\\
=&(\jmath_M\otimes\I_A)\circ(\I_M\otimes\ep_B\otimes\I_A)\circ\big(\I_M\otimes\eta_A^B\big)\circ(\I_M\otimes u_B\otimes \I_A)\circ(\jmath^{-1}_M\otimes\I_A)\\
=&(\jmath_M\otimes\I_A)\circ(\I_M\otimes\ep_B\otimes\I_A)\circ(\I_M\otimes u_B\otimes \I_A)\circ(\jmath^{-1}_M\otimes\I_A)
=\I_{M\otimes A}.
\end{aligned}
\]
This finish our proof that  the monoid $\bm{\CG}^{\bm{\o}}_{\!\otimes}(A)$ is actually a group.
\qed
\end{proof}

The following lemma shows that the above construction is functorial.

\begin{proposition}\label{costpdx}
We have a functor
$\bm{\CG}^{\bm{\o}}_{\!\otimes}:\category{cdgA}(\Bbbk) \rightsquigarrow \category{Grp}$, sending
\begin{itemize}
\item
each cdg-algebra $A$ to the group $\bm{\CG}^{\bm{\o}}_{\!\otimes}(A)$, and
\item
each morphism $f:A\rightarrow A^\pr$ of cdg-algebras to a morphism
$\bm{\CG}^{\bm{\o}}_{\!\otimes}(f)
:\bm{\CG}^{\bm{\o}}_{\!\otimes}(A)\rightarrow \bm{\CG}^{\bm{\o}}_{\!\otimes}(A^\pr)$
of groups defined by $\bm{\CG}^{\bm{\o}}_{\!\otimes}(f):=\bm{\CE}^{\bm{\o}}(f)$.
\end{itemize}
\end{proposition}

\begin{proof}
In Proposition \ref{costpdx},
we have already showed that $\bm{\CG}^{\bm{\o}}_{\!\otimes}(A)$ is a group for every $A$.
To show that $\bm{\CG}^{\bm{\o}}_{\!\otimes}(f)$ is a group homomorphism,
it suffices to check that for every morphism $f:A\to A^\pr$ of cdg-algebras, we have
$\bm{\CG}^{\bm{\o}}(f)(\eta_A)\in Z^0\mathsf{End}_\otimes (\bm{\o}{\otimes}A^\pr)$
whenever
$\eta_A\in Z^0\mathsf{End}_\otimes (\bm{\o}{\otimes}A)$, i.e.,
\begin{enumerate}[label=\hbox{ }(\arabic*),leftmargin=.6cm]

\item
$\bm{\CG}^{\bm{\o}}_{\!\otimes}(f)(\eta_A)\in Z^0\mathsf{End}_\otimes (\bm{\o}{\otimes}A^\pr)$;

\item
$\bm{\CG}^{\bm{\o}}_{\!\otimes}(f)(\eta_A)^\Bbbk=\I_{\Bbbk\otimes A^\pr}$;

\item
$\bm{\CG}^{\bm{\o}}_{\!\otimes}(f)(\eta_A)^{M\otimes_{m_B}M^\pr}
=\bm{\CG}^{\bm{\o}}_{\!\otimes} (f)(\eta_A)^M\otimes_{m_{A^\pr}}\bm{\CG}^{\bm{\o}}_{\!\otimes}(f)(\eta_A)^{M^\pr}$
for all right dg-comodules $(M,\g^M)$ and $(M^\pr,\g^{M^\pr})$ over $B$.
\end{enumerate}
Then, $\bm{\CG}^{\bm{\o}}_{\!\otimes}(f)=\bm{\CE}^{\bm{\o}}(f)$ is a group homomorphism due to Lemma \ref{costpda}, 
which also implies the functoriality of  $\bm{\CG}^{\bm{\o}}_{\!\otimes}(f)$.

Property $(1)$ is obvious since $\bm{\CG}^{\bm{\o}}_{\!\otimes}(f)$ is a cochain map. Property $(2)$ follows from
$\eta_A^\Bbbk=\I_{\Bbbk\otimes A}$ and $f\circ u_A=u_{A^\pr}$, since we have
\[
\begin{aligned}
\bm{\CG}^{\bm{\o}}_{\!\otimes}(f)(\eta_A)^\Bbbk
&=(\I_\Bbbk\otimes m_{A^\pr})\circ(\I_\Bbbk\otimes f\otimes \I_{A^\pr})
\circ(\eta_A^\Bbbk\otimes \I_{A^\pr})\circ(\I_\Bbbk\otimes u_A\otimes \I_{A^\pr})\circ(\cp_\Bbbk\otimes \I_{A^\pr})\\
&=(\I_\Bbbk\otimes m_{A^\pr})\circ(\I_\Bbbk\otimes u_{A^\pr}\otimes \I_{A^\pr})\circ(\cp_\Bbbk\otimes \I_{A^\pr})
=\I_{\Bbbk\otimes A^\pr}.
\end{aligned}
\]
Note that Property $(3)$ is equivalent to the condition
\[
\mq\left(\bm{\CG}^{\bm{\o}}_{\!\otimes}(f)(\eta_A)^{M\otimes_{m_B}M^\pr}\right)
=
\mq\big(
\bm{\CG}^{\bm{\o}}_{\!\otimes} (f)(\eta_A)^M \otimes_{m_{A^\pr}}\bm{\CG}^{\bm{\o}}_{\!\otimes}(f)(\eta_A)^{M^\pr}
\big),
\]
which can be checked as follows:
\eqalign{
\mq\big(
\bm{\CG}^{\bm{\o}}_{\!\otimes} (f)(\eta_A)^M & \otimes_{m_{A^\pr}}\bm{\CG}^{\bm{\o}}_{\!\otimes}(f)(\eta_A)^{M^\pr}
\big)\\
&=(\I_{M\otimes M^\pr}\otimes m_{A^\pr})\circ(\I_M\otimes\t\otimes\I_{A^\pr})\circ\left(
\mq\big(\bm{\CG}^{\bm{\o}}_{\!\otimes} (f)(\eta_A)^M\big)\otimes
\mq\big(\bm{\CG}^{\bm{\o}}_{\!\otimes}(f)(\eta_A)^{M^\pr}\big)
\right)\\
&=(\I_{M\otimes M^\pr}\otimes m_{A^\pr})\circ(\I_M\otimes\t\otimes\I_{A^\pr})\circ
(\I_M\otimes f\otimes \I_{M^\pr}\otimes f)\circ\big(
\mq(\eta_A^M)\otimes\mq(\eta_A^{M^\pr})
\big)\\
&=
(\I_{M\otimes M^\pr}\otimes f)\circ(\I_{M\otimes M^\pr}\otimes m_A)\circ(\I_M\otimes\t\otimes\I_A)\circ
\big(
\mq(\eta_A^M)\otimes\mq(\eta_A^{M^\pr})
\big)\\
&=
(\I_{M\otimes M^\pr}\otimes f)\circ\mq\left(\eta_A^M\otimes_{m_A}\eta_A^{M^\pr}\right)
=(\I_{M\otimes M^\pr}\otimes f)\circ\mq\left(\eta_A^{M\otimes_{m_B}M^\pr}\right)\\
&=\mq\left(\bm{\CG}^{\bm{\o}}_{\!\otimes}(f)(\eta_A)^{M\otimes_{m_B}M^\pr}\right)
,
}
where we have used $f\circ m_A=m_{A^\pr}\circ(f\otimes f)$ for the $3$rd equality
and $\eta_A^M\otimes_{m_A}\eta_A^{M^\pr}=\eta_A^{M\otimes_{m_B}M^\pr}$
for the $5$th equality.
\qed
\end{proof}

Later in this section we shall construct an isomorphism
$\bm{\CG}^{\bm{\o}}_{\!\otimes}\cong \bm{\CG}^B :\category{cdgA}(\Bbbk) \rightsquigarrow \category{dgA}(\Bbbk)$ of functors,
where  $\bm{\CG}^B$ is the functor represented by the cdg-Hopf algebra $B$ as defined in  Lemma \ref{grdgschemetwo}.

We remind that the group $\bm{\CG}^{B}(A)$ for each cdg-algebra $A$ is the
group formed by the set $\HOM_{\category{cdgA}(\Bbbk)}(B,A)$ of morphisms of cdg-algebras.
We also remind that the functor $\bm{\CG}^B :\category{cdgA}(\Bbbk) \rightsquigarrow \category{dgA}(\Bbbk)$
induces the functor $\bm{\mG}^B :\mathit{ho}\category{cdgA}(\Bbbk) \rightsquigarrow \category{dgA}(\Bbbk)$
on the homotopy category $\mathit{ho}\category{cdgA}(\Bbbk)$, where 
$\bm{\mG}^{\!B}(A)$ is the group formed 
by the set $\HOM_{\mathit{ho}\category{cdgA}(\Bbbk)}(B,A)$ of homotopy types of elements in 
$\HOM_{\category{cdgA}(\Bbbk)}(B,A)$.  
Likewise, we need to define homotopy types of elements in 
$Z^0\mathsf{End}_{\otimes}(\bm{\o}{\otimes}A)$.
Note that taking cohomology classes is not compatible with the tensor condition \eq{tensorial}.
Indeed, let
$\eta_A \in  Z^0\mathsf{End}_\otimes\big(\bm{\o}{\otimes}A)$ and  
$\tilde\eta_A = \eta_A +\d_A\l_A$ for some $\l_A \in \mathsf{End}(\bm{\o}{\otimes}A)$ of degree $-1$.
Then $\tilde\eta_A$ and $\eta_A$ belong to the same cohomology class but $\tilde\eta_A$, in general,
is not a tensor natural transformation.

\begin{definition}
A homotopy pair on $Z^0\mathsf{End}_{\otimes}(\bm{\o}{\otimes}A)$ is a pair of one parameter families
$\big(\eta(t)_A, \l(t)_A\big)\in\mathsf{End}(\bm{\o}{\otimes}A)^0[t]\oplus \mathsf{End}(\bm{\o}{\otimes}A)^{-1}[t]$,
parametrized by the time variable $t$ with polynomial dependence, satisfying the homotopy flow equation
$\Fr{d}{dt}\eta(t)_A= \d_A\l(t)_A$ generated by $\eta(t)_A$ subject to the following
conditions:
\eqalign{
\eta(0)_A\in  Z^0\mathsf{End}_\otimes(\bm{\o}{\otimes}A)
,\quad
\begin{cases}
\l(t)_A^{\Bbbk}=0,\cr
\l(t)_A^{M\otimes_{m_B\!} M^\pr} 
= \l(t)_A^M\otimes_{m_A} \eta(t)_A^{M^\pr} +\eta(t)_A^M\otimes_{m_A} \l(t)_A^{M^\pr}
.
\end{cases}
}
\end{definition}

Let $\big(\eta(t)_A, \l(t)_A\big)$ be a homotopy pair on $Z^0\mathsf{End}_{\otimes}(\bm{\o}{\otimes}A)$.
It follows from the homotopy flow equation that $\eta(t)_A$ is uniquely determined by
$\eta(t)_A= \eta(0)_A + \d_A\int^t_0\l(s)_A \mathit{ds}$, and we have $\d_A\eta(t)_A=0$ since $\d_A\eta(0)_A=0$.
From the condition $\eta_A(0)^{\Bbbk}=\I_{\Bbbk\otimes A}$ and $\l(t)_A^{\Bbbk}=0$, we have
$\eta_A^{\Bbbk}(t)=\I_{\Bbbk\otimes A}$.
Moreover, by applying  Lemma \ref{ctensorderi}, we can check that
\eqalign{
\Fr{d}{dt}\Big(\eta(t)_A^{M\otimes_{m_B}\! M^\pr}&-\eta(t)_A^M\otimes_{m_A}\! \eta(t)_A^{M^\pr}\Big)
\cr
&
=\d_A\left(\l(t)_A^{M\otimes_{m_B}\! M^\pr} 
- \l(t)_A^M\otimes_{m_A}\! \eta(t)_A^{M^\pr} -\eta(t)_A^M\otimes_{m_A}\! \l(t)_A^{M^\pr}\right)
\cr
&
=0.
}
It follows that we  have $\eta(t)_A^{M\otimes_{m_B\!} M^\pr} =\eta(t)_A^M\otimes_{m_A} \eta(t)_A^{M^\pr}$
for all $t$ since $\eta(0)_A^{M\otimes_{m_B\!} M^\pr} =\eta(0)_A^M\otimes_{m_A} \eta(0)_A^{M^\pr}$.
Therefore  $\eta(t)_A$ is a family of elements in $Z^0\mathsf{End}_{\otimes}(\bm{\o}\!\otimes \! A)$.
Then,  we declare that $\eta(1)_A$ is \emph{homotopic} to $\eta(0)_A$ by the homotopy $\int^1_0\l(t)_A dt$,
and denote $\eta(0)_A\sim \eta(1)_A$, which is clearly an equivalence relation.
In other words,  two elements $\eta_A$ and $\tilde\eta_A$ 
in the set $Z^0\mathsf{End}_{\otimes}(\bm{\o}\!\otimes \! A)$ 
are homotopic if there is a homotopy flow connecting  them (by the time $1$ map).  
Then,  we also say that $\eta_A$ and $\tilde\eta_A$ have the same homotopy type, 
and denote it as $[\eta_A]=[\tilde \eta_A]$.

Let $\mathit{ho}Z^0\mathsf{End}_{\otimes}(\bm{\o}\!\otimes \! A)$ be the set of homotopy types of elements in
$Z^0\mathsf{End}_{\otimes}(\bm{\o}\!\otimes \! A)$.
It is a routine check that 
$\eta^\pr_A\circ \eta_A \sim \tilde\eta^\pr_A\circ \tilde\eta_A \in Z^0\mathsf{End}_{\otimes}(\bm{\o}\!\otimes \! A)$
whenever $\eta^{\pr}_A\sim \tilde\eta^{\pr}_A, \eta_A \sim \tilde\eta_A 
\in Z^0\mathsf{End}_{\otimes}(\bm{\o}\!\otimes \! A)$
and the homotopy type of $\eta^{\pr}_A\circ \eta_A$ 
depends only on the homotopy types of $\eta^{\pr}_A$ and $\eta_A$.
Therefore we have a well-defined associative composition 
$[\eta^\pr_A]\diamond [\eta_A]:= [\eta^\pr_A\circ \eta_A]$. 
This shows that we have a group
\eqnalign{defining hoZEnd}{
\bm{\mG}^{\bm{\o}}_{\!\otimes}(A):=
\big( \mathit{ho}Z^0\mathsf{End}_{\otimes}(\bm{\o}\!\otimes \! A), [\I_{\bm{\o}\otimes A}], \diamond\big).
}
The following lemma shows that this construction is functorial.

\begin{proposition}\label{costpdy}
We have a functor
$\bm{\mG}^{\bm{\o}}_{\!\otimes}:\mathit{ho}\category{cdgA}(\Bbbk) \rightsquigarrow \category{Grp}$, sending
\begin{itemize}

\item
each cdg-algebra $A$ to the group 
$\bm{\mG}^{\bm{\o}}_{\!\otimes}(A)$, and

\item
each morphism $[f]\in \HOM_{\hcdga}\big(A, A^\pr\big)$ 
to a morphism 
$\bm{\mG}^{\bm{\o}}_{\!\otimes}([f]):\bm{\mG}^{\bm{\o}}_{\!\otimes}(A)\rightarrow \bm{\mG}^{\bm{\o}}_{\!\otimes}(A^\pr)$
of groups defined by
$\bm{\mG}^{\bm{\o}}_{\!\otimes}([f])\big([\eta_A]\big):=
\Big[\bm{\CG}^{\bm{\o}}_{\!\otimes}(f)\big(\eta_A\big)\Big]$
for all $[\eta_A]\in \mathit{ho}Z^0\mathsf{End}_{\otimes}(M\otimes A)$,
where $f \in \HOM_{\cdga}\big(A, A^\pr\big)$ and $\eta_A \in \mathsf{Z^0End}_\otimes (M\otimes A)$
are arbitrary representatives of $[f]$ and $[\eta_A]$, respectively.
\end{itemize}
\end{proposition}

\begin{proof}
All we need to show is that 
$\bm{\CG}^{\bm{\o}}_{\!\otimes}(f)\big(\eta_A\big)
\sim
\bm{\CG}^{\bm{\o}}_{\!\otimes}(\tilde{f})\big(\tilde{\eta}_A\big)$
in $Z^0\mathsf{End}_{\otimes}(\bm{\o}{\otimes}A^\pr)$
whenever $f\sim \tilde{f}$ in $\HOM_{\category{cdgA}(\Bbbk)}(A,A^\pr)$ and $\eta_A\sim\tilde{\eta}_A$
in $Z^0\mathsf{End}_{\otimes}(\bm{\o}{\otimes}A)$.
It suffices to show the following statement: 
Let $\big(f(t),s(t)\big)$ be a homotopy pair on $\HOM_{\category{cdgA}(\Bbbk)}(A,A^\pr)$ and $\big(\eta(t)_A,\l(t)_A\big)$ 
be a homotopy pair on $Z^0\mathsf{End}_{\otimes}(\bm{\o}{\otimes}A)$. Then the pair
\[
\Big(
\nu(t)_{A^\pr}:=
\bm{\CE}^{\bm{\o}}\big(f(t)\big)\big(\eta(t)_A\big),
\quad
\chi(t)_{A^\pr}:=
\bm{\CE}^{\bm{\o}}\big(f(t)\big)\big(\l(t)_A\big)+
\bm{\CE}^{\bm{\o}}\big(s(t)\big)\big(\eta(t)_A\big)
\Big)
\]
is a homotopy pair on $Z^0\mathsf{End}_{\otimes}(\bm{\o}{\otimes}A^\pr)$, i.e.,
\begin{enumerate}[label=\hbox{ }(\arabic*),leftmargin=.6cm]
\item
$\frac{d}{dt}\nu(t)_{A^\pr}=\d_{A^\pr}\chi(t)_{A^\pr}$;

\item
$\nu(0)_A\in Z^0\mathsf{End}_{\otimes}(\bm{\o}{\otimes}A^\pr)$;

\item
$\chi(t)_{A^\pr}^\Bbbk=0$;

\item
$\chi(t)_{A^\pr}^{M\otimes_{m_B}M^\pr}
=\chi(t)_{A^\pr}^M\otimes_{m_{A^\pr}}\nu(t)_{A^\pr}^{M^\pr}
+\nu(t)_{A^\pr}^M\otimes_{m_{A^\pr}}\chi(t)_{A^\pr}^{M^\pr}$
holds for all right dg-comodules $(M,\g^M)$ and $(M^\pr,\g^{M^\pr})$ over $B$.
\end{enumerate}
For property $(1)$, let $(M,\g^M)$ be a right dg-comodule over $B$. Then we have
\eqalign{
\frac{d}{dt}\nu(t)_{A^\pr}^M
=&\frac{d}{dt}
\mp
\Big(
\big(\I_M\otimes f(t)\big)\circ\mq\big(\eta(t)_A^M\big)
\Big)\\
=&\mp
\Big(
\big(\I_M\otimes d_{A,A^\pr}s(t)\big)\circ\mq\big(\eta(t)_A^M\big)
+
\big(\I_M\otimes f(t)\big)\circ\mq\big((\d_A\l(t)_A)^M\big)
\Big)\\
=&\big(\d_{A^\pr}\chi(t)_{A^\pr}\big)^M,
}
where we used $d_{A,A^\pr}f(t)=0$ and $(\d_A\eta(t)_A)^M=0$ on the $3$rd equality.
Property $(2)$ is obvious since $\eta(0)_A$ is in $Z^0\mathsf{End}_{\otimes}(\bm{\o}{\otimes}A)$
and $f(0):A\to A^\pr$ is a morphism of cdg-algebras.
Property $(3)$ follows from $\l(t)_A^\Bbbk=0$, $\eta(t)_A^\Bbbk=\I_{\Bbbk\otimes A}$
and $s(t)\circ u_A=0$, since we have
\eqalign{
\chi(t)_{A^\pr}^\Bbbk
=&
\mp
\Big(
\big(\I_\Bbbk\otimes f(t)\big)\circ\mq\big(\l(t)_A^\Bbbk\big)
+
\big(\I_\Bbbk\otimes s(t)\big)\circ\mq\big(\eta(t)_A^\Bbbk\big)
\Big)
=
\mp
\Big(
\big(\I_\Bbbk\otimes s(t)\big)\circ(\I_\Bbbk\otimes u_A)\circ\cp_\Bbbk
\Big)
\cr
=&0.
}
Note that Property $(4)$ is equivalent to the condition
\eqn{modifying proof presentation 1}{
\mq\left(
\chi(t)_{A^\pr}^{M\otimes_{m_B}M^\pr}\right)
=
\mq\left(
\chi(t)_{A^\pr}^M\otimes_{m_{A^\pr}}\nu(t)_{A^\pr}^{M^\pr}
+
\nu(t)_{A^\pr}^M\otimes_{m_{A^\pr}}\chi(t)_{A^\pr}^{M^\pr}
\right),
}
which can be checked as follows. We consider the $1$st term in the RHS of \eq{modifying proof presentation 1}:
\eqalign{
\mq\Big(\chi(t)_{A^\pr}^M\otimes_{m_{A^\pr}}\nu(t)_{A^\pr}^{M^\pr}\Big)
=&(\I_{M\otimes M^\pr}\otimes m_{A^\pr})\circ(\I_M\otimes\t\otimes\I_{A^\pr})\circ
\Big(
\mq\big(\chi(t)_{A^\pr}^M\big)\otimes\mq\big(\nu(t)_{A^\pr}^{M^\pr}\big)
\Big)\\
=&
(\I_{M\otimes M^\pr}\otimes m_{A^\pr})\circ(\I_M\otimes\t\otimes\I_{A^\pr})\\
&\circ \big(\I_M\otimes s(t)\otimes \I_{M^\pr}\otimes f(t)\big)\circ
\Big(
\mq\big(\eta(t)_A^M\big)\otimes \mq\big(\eta(t)_A^{M^\pr}\big)
\Big)\\
+&(\I_{M\otimes M^\pr}\otimes m_{A^\pr})\circ(\I_M\otimes\t\otimes\I_{A^\pr})\\
&\circ \big(\I_M\otimes f(t)\otimes \I_{M^\pr}\otimes f(t)\big)\circ
\Big(
\mq\big(\l(t)_A^M\big)\otimes \mq\big(\eta(t)_A^{M^\pr}\big)
\Big).
}
Combining with the similar calculation for the $2$nd term in the RHS of \eq{modifying proof presentation 1}, we obtain that
\[
\begin{aligned}
\mq\Big(
\chi(t)_{A^\pr}^M & \otimes_{m_{A^\pr}}\nu(t)_{A^\pr}^{M^\pr}
+\nu(t)_{A^\pr}^M\otimes_{m_{A^\pr}}\chi(t)_{A^\pr}^{M^\pr}
\Big)\\
=&
(\I_{M\otimes M^\pr}\otimes m_{A^\pr})\circ(\I_M\otimes\t\otimes\I_{A^\pr})\\
&\circ \big(\I_M\otimes s(t)\otimes \I_{M^\pr}\otimes f(t)+\I_M\otimes f(t)\otimes \I_{M^\pr}\otimes s(t)\big)\circ
\Big(
\mq\big(\eta(t)_A^M\big)\otimes \mq\big(\eta(t)_A^{M^\pr}\big)
\Big)\\
+&(\I_{M\otimes M^\pr}\otimes m_{A^\pr})\circ(\I_M\otimes\t\otimes\I_{A^\pr})\\
&\circ \big(\I_M\otimes f(t)\otimes \I_{M^\pr}\otimes f(t)\big)\circ
\Big(
\mq\big(\l(t)_A^M\big)\otimes \mq\big(\eta(t)_A^{M^\pr}\big)
+
\mq\big(\eta(t)_A^M\big)\otimes \mq\big(\l(t)_A^{M^\pr}\big)
\Big)\\
=&
\big(\I_{M\otimes M^\pr}\otimes s(t)\big)\circ(\I_{M\otimes M^\pr}\otimes m_A)\circ(\I_M\otimes\t\otimes\I_A)
\circ
\Big(
\mq\big(\eta(t)_A^M\big)\otimes \mq\big(\eta(t)_A^{M^\pr}\big)
\Big)\\
+&
\big(\I_{M\otimes M^\pr}\otimes f(t)\big)\circ(\I_{M\otimes M^\pr}\otimes m_A)\circ(\I_M\otimes\t\otimes\I_A)\\
&\circ
\Big(
\mq\big(\l(t)_A^M\big)\otimes \mq\big(\eta(t)_A^{M^\pr}\big)
+
\mq\big(\eta(t)_A^M\big)\otimes \mq\big(\l(t)_A^{M^\pr}\big)
\Big)\\
=&
\big(\I_{M\otimes M^\pr}\otimes s(t)\big)\circ
\mq\big(
\eta(t)_A^M\otimes_{m_A}\eta(t)_A^{M^\pr}
\big)\\
+&
\big(\I_{M\otimes M^\pr}\otimes f(t)\big)\circ
\mq\big(
\l(t)_A^M\otimes_{m_A}\eta(t)_A^{M^\pr}+
\eta(t)_A^M\otimes_{m_A}\l(t)_A^{M^\pr}
\big)\\
=&
\big(\I_{M\otimes M^\pr}\otimes s(t)\big)\circ
\mq\big(
\eta(t)_A^{M\otimes_{m_B}M^\pr}
\big)
+
\big(\I_{M\otimes M^\pr}\otimes f(t)\big)\circ
\mq\big(
\l(t)_A^{M\otimes_{m_B}M^\pr}
\big)\\
=&
\mq\big(
\chi(t)_{A^\pr}^{M\otimes_{m_B}M^\pr}
\big).
\end{aligned}
\]
In the above, we used $f(t)\circ m_A=m_{A^\pr}\circ\big(f(t)\otimes f(t)\big)$
and $s(t)\circ m_A=m_{A^\pr}\circ\big(f(t)\otimes s(t)+s(t)\otimes f(t)\big)$
on the $2$nd equality, and used
$\eta(t)_A^M\otimes_{m_A}\eta(t)_A^{M^\pr}=\eta(t)_A^{M\otimes_{m_B}M^\pr}$ and
$\l(t)_A^M\otimes_{m_A}\eta(t)_A^{M^\pr}+\eta(t)_A^M\otimes_{m_A}\l(t)_A^{M^\pr}=\l(t)_A^{M\otimes_{m_B}M^\pr}$
on the $4$th equality.
\qed
\end{proof}

Now we are ready to state the main theorem of this subsection.

\begin{theorem}\label{homainth}
We have a natural isomorphism of functors
$$
\bm{\mG}^{\bm{\o}}_{\!\otimes}\cong \bm{\mG}^{\!B}:
\mathit{ho}\category{cdgA}(\Bbbk) \rightsquigarrow \category{Grp}
.
$$
Equivalently, the functor
$\bm{\mG}^{\bm{\o}}_{\!\otimes}$ is representable and  is represented by the cdg-Hopf algebra $B$.
\end{theorem}

The remaining part of this subsection is devoted to the proof of the above theorem, which
is divided into several pieces.

\begin{proposition}\label{homainpr}
We have natural isomorphisms of functors
\eqalign{
\bm{\CE}^{\bm{\o}}\cong \bm{\CE}^{B}:\category{cdgA}(\Bbbk) \rightsquigarrow \category{dgA}(\Bbbk)
,\qquad\quad
\bm{\CG}^{\bm{\o}}_{\!\otimes}\cong \bm{\CG}^{\!\!B}
:\category{cdgA}(\Bbbk) \rightsquigarrow \category{Grp}
.
}
In particular, the functor $\bm{\CG}^{\bm{\o}}_{\!\otimes}$ is representable
and is represented by the cdg-Hopf algebra $B$.
\end{proposition}

The proof of this proposition is based on the forthcoming two lemmas.
Remind that in Lemma \ref{grdgschemeone}, we defined the dg-algebra
$\bm{\CE}^{\!\!B}(A)=\big(\Hom(B,A), u_A\circ \ep_B, \star_{B,A},d_{B,A}\big)$
for every cdg-algebra $A$.

\begin{lemma}\label{ctanha}
We have an isomorphism
$\xymatrix{\bm{\widebreve{\eta}}_{\!A}: \bm{\CE}^{\!\!B}(A)
\ar@/^/[r] & \ar@/^/[l] \bm{\CE}^{\bm{\o}}(A):\bm{\widebreve{g}}_{\!A}}
$ 
of dg-algebras for every cdg-algebra $A$, where
\begin{itemize}
\item
for each $\a\in \Hom(B,A)$, the component of
$\bm{\widebreve{\eta}}_{\!A}(\a) \in \mathsf{End}(\bm{\o}{\otimes}A)$
at a right dg-comodule $(M,\g^M)$ over $B$ is defined by
\eqalign{
\bm{\widebreve{\eta}}_{\!A}(\a)^{M}
:=&\mp\big((\I_M\otimes \a)\circ\g^M\big)
=(\I_M\otimes m_A)\circ(\I_M \otimes \a \otimes \I_A) \circ (\g^M \otimes \I_A).
}

\item for each $\eta_{\!A} \in\mathsf{End}(\bm{\o}{\otimes}A)$,
the linear map $\bm{\widebreve{g}}_{\!A}(\eta_{\!A}) \in \Hom(B, A)$ is defined by
\eqalign{
\bm{\widebreve{g}}_{\!A}(\eta_A)
:=&\imath_A\circ(\ep_B\otimes\I_A)\circ\mq(\eta_{\!A}^B)
=\imath_A\circ(\ep_B\otimes\I_A)\circ\eta_A^B\circ(\I_B\otimes u_A)\circ\jmath^{-1}_B.
}
\end{itemize}
\end{lemma}

\begin{proof}
The map $\bm{\widebreve{g}}_{\!A}$ is well-defined, 
since $\bm{\widebreve{g}}_{\!A}(\eta_A)$ is obviously a $\Bbbk$-linear map.
The map $\bm{\widebreve{\eta}}_{\!A}$ is also well-defined.
This is because for every morphism $\p:(M,\g^M)\to(M^\pr,\g^{M^\pr})$ of right dg-comodules over $B$, 
we have the following commutative diagram:
\[
\xymatrixrowsep{3pc}
\xymatrixcolsep{3.6pc}
\xymatrix{
\ar@/^1pc/[rrr]^-{\bm{\widebreve{\eta}}_{\!A}(\a)^{M}}
\ar[d]_-{\p\otimes\I_A}
M\otimes A \ar@{..>}[r]_-{\g^M\otimes\I_A} &
M\otimes B\otimes A \ar@{..>}[r]_-{\I_M\otimes\a\otimes\I_A} \ar@{..>}[d]^-{\p\otimes\I_{B\otimes A}}&
M\otimes A\otimes A \ar@{..>}[r]_-{\I_M\otimes m_A} \ar@{..>}[d]^-{\p\otimes\I_{A\otimes A}}&
M\otimes A \ar[d]^-{\p\otimes\I_A}
\cr
\ar@/_1pc/[rrr]_-{\bm{\widebreve{\eta}}_{\!A}(\a)^{M^\pr}}
M^\pr\otimes A \ar@{..>}[r]^-{\g^{M^\pr}\otimes\I_A}&
M^\pr\otimes B\otimes A \ar@{..>}[r]^-{\I_{M^\pr}\otimes\a\otimes\I_A}&
M^\pr\otimes A\otimes A \ar@{..>}[r]^-{\I_{M^\pr}\otimes m_A}&
M^\pr\otimes A
}.
\]
This shows that $\bm{\widebreve{\eta}}_{\!A}(\alpha)$ is a natural transformation.

Next, we check that $\bm{\widebreve{g}}_{\!A}$ and $\bm{\widebreve{\eta}}_{\!A}$ are inverse to each other. 

\begin{itemize}
\item
$\bm{\widebreve{g}}_{\!A}\big(\bm{\widebreve{\eta}}_{\!A}(\a)\big)=\a$ holds for all $\a\in \Hom(B,A)$:
\[
\begin{aligned}
\bm{\widebreve{g}}_{\!A}\big(\bm{\widebreve{\eta}}_{\!A}(\a)\big)
=\imath_A\circ(\ep_B\otimes\I_A)\circ\mq\big(\bm{\widebreve{\eta}}_{\!A}(\a)^B\big)
=\imath_A\circ(\ep_B\otimes\I_A)\circ(\I_B\otimes\alpha)\circ\cp_B
=\alpha.
\end{aligned}
\]

\item
$\bm{\widebreve{\eta}}_{\!A}\big(\bm{\widebreve{g}}_{\!A}(\eta_A)\big)
=\eta_A$ holds for all $\eta_A\in\mathsf{End}(\bm{\o}{\otimes}A)$:
Let $(M,\g^M)$ be a right dg-comodule over $B$.  Lemma \ref{First three comodule maps}\emph{(a)} states that 
$\g^M:\big(\O\otimes M,m_\O\otimes\I_M\big)\to\big(M,\g^M\big)$
is a morphism of right dg-modules over $B$. Since $\eta_A$ is a natural transformation, the following diagram commutes
$$
\xymatrixrowsep{1.3pc}
\xymatrixcolsep{3pc}
\xymatrix{
M\otimes A \ar[r]^-{\g^M\otimes\I_A} \ar[d]^-{\eta_A^M}&
M\otimes B\otimes A \ar[d]^-{\eta_A^{M\otimes B}=\I_M\otimes\eta_A^B}\\
M\otimes A \ar[r]^-{\g^M\otimes\I_A}&
M\otimes B\otimes A
}\quad\hbox{i.e.,}\quad
(\g^M\otimes\I_A)\circ\eta_A^M=(\I_M\otimes \eta_A^B)\circ(\g^M\otimes\I_A).
$$
The equality on the diagram is due to Lemma \ref{cofree comodules and eta}. Thus we have
\eqalign{
\bm{\widebreve{\eta}}_{\!A}\big(\bm{\widebreve{g}}_{\!A}(\eta_A)\big)^M
=
&
(\I_M\otimes m_A)\circ\Big(\I_M \otimes
\big(\imath_A\circ(\ep_B\otimes\I_A)\circ\mq(\eta_A^B)\big)
\otimes \I_A\Big) \circ (\g^M \otimes \I_A)
\cr
=
&
(\I_M\otimes \imath_A)\circ(\I_M\otimes \ep_B\otimes \I_A)\circ
\big(\I_M\otimes\mp(\mq(\eta_A^B))\big)\circ(\g^M\otimes\I_A)
\cr
=
&
(\I_M\otimes \imath_A)\circ(\I_M\otimes \ep_B\otimes \I_A)\circ
(\I_M\otimes\eta_A^B)\circ(\g^M\otimes\I_A)
\cr
=
&
(\I_M\otimes \imath_A)\circ(\I_M\otimes \ep_B\otimes \I_A)\circ
(\g^M\otimes\I_A)\circ\eta_A^M
=\eta_A^M.
}
\end{itemize}

We are left to show that $\bm{\widebreve{\eta}}_{\!A}$ and $\bm{\widebreve{g}}_{\!A}$ are morphisms of dg-algebras.
Since they are inverse to each other, 
it suffices to show that $\bm{\widebreve{\eta}}_{\!A}$ is a morphism of dg-algebras. 
Clearly, $\bm{\widebreve{\eta}}_{\!A}$ is a $\Bbbk$-linear map of degree $0$.
Let $(M,\g^M)$ be a right dg-comodule over $B$.
\begin{itemize}
\item
$\bm{\widebreve{\eta}}_{\!A}$ is a cochain map, 
i.e. $\d_A\circ\bm{\widebreve{\eta}}_{\!A}=\bm{\widebreve{\eta}}_{\!A}\circ d_{B,A}$: For $\forall \a \in \Hom(B,A)$,
\[
\begin{aligned}
\d_A\big(\bm{\widebreve{\eta}}_{\!A}(\alpha)\big)^M
&=d_{M\otimes A,M\otimes A}\big((\I_M\otimes m_A)\circ(\I_M \otimes \a \otimes \I_A) \circ (\g^M \otimes \I_A)\big)\\
&=(\I_M\otimes m_A)\circ(\I_M \otimes d_{B,A}\a \otimes \I_A) \circ (\g^M \otimes \I_A)\\
&=\bm{\widebreve{\eta}}_{\!A}(d_{B,A}\alpha)^M.
\end{aligned}
\]

\item
$\bm{\widebreve{\eta}}_{\!A}$ sends the identity to the identity, i.e. $\bm{\widebreve{\eta}}_{\!A}(u_A\circ \ep_B)=\I_{\bm{\o}\otimes A}$:
$$
\bm{\widebreve{\eta}}_{\!A}(u_A\circ \ep_B)^M
:=(\I_M\otimes m_A)\circ\big(\I_M \otimes (u_A\circ \ep_B) \otimes \I_A\big) \circ (\g^M \otimes \I_A)
=\I_{M\otimes A}.
$$

\item
$\bm{\widebreve{\eta}}_{\!A}$ preserves the binary operations, 
i.e. $\bm{\widebreve{\eta}}_{\!A}(\a_1\star_{B,A} \a_2)
=\bm{\widebreve{\eta}}_{\!A}(\a_1)\circ \bm{\widebreve{\eta}}_{\!A}(\a_2)$
for all $\a_1,\a_2 \in \Hom(B,A)$:
\eqalign{
\bm{\widebreve{\eta}}_{\!A} (\a_1 &\star_{B,A} \a_2)^{M} 
:=
(\I_M\otimes m_A)\circ\big(\I_M \otimes (m_A\circ(\a_1\otimes \a_2)\circ\cp_B) \otimes \I_A\big) \circ (\g^M \otimes \I_A)
\cr
=
&
(\I_M\otimes m_A)\circ(\I_M \otimes \a_1 \otimes \I_A) \circ (\g^M \otimes \I_A)
\circ
(\I_M\otimes m_A)\circ(\I_M \otimes \a_2 \otimes \I_A) \circ (\g^M \otimes \I_A)
\cr
=
&\bm{\widebreve{\eta}}_{\!A}(\a_1)^{M} \circ \bm{\widebreve{\eta}}_{\!A}(\a_2)^{M}.
}
We used the associativity of $m_A$ and the coaction axiom of $\g^M$ on the $2$nd equality.
\end{itemize}
\qed
\end{proof}

In Lemma \ref{grdgschemetwo}, we showed that
$\bm{\CG}^{\!\!B}(A)=\big(\HOM_{\category{cdgA}(\Bbbk)}(B,A), u_A\circ\ep_B, \star_{B,A}\big)$
is a group for every cdg-algebra $A$. 
The inverse of $g \in \HOM_{\category{cdgA}(\Bbbk)}(B,A)$ is given by $g^{-1}:=g\circ\vs_B$.
Remind that 
$\HOM_{\category{cdgA}(\Bbbk)}(B,A)$ is the subset of $\Hom(B,A)$ consisting of morphisms of cdg-algebras:
$$
\HOM_{\category{cdgA}(\Bbbk)}(B,A)
=\Big\{ g \in \Hom(B,A)^0\Big|d_{B,A}g =0, 
\; g\circ m_B =m_A\circ (g\otimes g),
\;   g\circ u_B = u_A \Big\}.
$$

\begin{lemma}\label{ctanhc}
For  every cdg-algebra $A$, the isomorphism in Lemma \ref{ctanha} gives an isomorphism
$\xymatrix{\bm{\widebreve{\eta}}_{\!A}: \bm{\CG}^{\!\!B}(A)\ar@/^/[r] 
& \ar@/^/[l] \bm{\CG}^{\bm{\o}}_{\!\otimes}(A):\bm{\widebreve{g}}_{\!A}}$ 
of groups.
\end{lemma}

\begin{proof}
It suffices to check that $\bm{\widebreve{g}}_{\!A}\left(Z^0\mathsf{End}_\otimes(\bm{\o}{\otimes}A)\right)$
is contained in $\HOM_{\category{cdgA}(\Bbbk)}(B,A)$ 
and $\bm{\widebreve{\eta}}_{\!A}\left(\HOM_{\category{cdgA}(\Bbbk)}(B,A)\right)$
is contained in $Z^0\mathsf{End}_\otimes(\bm{\o}{\otimes}A)$.

1. For every $\eta_A \in Z^0\mathsf{End}_\otimes(\bm{\o}{\otimes}A)$
we have $\bm{\widebreve{g}}_{\!A}(\eta_A)\in \HOM_{\category{cdgA}(\Bbbk)}(B,A)$.
\begin{itemize}
\item
$\bm{\widebreve{g}}_{\!A}(\eta_A)$ is of degree $0$ and $d_{B,A}\bm{\widebreve{g}}_{\!A}(\eta_A)= 0$: 
This is immediate since $\eta_A$ is of degree $0$ with $d_A\eta_A=0$, and $\bm{\widebreve{g}}_{\!A}$ is a cochain map 
by Lemma \ref{ctanha}.

\item
$\bm{\widebreve{g}}_{\!A}(\eta_A)\circ u_B = u_A$: 
Lemma \ref{First three comodule maps}\emph{(b)} states that  the unit
$u_B:(\Bbbk,\g^\Bbbk)\to(B,\cp_B)$ 
is a morphism of right dg-comodules over $B$. Since $\eta_A$ is a tensor natural transformation, the following diagram commutes:
\[
\xymatrixrowsep{1.3pc}
\xymatrixcolsep{3.5pc}
\xymatrix{
\Bbbk\otimes A \ar[r]^-{u_B\otimes \I_A} \ar[d]_-{\eta_A^\Bbbk=\I_{\Bbbk\otimes A}}&
B\otimes A \ar[d]^-{\eta_A^B}\\
\Bbbk\otimes A \ar[r]^-{u_B\otimes \I_A}&
B\otimes A
,}
\quad\hbox{i.e.,}\quad
\eta_A^B\circ(u_B\otimes\I_A)=(u_B\otimes\I_A).
\]
Therefore we have
\eqalign{
\bm{\widebreve{g}}_{\!A}(\eta_A)\circ u_B
&=
\imath_A\circ(\ep_B\otimes\I_A)\circ\eta_A^B\circ(u_B\otimes \I_A)\circ(\I_\Bbbk\otimes u_A)\circ\cp_\Bbbk\\
&=
\imath_A\circ(\ep_B\otimes\I_A)\circ (u_B\otimes \I_A)\circ(\I_\Bbbk\otimes u_A)\circ\cp_\Bbbk=u_A
.
}

\item
$\bm{\widebreve{g}}_{\!A}(\eta_A)\circ m_B = m_A\circ\left(\bm{\widebreve{g}}_{\!A}(\eta_A)\otimes \bm{\widebreve{g}}_{\!A}(\eta_A)\right)$:
Lemma \ref{First three comodule maps}\emph{(c)} states that the product 
$m_B:(B\otimes B,\g^{B\otimes_{m_B}B}) \to (B,\cp_B)$
is a morphism of right dg-comodules over $B$.
Since $\eta_A$ is a tensor natural transformation, the following diagram commutes:
\[
\xymatrixrowsep{2pc}
\xymatrixcolsep{3pc}
\xymatrix{
B\otimes B\otimes A \ar[r]^-{m_B\otimes\I_A} \ar[d]_-{\eta_A^{B\otimes_{m_B}B}=\eta_A^B\otimes_{m_A}\eta_A^B}&
B\otimes A \ar[d]^-{\eta_A^B}\\
B\otimes B\otimes A \ar[r]^-{m_B\otimes\I_A}&
B\otimes A
,}
\quad\hbox{i.e.,}\quad
(m_B\otimes\I_A)\circ(\eta_A^B\otimes_{m_A}\eta_A^B)=\eta_A^B\circ(m_B\otimes \I_A).
\]
Therefore we obtain that
\eqalign{
\bm{\widebreve{g}}_{\!A}(\eta_A)\circ m_B
&=\imath_A\circ(\ep_B\otimes\I_A)\circ\eta_A^B\circ(m_B\otimes \I_A)\circ(\I_{B\otimes B}\otimes u_A)\circ\jmath^{-1}_{B\otimes B}\\
&=\imath_A\circ(\ep_B\otimes\I_A)\circ(m_B\otimes\I_A)\circ(\eta_A^B\otimes_{m_A}\eta_A^B)\circ(\I_{B\otimes B}\otimes u_A)\circ\jmath^{-1}_{B\otimes B}\\
&=m_A\circ\left(\bm{\widebreve{g}}_{\!A}(\eta_A)\otimes \bm{\widebreve{g}}_{\!A}(\eta_A)\right)
.
}
\end{itemize}

2. 
For  every $g \in \HOM_{\category{cdgA}(\Bbbk)}(B,A)$, we have 
$\bm{\widebreve{\eta}}_{\!A}(g) \in Z^0\mathsf{End}_\otimes(\bm{\o}{\otimes}A)$.

\begin{itemize}

\item $\bm{\widebreve{\eta}}_{\!A}(g)$ is of degree $0$ and satisfies $\d_A \bm{\widebreve{\eta}}_{\!A}(g)=0$: 
This is immediate, since $g$ is of degree $0$ with $d_{B,A}g=0$, 
and $\bm{\widebreve{\eta}}_{\!A}$ is a cochain map.

\item $\bm{\widebreve{\eta}}_{\!A}(g)^{\Bbbk}= \I_{\Bbbk\otimes A}$: Using $g\circ u_B=u_A$, we have
\[
\begin{aligned}
\bm{\widebreve{\eta}}_{\!A}(g)^{\Bbbk}
&=(\I_\Bbbk\otimes m_A)\circ(\I_\Bbbk\otimes g\otimes \I_A)\circ(\I_\Bbbk\otimes u_B\otimes \I_A)\circ(\cp_\Bbbk\otimes \I_A)\\
&=(\I_\Bbbk\otimes m_A)\circ(\I_\Bbbk\otimes u_A\otimes \I_A)\circ(\cp_\Bbbk\otimes \I_A)=\I_{\Bbbk\otimes A}.
\end{aligned}
\]

\item

$\bm{\widebreve{\eta}}_{\!A}(g)^{M\otimes_{m_B}\! M^\pr}
=\bm{\widebreve{\eta}}_{\!A}(g)^M\otimes_{m_A}\! \bm{\widebreve{\eta}}_{\!A}(g)^{M^\pr}$
holds for all right dg-comodules $(M,\g^M)$ and $(M^\pr,\g^{M^\pr})$ over $B$.
This is equivalent to the condition
$$
\mq\big(\bm{\widebreve{\eta}}_{\!A}(g)^{M\otimes_{m_B}\! M^\pr}\big)
=
\mq\big(\bm{\widebreve{\eta}}_{\!A}(g)^M\otimes_{m_A}\! \bm{\widebreve{\eta}}_{\!A}(g)^{M^\pr}\big)
,
$$
which can be checked as follows. Using $m_A\circ(g\otimes g)=g\circ m_B$, we have
\[
\begin{aligned}
\mq\big(\bm{\widebreve{\eta}}_{\!A}(g)^M\otimes_{m_A}\!& \bm{\widebreve{\eta}}_{\!A}(g)^{M^\pr}\big)
=
(\I_{M\otimes M^\pr}\otimes m_A)\circ(\I_M\otimes\t\otimes \I_A)\otimes
\Big(
\mq\big(\bm{\widebreve{\eta}}_{\!A}(g)^M\big)\otimes\mq\big(\bm{\widebreve{\eta}}_{\!A}(g)^{M^\pr}\big)
\Big)
\\
&=
(\I_{M\otimes M^\pr}\otimes m_A)\circ(\I_M\otimes\t\otimes\I_A)\circ(\I_M\otimes g\otimes \I_{M^\pr}\otimes g)\circ(\g^M\otimes \g^{M^\pr})
\\
&=
(\I_{M\otimes M^\pr}\otimes g)\circ(\I_{M\otimes M^\pr}\otimes m_B)\circ(\I_M\otimes\t\otimes\I_B)\circ(\g^M\otimes \g^{M^\pr})
\\
&=
(\I_{M\otimes M^\pr}\otimes g)\circ(\g^{M\otimes_{m_B}M^\pr})
=\mq\big(\bm{\widebreve{\eta}}_{\!A}(g)^{M\otimes_{m_B}\! M^\pr}\big).
\end{aligned}
\]
\end{itemize}
\qed
\end{proof}

Now we finish the proof of Proposition \ref{homainpr}.

\begin{proof}[Proposition \ref{homainpr}]
We claim that the isomorphism $\bm{\widebreve{\eta}}_{\!A}:\bm{\CE}^{\!\!B}(A)\to\bm{\CE}^{\bm{\o}}(A)$ 
is natural in $A\in\category{cdgA}(\Bbbk)$. This will give us a natural isomorphism
\[
\bm{\widebreve{\eta}}:\bm{\CE}^{\!\!B}\Longrightarrow\bm{\CE}^{\bm{\o}}
:\category{cdgA}(\Bbbk) \rightsquigarrow \category{dgA}(\Bbbk),
\]
whose component at a cdg-algebra $A$ is $\bm{\widebreve{\eta}}_{\!A}$. Then 
$\bm{\widebreve{g}}=\{\bm{\widebreve{g}}_{\!A}\}$ automatically becomes a natural transformation, 
which is the inverse of $\bm{\widebreve{\eta}}$. 
Moreover, $\bm{\widebreve{\eta}}$ will canonically induce a natural isomorphism
$$
\bm{\widebreve{\eta}}:\bm{\CG}^B\Longrightarrow\bm{\CG}^{\bm{\o}}_{\!\otimes}
:\category{cdgA}(\Bbbk) \rightsquigarrow \category{Grp},
$$
with its inverse, again, $\bm{\widebreve{g}}$.
We need to show that for every morphism $f:A\to A^\pr$  of cdg-algebras  the following diagram commutes:
\[
\xymatrixrowsep{1.3pc}
\xymatrix{
\bm{\CE}^{\!\!B}(A) \ar[r]^-{\bm{\widebreve{\eta}}_{\!A}} \ar[d]_-{\bm{\CE}^{\!\!B}(f)}&
\bm{\CE}^{\bm{\o}}(A) \ar[d]^-{\bm{\CE}^{\bm{\o}}(f)}\\
\bm{\CE}^{\!\!B}(A^\pr) \ar[r]^-{\bm{\widebreve{\eta}}_{A^\pr}}&
\bm{\CE}^{\bm{\o}}(A^\pr),
}
\quad\text{i.e.,}\quad
\bm{\CE}^{\bm{\o}}(f)\circ\bm{\widebreve{\eta}}_{\!A}=\bm{\widebreve{\eta}}_{A^\pr}\circ\bm{\CE}^{\!\!B}(f).
\]
It suffices to show that for every linear map $g:B\to A$ and every right dg-comodule $(M,\g^M)$ over $B$, we have
$\mq\Big(\bm{\CE}^{\bm{\o}}(f)\big(\bm{\widebreve{\eta}}_{\!A}(g)\big)^M\Big)
=
\mq\Big(\bm{\widebreve{\eta}}_{\!A}(f\circ g)^M\Big)$.
Indeed, we have
\eqalign{
\mq\Big(\bm{\CE}^{\bm{\o}}(f)\big(\bm{\widebreve{\eta}}_{\!A}(g)\big)^M\Big)
&=(\I_M\otimes f)\circ\mq\Big(\bm{\widebreve{\eta}}_{\!A}(g)^M\Big)\\
&=(\I_M\otimes f)\circ(\I_M\otimes g)\circ\g^M\\
&=\big(\I_M\otimes(f\circ g)\big)\circ\g^M
=\mq\Big(\bm{\widebreve{\eta}}_{\!A}(f\circ g)^M\Big).
}
\qed
\end{proof}

Finally we can finish the proof of Theorem \ref{homainth}.

\begin{proof}[Theorem \ref{homainth}]
By Proposition \ref{homainpr} and the definitions of the functors
$\bm{\mG}^B$ and $\bm{\mG}^{\bm{\o}}_\otimes$,
it suffices to show that for each cdg-algebra $A$,
\begin{enumerate}[label=$({\alph*})$,leftmargin=.8cm]
\item $\bm{\widebreve{\eta}}_{\!A}$ sends a homotopy pair 
$\big(g(t),\chi(t)\big)$ on $\HOM_{\category{cdgA}(\Bbbk)}(B,A)$
to a homotopy pair
$\Big(\bm{\widebreve{\eta}}_{\!A}\big(g(t)\big),\bm{\widebreve{\eta}}_{\!A}\big(\chi(t)\big)\Big)$
on $Z^0\mathsf{End}_\otimes(\bm{\o}{\otimes}A)$, and

\item $\bm{\widebreve{g}}_{\!A}$ sends a homotopy pair
$\big(\eta(t)_A,\l(t)_A\big)$ on $Z^0\mathsf{End}_\otimes(\bm{\o}{\otimes}A)$
to a homotopy pair $\Big(\bm{\widebreve{g}}_{\!A}\big(\eta(t)_A\big),\bm{\widebreve{g}}_{\!A}\big(\l(t)_A\big)\Big)$
on $\HOM_{\category{cdgA}(\Bbbk)}(B,A)$.

\end{enumerate}
Then $\bm{\widebreve{\eta}}_{\!A}$ and $\bm{\widebreve{g}}_{\!A}$ will give an isomorphism of groups 
$\bm{\mG}^{\bm{\o}}_{\!\otimes}(A)\cong \bm{\mG}^{\!B}(A)$ for every cdg-algebra $A$.
Moreover, this isomorphism is natural in $A\in\category{cdgA}(\Bbbk)$ by Proposition \ref{homainpr} and Lemma \ref{costpdy}
so that we have a natural isomorphism
\[
\bm{\mG}^{\bm{\o}}_{\!\otimes}\cong\bm{\mG}^{\!B}
:\hcdga\rightsquigarrow\category{Grp}.
\]

We will prove $(a)$ only since the proof of $(b)$ is similar. We need to check the following properties:
\begin{enumerate}[label=(\arabic*),leftmargin=.8cm]
\item
$\frac{d}{dt}\bm{\widebreve{\eta}}_{\!A}\big(g(t)\big)=\d_A\bm{\widebreve{\eta}}_{\!A}\big(\chi(t)\big)$;

\item
$\bm{\widebreve{\eta}}_{\!A}\big(g(0)\big)\in Z^0\mathsf{End}_\otimes(\bm{\o}{\otimes}A)$;

\item
$\bm{\widebreve{\eta}}_{\!A}\big(\chi(t)\big)^\Bbbk=0$;

\item
$\bm{\widebreve{\eta}}_{\!A}\big(\chi(t)\big)^{M\otimes_{m_B}M^\pr}
=\bm{\widebreve{\eta}}_{\!A} \big(\chi(t)\big)^M  \otimes_{m_A}\bm{\widebreve{\eta}}_{\!A}\big(g(t)\big)^{M^\pr}
+\bm{\widebreve{\eta}}_{\!A}\big(g(t)\big)^M\otimes_{m_A}\bm{\widebreve{\eta}}_{\!A}\big(\chi(t)\big)^{M^\pr}$
holds for all right dg-comodules $(M,\g^M)$ and $(M^\pr,\g^{M^\pr})$ over $B$.
\end{enumerate}
Property $(1)$ follows from the condition $\frac{d}{dt}g(t)=d_{B,A}\chi(t)$, since we have
$\frac{d}{dt}\bm{\widebreve{\eta}}_{\!A}\big(g(t)\big)
=\bm{\widebreve{\eta}}_{\!A}\big(\frac{d}{dt}g(t)\big)
=\bm{\widebreve{\eta}}_{\!A}\big(d_{B,A}\chi(t)\big)
=\d_A\bm{\widebreve{\eta}}_{\!A}\big(\chi(t)\big)$.
Property $(2)$ follows from the condition $g(0)\in \HOM_{\category{cdgA}(\Bbbk)}(B,A)$.
Property $(3)$ follows from the condition $\chi(t)\circ u_B=0$, since we have
$\bm{\widebreve{\eta}}_{\!A}\big(\chi(t)\big)^\Bbbk
=(\I_\Bbbk\otimes m_A)\circ\big(\I_\Bbbk\otimes(\chi(t)\circ u_B)\otimes\I_A\big)\circ(\cp_\Bbbk\otimes \I_A)
=0$.
Note that Property $(4)$ is equivalent to the identity
\eqnalign{modifying proof presentation 2}{
&
\mq\left(\bm{\widebreve{\eta}}_{\!A}\big(\chi(t)\big)^{M\otimes_{m_B}M^\pr}\right)
\cr
&\qquad
=
\mq\left(
\bm{\widebreve{\eta}}_{\!A} \big(\chi(t)\big)^M  \otimes_{m_A}\bm{\widebreve{\eta}}_{\!A}\big(g(t)\big)^{M^\pr}
+
\bm{\widebreve{\eta}}_{\!A}\big(g(t)\big)^M\otimes_{m_A}\bm{\widebreve{\eta}}_{\!A}\big(\chi(t)\big)^{M^\pr}
\right)
,
}
which can be checked as follows.
We begin with the $1$st term in the RHS of \eq{modifying proof presentation 2}:
\[
\begin{aligned}
\mq\Big(
\bm{\widebreve{\eta}}_{\!A} & \big(\chi(t)\big)^M  \otimes_{m_A}\bm{\widebreve{\eta}}_{\!A}\big(g(t)\big)^{M^\pr}
\Big)\\
&=(\I_{M\otimes M^\pr}\otimes m_A)\circ(\I_M\otimes\t\otimes\I_A)\circ\left(
\mq\Big(\bm{\widebreve{\eta}}_{\!A} \big(\chi(t)\big)^M\Big)\otimes
\mq\Big(\bm{\widebreve{\eta}}_{\!A}\big(g(t)\big)^{M^\pr}\Big)
\right)\\
&=(\I_{M\otimes M^\pr}\otimes m_A)\circ(\I_M\otimes\t\otimes\I_A)\circ
\big(\I_M\otimes \chi(t)\otimes \I_M\otimes g(t)\big)\circ(\g^M\otimes \g^{M^\pr})\\
&=(\I_{M\otimes M^\pr}\otimes m_A)\circ\big(\I_{M\otimes M^\pr}\otimes \chi(t)\otimes g(t)\big)
\circ(\I_M\otimes\t\otimes\I_B)\circ(\g^M\otimes \g^{M^\pr}).
\end{aligned}
\]
After the similar calculation for the $2$nd term in the RHS of \eq{modifying proof presentation 2}, we have
\[
\begin{aligned}
\mq\Big(
& \bm{\widebreve{\eta}}_{\!A} \big(\chi(t)\big)^M  \otimes_{m_A}\bm{\widebreve{\eta}}_{\!A}\big(g(t)\big)^{M^\pr}
+
\bm{\widebreve{\eta}}_{\!A}\big(g(t)\big)^M\otimes_{m_A}\bm{\widebreve{\eta}}_{\!A}\big(\chi(t)\big)^{M^\pr}
\Big)\\
&=(\I_{M\otimes M^\pr}\otimes m_A)\circ
\Big(
\I_{M\otimes M^\pr}\otimes \big(\chi(t)\otimes g(t)+g(t)\otimes \chi(t)\big)
\Big)
\circ(\I_M\otimes\t\otimes\I_B)\circ(\g^M\otimes \g^{M^\pr})\\
&=\big(\I_{M\otimes M^\pr}\otimes \chi(t)\big)\circ (\I_{M\otimes M^\pr}\otimes m_B)
\circ(\I_M\otimes\t\otimes\I_B)\circ(\g^M\otimes \g^{M^\pr})\\
&=\big(\I_{M\otimes M^\pr}\otimes \chi(t)\big)\circ\g^{M\otimes_{m_B}M^\pr}
=\mq\Big(
\bm{\widebreve{\eta}}_{\!A}\big(\chi(t)\big)^{M\otimes_{m_B}M^\pr}
\Big).
\end{aligned}
\]
We used $m_A\circ\big(\chi(t)\otimes g(t)+g(t)\otimes \chi(t)\big)=\chi(t)\circ m_B$ on the $2$nd equality.
\qed
\end{proof}

\subsection{Reduction to the dg-category of finite dimensional linear representations}

 A \emph{finite dimensional} linear representation of $\bm{\CG}^B$ is a representation $(M,\bm{\r}^M)$ of $\bm{\CG}^B$ whose
underlying cochain complex $M$ is finite dimensional over $\Bbbk$. Denote
\begin{itemize}
\item
$\dgcat{CoCh}(\Bbbk)_{\mathit{f}}$ as the full sub dg-category of $\dgcat{CoCh}(\Bbbk)$ consisting of finite dimensional
cochain complexes over $\Bbbk$;

\item
$\dgcat{Rep}(\bm{\CG}^B)_{\mathit{f}}$ as the full sub dg-category of $\dgcat{Rep}(\bm{\CG}^B)$
consisting of finite dimensional linear representations of $\CG^B$;

\item
$\dgcat{dgComod}_R(B)_{\mathit{f}}$ as the full sub dg-category of $\dgcat{dgComod}_R(B)$
consisting of finite dimensional right dg-comodules over $B$.
\end{itemize}
Then $\big(\dgcat{CoCh}(\Bbbk)_{\mathit{f}},\otimes,\Bbbk\big)$,
$\big(\dgcat{Rep}(\bm{\CG}^B)_{\mathit{f}},\bm{\otimes},(\Bbbk,\bm{\r}^\Bbbk)\big)$ and
$\big(\dgcat{dgComod}_R(B)_{\mathit{f}},\otimes_{m_B},(\Bbbk,\g^\Bbbk)\big)$
are dg-tensor categories, and there is an isomorphism of dg-tensor categories
\[
\left(\dgcat{Rep}(\bm{\CG}^B)_{\mathit{f}},\bm{\otimes},(\Bbbk,\bm{\r}^\Bbbk)\right)
\cong
\left(\dgcat{dgComod}_R(B)_{\mathit{f}},\otimes_{m_B},(\Bbbk,\g^\Bbbk)\right)
\]
by the arguments in Theorem \ref{repmod}. We denote
$\bm{\o}_\mathit{f}:\dgcat{dgComod}_R(B)_{\mathit{f}}\rightsquigarrow\dgcat{CoCh}(\Bbbk)_{\mathit{f}}$
as the forgetful functor, which is also a dg-tensor functor.
In this subsection, we define  three functors
\[
\begin{aligned}
\bm{\CE}^{\bm{\o}_\mathit{f}}:\category{cdgA}(\Bbbk)\rightsquigarrow\category{dgA}(\Bbbk),\qquad
\bm{\CG}^{\bm{\o}_\mathit{f}}:\category{cdgA}(\Bbbk)\rightsquigarrow\category{Grp}\quad\text{and}\quad
\bm{\mG}_\otimes^{\!\bm{\o}_\mathit{f}}:\mathit{ho}\category{cdgA}(\Bbbk)\rightsquigarrow\category{Grp},
\end{aligned}
\]
analogous to the three functors $\bm{\CE}^{\bm{\o}_\mathit{f}}$, $\bm{\CG}^{\bm{\o}}$ and $\bm{\CG}^{\bm{\o}}$
in the previous subsection.
We shall  show that there are natural isomorphisms of functors
\[
\bm{\CE}^{\bm{\o}_\mathit{f}}\cong\bm{\CE}^{\bm{\o}},\quad
\bm{\CG}_\otimes^{\bm{\o}_\mathit{f}}\cong \bm{\CG}^{\bm{\o}}_{\!\otimes}\quad\text{and}\quad
\bm{\mG}_\otimes^{\!\bm{\o}_\mathit{f}}\cong \bm{\mG}^{\bm{\o}}_{\!\otimes}.
\]
This will imply that the reconstructions of $\bm{\CG}^B$ and $\bm{\mG}^B$ 
are also valid if we work with their finite dimensional representations.

In this subsection we adapt the Einstein summation convention.
We begin with the  following well-known lemma.

\begin{lemma} \label{finite filteredness of comodules}
Let $B$ be a dg-coalgebra, and $(M,\g^M)$ be a right dg-comodule over $B$. Then $(M,\g^M)$ is the union of its finite dimensional subcomodules over $B$.
\end{lemma}
\begin{proof}
It suffices to show that for each $m\in M$, there is a finite dimensional subcomodule of $M$ containing $m$.
Fix a basis $\{b_i\}$ of $B$ over $\Bbbk$ and write
\eqnalign{finite assumption eq1}{
\cp_B(b_i)=\cp^{jk}_i b_j\otimes b_k,
\qquad
\g^M(m)=m^i\otimes b_i
}
for some $\cp^{jk}_i\in \Bbbk$ and $m^i\in M$. 
Note that the indices of the sums in \eq{finite assumption eq1} run finite. 
Let $M_\mathit{f}^m$ be the $\Bbbk$-subspace of $M$ spanned by the elements $m$, $m^i$, $d_Mm$ and $d_Mm^i$. 
Clearly, $(M_\mathit{f}^m,d_M)$ is a finite dimensional subcomplex of $(M,d_M)$ containing $m$.
Moreover, $(M_\mathit{f}^m,\g^M)$ is a subcomodule of $(M,\g^M)$ over $B$ because
\begin{itemize}
\item
$\g^M(m)=m^i\otimes b_i\in M_\mathit{f}^m\otimes B$.

\item
$\g^M(d_Mm)=d_{M\otimes B}\big(\g^M(m)\big)
=d_Mm^i\otimes b_i+(-1)^{|m^i|}m^i\otimes d_Bb_i\in M_\mathit{f}^m\otimes B$.

\item
the comodule axiom of $\g^M$ implies that
\[
\begin{aligned}
\g^M(m^k)\otimes b_k
&=(\g^M\otimes\I_B)(m^k\otimes b_k)
=(\g^M\otimes\I_B)\circ\g^M\big(m\big)
=(\I_M\otimes \cp_B)\circ \g^M\big(m\big)\\
&=(\I_M\otimes \cp_B)(m^i\otimes b_i)
=\cp^{jk}_i\cdot m^i\otimes b_j\otimes b_k.
\end{aligned}
\]
Therefore we have 
$\g^M(m^k)=\cp^{jk}_i\cdot m^i\otimes b_j\in M_\mathit{f}^m\otimes B$.

\item
$\g^M(d_Mm^k)=d_{M\otimes B}\big(\g^M(m^k)\big)
=\cp^{jk}_i\cdot d_Mm^i\otimes b_j + (-1)^{|m^i|}\cp^{jk}_i\cdot m^i\otimes d_Bb_j
\in M_\mathit{f}^m\otimes B$.
\end{itemize}
Thus for each $m\in M$, there exists a finite dimensional subcomodule $M_\mathit{f}^m$ of $M$ over $B$ containing $m$.
We conclude that $M$ is the union of its finite dimensional subcomodules over $B$.
\qed
\end{proof}

We may restate the result of Lemma \ref{finite filteredness of comodules} as follows: Every right dg-comodule $(M,\g^M)$ over $B$ is a
filtered colimit of its finite dimensional subcomodules over $B$. Equivalently, we have $M\cong\displaystyle\varinjlim_\alpha M^\alpha_\mathit{f}$
where the index $\alpha$ runs over all finite dimensional subcomodules $M^\alpha_\mathit{f}$ of $M$ over $B$.

We will denote an element in $\mathsf{End}(\bm{\o}_{\mathit{f}}\!\otimes\!A)$ as $\xi_A$.
For every $\xi_A\in\mathsf{End}(\bm{\o}_{\mathit{f}}\!\otimes\!A)$, we also have
$M\otimes A\cong\displaystyle\varinjlim_{\a}(M_\mathit{f}^\a\otimes A)$
and there exists a unique morphism
\[
(\text{$\varinjlim$}_A\xi_A)^M:=\varinjlim_{\alpha}\big(\xi_A^{M^\alpha_\mathit{f}}\big)
:M\otimes A\to M\otimes A
\]
of right dg-modules over $A$, making the following diagram commutative for all $\a$:
\[
\xymatrixcolsep{3.5pc}
\xymatrix{
M\otimes A \ar[r]^-{(\varinjlim_A\xi_A)^M}&
M\otimes A\\
M_\mathit{f}^\a\otimes A \ar[r]^-{\xi_A^{M_\mathit{f}^\a}} \ar@{^{(}->}[u]&
M_\mathit{f}^\a\otimes A \ar@{^{(}->}[u]
}.
\]




\begin{lemma} \label{finite lemma 1}
For every cdg-algebra $A$, there is an isomorphism of dg-algebras
\[
\xymatrix{
\varinjlim_A:\mathsf{End}(\bm{\o}_{\mathit{f}}\!\otimes\!A)
\ar@<0.2pc>[r]&
\ar@<0.2pc>[l]
\mathsf{End}(\bm{\o}\!\otimes\!A): \operatorname{res}_A,
}
\]
which sends
\begin{itemize}
\item
each $\eta_A\in \mathsf{End}(\bm{\o}\!\otimes\!A)$ to its restriction
$\operatorname{res}_A\eta_A\in\mathsf{End}(\bm{\o}_{\mathit{f}}\!\otimes\!A)$
whose component at a finite dimensional right dg-comodule $(M_\mathit{f},\g^{M_\mathit{f}})$ over $B$ is given by
\[
(\operatorname{res}_A\eta_A)^{M_\mathit{f}}:=\eta_A^{M_\mathit{f}}:M_\mathit{f}\otimes A\to M_\mathit{f}\otimes A.
\]

\item
each $\xi_A\in\mathsf{End}(\bm{\o}_{\mathit{f}}\!\otimes\!A)$ to
$\varinjlim_A\xi_A\in\mathsf{End}(\bm{\o}\!\otimes\!A)$
whose component at a right dg-comodule $(M,\g^M)$ over $B$ is given by
\[
(\text{$\varinjlim$}_A\xi_A)^M=\varinjlim_\alpha\big(\xi_A^{M_\mathit{f}^\alpha}\big)
:M\otimes A\to M\otimes A,
\]
where the index $\alpha$ runs over all finite dimensional subcomodules $M_\mathit{f}^\alpha$ of $M$ over $B$.
\end{itemize}
Moreover, the above isomorphism is natural in $A\in\category{cdgA}(\Bbbk)$.
\end{lemma}
\begin{proof}
For every $\xi_A\in\mathsf{End}(\bm{\o}_{\mathit{f}}\!\otimes\!A)$, we claim that $\varinjlim_A\xi_A$ is a natural transformation.
Given a morphism $\p:(M,\g^M)\to(M^\pr,\g^{M^\pr})$ of right dg-comodules over $B$, we need to show
\eqn{finite lemma 1 eq 1}{
(\text{$\varinjlim$}_A\xi_A)^M\circ(\p\otimes\I_A)
=(-1)^{|\p||\xi_A|}(\p\otimes\I_A)\circ(\text{$\varinjlim$}_A\xi_A)^{M^\pr}
:M\otimes A\to M^\pr\otimes A.
}
Let $M^\alpha_\mathit{f}$ be a finite dimensional subcomodule of $M$ over $B$.
Then by Lemma \ref{finite filteredness of comodules}, there exists a finite dimensional subcomodule
$M_\mathit{f}^{\pr\a}$ of $M^\pr$ such that
$\p(M^\alpha_\mathit{f})\subseteq M_\mathit{f}^{\pr\a}$.
Since $\xi_A$ is a natural transformation, we have
\eqn{finite lemma 1 eq 2}{
\xi_A^{M_\mathit{f}^{\pr\a}}\circ(\p\otimes\I_A)
=(-1)^{|\p||\xi_A|}(\p\otimes\I_A)\circ\xi_A^{M_\mathit{f}^\a}
:M_\mathit{f}^\a\otimes A\to M_\mathit{f}^{\pr\a}\otimes A.
}
Let $\a$ run over all finite dimensional subcomodules $M^\a_\mathit{f}$ of $M$ over $B$.
By applying $\displaystyle\varinjlim_\a$ on \eq{finite lemma 1 eq 2}, and composing the inclusion
$\displaystyle\varinjlim_\a\big(M^{\pr\a}_\mathit{f}\otimes A\big)\hookrightarrow M^\pr\otimes A$,
we get \eq{finite lemma 1 eq 1}.

Clearly the map $\operatorname{res}_A$ is well-defined. It follows from the definitions that
both
$\operatorname{res}_A$ and $\varinjlim_A$ are morphisms of dg-algebras, are natural  in $A$ and are inverse to each other.
\qed
\end{proof}

Let us define $Z^0\mathsf{End}(\bm{\o}_{\mathit{f}}\!\otimes\!A)$,
$\mathsf{End}_\otimes(\bm{\o}_{\mathit{f}}\!\otimes\!A)$,
$Z^0\mathsf{End}_\otimes(\bm{\o}_{\mathit{f}}\!\otimes\!A)$ and
$\mathit{ho}Z^0\mathsf{End}_\otimes(\bm{\o}_{\mathit{f}}\!\otimes\!A)$
as we did in the previous subsection.

\begin{lemma}\label{finite lemma 2}
The isomorphism in Lemma \ref{finite lemma 1} gives an isomorphism
\[
\xymatrix{
\text{$\varinjlim$}_A:Z^0\mathsf{End}_\otimes(\bm{\o}_{\mathit{f}}\!\otimes\!A)
\ar@<0.2pc>[r]&
\ar@<0.2pc>[l]
Z^0\mathsf{End}_\otimes(\bm{\o}\!\otimes\!A): \operatorname{res}_A
}
\]
of monoids that are natural in $A\in \category{cdgA}(\Bbbk)$.
In particular, $Z^0\mathsf{End}_\otimes(\bm{\o}_{\mathit{f}}\!\otimes\!A)$ is a group.
\end{lemma}
\begin{proof}
It suffices to check that $\varinjlim_A\left(Z^0\mathsf{End}_\otimes(\bm{\o}_{\mathit{f}}\!\otimes\!A)\right)$
is contained in $Z^0\mathsf{End}_\otimes(\bm{\o}\!\otimes\!A)$ and
$\operatorname{res}_A\left(Z^0\mathsf{End}_\otimes(\bm{\o}\!\otimes\!A)\right)$
is contained in $Z^0\mathsf{End}_\otimes(\bm{\o}_{\mathit{f}}\!\otimes\!A)$.
\begin{itemize}
\item
It is clear from the definitions that
$\operatorname{res}_A\left(Z^0\mathsf{End}_\otimes(\bm{\o}\!\otimes\!A)\right)
\subseteq Z^0\mathsf{End}_\otimes(\bm{\o}_{\mathit{f}}\!\otimes\!A)$.

\item
Let $\xi_A\in Z^0\mathsf{End}_\otimes(\bm{\o}_{\mathit{f}}\!\otimes\!A)$.
Clearly, we have $\varinjlim_A\xi_A\in Z^0\mathsf{End}(\bm{\o}_{\mathit{f}}\!\otimes\!A)$ and
$(\varinjlim_A\xi_A)^\Bbbk=\xi_A^\Bbbk=\I_{\Bbbk\otimes A}$.
Moreover, for right dg-comodules $(M,\g^M)$ and $(M^\pr,\g^{M^\pr})$ over $B$, we have
\eqn{finite lemma 2 eq 1}{
(\text{$\varinjlim$}_A\xi_A)^{M\otimes_{m_B}M^\pr}
=(\text{$\varinjlim$}_A\xi_A)^M\otimes_{m_A}(\text{$\varinjlim$}_A\xi_A)^{M^\pr}
:M\otimes M^\pr\otimes A\to M\otimes M^\pr\otimes A.
}
For the proof, let $M_\mathit{f}^\a$ and $M_\mathit{f}^{\pr\b}$
be finite dimensional subcomodules of $M$ and $M^\pr$ over $B$, respectively.
Then the tensor product $M_\mathit{f}^\a\otimes_{m_B} M_\mathit{f}^{\pr\b}$
is a finite dimensional subcomodule of $M\otimes_{m_B}M^\pr$ over $B$.
Since $\xi_A$ is a tensor natural transformation, we have
\eqn{finite lemma 2 eq 2}{
\xi_A^{M_\mathit{f}^\a\otimes_{m_B} M_\mathit{f}^{\pr\b}}
=\xi_A^{M_\mathit{f}^\a}\otimes_{m_A}\xi_A^{M_\mathit{f}^{\pr\b}}
:M_\mathit{f}^\a\otimes M_\mathit{f}^{\pr\b}\otimes A\to M_\mathit{f}^\a\otimes M_\mathit{f}^{\pr\b}\otimes A.
}
Note that $M_\mathit{f}^\a\otimes_{m_B} M_\mathit{f}^{\pr\b}$ covers $M\otimes_{m_B}M^\pr$
as the indices $\a$, $\b$ run all finite dimensional subcomodules of $M$, $M^\pr$ over $B$, respectively.
Thus we get \eq{finite lemma 2 eq 1} by applying the colimit
$\displaystyle\varinjlim_{\a,\b}$ on \eq{finite lemma 2 eq 2}.
\qed
\end{itemize}
\end{proof}

\begin{lemma}\label{finite lemma 3}
The isomorphism in Lemma \ref{finite lemma 2} induces an isomorphism
\[
\xymatrix{
\bm{\text{$\varinjlim$}}_A:\mathit{ho}Z^0\mathsf{End}_\otimes(\bm{\o}_{\mathit{f}}\!\otimes\!A)
\ar@<0.2pc>[r]&
\ar@<0.2pc>[l]
\mathit{ho}Z^0\mathsf{End}_\otimes(\bm{\o}\!\otimes\!A): \bm{\operatorname{res}}_A
}
\]
of monoids that is natural in $A\in \mathit{ho}\category{cdgA}(\Bbbk)$, where we define
\[
\bm{\text{$\varinjlim$}}_A[\xi_A]:=[\text{$\varinjlim$}_A\xi_A]
\quad\text{and}\quad
\bm{\operatorname{res}}_A[\eta_A]:=[\operatorname{res}_A\eta_A]
\]
for $[\xi_A]\in\mathit{ho}Z^0\mathsf{End}_\otimes(\bm{\o}_{\mathit{f}}\!\otimes\!A)$
and $[\eta_A]\in\mathit{ho}Z^0\mathsf{End}_\otimes(\bm{\o}\!\otimes\!A)$.
In particular, $\mathit{ho}Z^0\mathsf{End}_\otimes(\bm{\o}_{\mathit{f}}\!\otimes\!A)$ is a group.
\end{lemma}
\begin{proof}
It suffices to show that the maps $\varinjlim_A$ and $\operatorname{res}_A$ sends homotopy pairs on each side to the other.
Clearly, $\operatorname{res}_A$ sends homotopy pairs on $Z^0\mathsf{End}_\otimes(\bm{\o}\!\otimes\!A)$
to homotopy pairs on $Z^0\mathsf{End}_\otimes(\bm{\o}_{\mathit{f}}\!\otimes\!A)$. Moreover,
$\varinjlim_A$ sends homotopy pairs on $Z^0\mathsf{End}_\otimes(\bm{\o}_{\mathit{f}}\!\otimes\!A)$
to homotopy pairs on $Z^0\mathsf{End}_\otimes(\bm{\o}\!\otimes\!A)$.
This can be checked by taking colimits, analogous to the arguments introduced in Lemma \ref{finite lemma 1} and Lemma \ref{finite lemma 2}.
\qed
\end{proof}

\begin{theorem}
We have functors
\[
\begin{aligned}
\bm{\CE}^{\bm{\o}_\mathit{f}}:\category{cdgA}(\Bbbk)\rightsquigarrow\category{dgA}(\Bbbk)
,\qquad
\bm{\CG}^{\bm{\o}_\mathit{f}}:\category{cdgA}(\Bbbk)\rightsquigarrow\category{Grp}
\quad\text{and}\quad
\bm{\mG}_\otimes^{\!\bm{\o}_\mathit{f}}:\mathit{ho}\category{cdgA}(\Bbbk)\rightsquigarrow\category{Grp},
\end{aligned}
\]
sending each cdg-algebra $A$ to
\[
\bm{\CE}^{\bm{\o}_\mathit{f}}(A):=\mathsf{End}(\bm{\o}_{\mathit{f}}\!\otimes\!A)
,\quad
\bm{\CG}^{\bm{\o}_\mathit{f}}(A):=Z^0\mathsf{End}_\otimes(\bm{\o}_{\mathit{f}}\!\otimes\!A)
\quad\text{and}\quad
\bm{\mG}_\otimes^{\!\bm{\o}_\mathit{f}}(A):=\mathit{ho}Z^0\mathsf{End}_\otimes(\bm{\o}_{\mathit{f}}\!\otimes\!A).
\]
Moreover, we have natural isomorphisms
$\bm{\CE}^{\bm{\o}_\mathit{f}}\cong\bm{\CE}^{\bm{\o}}$,
$\bm{\CG}_\otimes^{\bm{\o}_\mathit{f}}\cong \bm{\CG}^{\bm{\o}}_{\!\otimes}$ and
$\bm{\mG}_\otimes^{\!\bm{\o}_\mathit{f}}\cong \bm{\mG}^{\bm{\o}}_{\!\otimes}$.
\end{theorem}
\begin{proof}
By Lemma \ref{finite lemma 1}, we have a natural isomorphism
$\varinjlim:\bm{\CE}^{\bm{\o}_\mathit{f}}\Rightarrow\bm{\CE}^{\bm{\o}}$
whose component at a cdg-algebra $A$ is $\varinjlim_A:\bm{\CE}^{\bm{\o}_\mathit{f}}(A)\to\bm{\CE}^{\bm{\o}}(A)$,
with its inverse $\operatorname{res}:\bm{\CE}^{\bm{\o}}\Rightarrow\bm{\CE}^{\bm{\o}_\mathit{f}}$
defined by $\operatorname{res}:=\{\operatorname{res}_A:
\bm{\CE}^{\bm{\o}}(A)\to\bm{\CE}^{\bm{\o}_\mathit{f}}(A)\}$.
Lemma \ref{finite lemma 2} implies that
$\varinjlim$ and $\operatorname{res}$ induce a natural isomorphism
$\bm{\CG}_\otimes^{\bm{\o}_\mathit{f}}\cong \bm{\CG}^{\bm{\o}}_{\!\otimes}$, and
 Lemma \ref{finite lemma 3} implies that
$\bm{\varinjlim}$ and $\bm{\operatorname{res}}$ induce a natural isomorphism
$\bm{\mG}_\otimes^{\!\bm{\o}_\mathit{f}}\cong \bm{\mG}^{\bm{\o}}_{\!\otimes}$.
\qed
\end{proof}

\subsection{Remark on rigidity}

A more conventional way of recovering the antipode $\vs_B$ of a cdg-Hopf algebra $B$ from the category of finite dimensional dg-comodules over $B$
may be by considering dg-versions of the rigidity, as introduced in \cite{Deligne90,Rivano}---see also  \cite{DM}, \cite{EGNO} and \cite{Szamuely}.

In this subsection, we provide an independent proof that
$\bm{\CG}_\otimes^{\bm{\o}_\mathit{f}}(A)$ is a group for every cdg-algebra $A$
using the dg-version of rigidity in a dg-tensor category.
At the end of this subsection, we will directly show that the inverse $S(\xi_A)$ of an element
$\xi_A\in\bm{\CG}_\otimes^{\bm{\o}_\mathit{f}}(A)$
defined in this way agrees with our previous inverse $\vs(\eta_A)$ of $\eta_A\in\bm{\CG}^{\bm{\o}}_{\!\otimes}(A)$
introduced in Proposition \ref{group valued},
via the isomorphism in Lemma \ref{finite lemma 1}.
This implies that our Tannakian reconstruction, restricted to commutative Hopf algebras and associated affine
group schemes, an alternative to but agrees with the well-known reconstruction theorem \cite{Deligne90,Rivano} from the
category of finite dimensional linear representations. 

In our previous paper \cite{JLJSP}, we showed that our method of Tannakian reconstruction
also works in the categorial dual version to affine group dg-schemes,
where we do not have the analogous property of restricting to the finite dimensional linear representations.

In this subsection we adapt the Einstein summation convention.

Let $(M,\g^M)$ be a finite dimensional right dg-comodule over $B$. 
Using the antipode $\vs_B$, we can define a right dg-comodule structure on the dual cochain complex
$M^\vee:=\Hom(M,\Bbbk)$ as follows.
Consider the cochain isomorphism
\[
\Hom(M,B)\cong\Hom(M,\Bbbk)\otimes B=M^\vee\otimes B.
\]
We can explicitly describe the above isomorphism by taking a basis of $M$.
Fix a basis $\{x_i\}$ of $M$ over $\Bbbk$ and denote $\{e^i\}$ as the corresponding dual basis of $M^\vee$.
We will write $(-1)^{|i|}=(-1)^{|x_i|}=(-1)^{|e^i|}$.
Then the above isomorphism sends a linear map $l:M\to B$ to 
$(-1)^{|i|(|l|+|i|)}e^i\otimes l(x_i)$.
Moreover, the isomorphism is independent of the choice of the basis.

Define a cochain map
$\overline{\g}^{M^\vee}:M^\vee=\Hom(M,\Bbbk)\to\Hom(M,B)$,
which sends a linear map $h:M\to \Bbbk$ to
$\overline{\g}^{M^\vee}(h):=\imath_B\circ(h\otimes \vs_B)\circ\g^M:M\to B$.
Composing it with the above isomorphism, we get a cochain map
\[
\g^{M^\vee}:M^\vee\to M^\vee\otimes B.
\]
One can show that $(M^\vee,\g^{M^\vee})$ is a right dg-comodule over $B$,
using the fact that the antipode $\vs_B:B\to B$ is an anti-morphism of dg-coalgebras.
Consider the \emph{evaluation} map, denoted by $\operatorname{ev}_M$,
and the \emph{coevaluation} map, denoted by $\operatorname{cv}_M$,
of the underlying cochain complex $M$ defined as follows:
\[
\begin{aligned}
\operatorname{ev}_M:M^\vee\otimes M&\to \Bbbk,&
\operatorname{cv}_M:\Bbbk&\to M\otimes M^\vee,\qquad\\
h\otimes m&\mapsto h(m),&
1&\mapsto x_i\otimes e^i,
\end{aligned}
\]
which are cochain maps and 
make the following diagrams commute:
\[
\xymatrixcolsep{3.5pc}
\xymatrix{
M \ar[r]^-{\jmath^{-1}_M} \ar[d]_-{\imath^{-1}_M}&
M\otimes \Bbbk \\
\Bbbk\otimes M \ar[r]^-{\operatorname{cv}_M\otimes \I_M}&
M\otimes M^\vee\otimes M \ar[u]_-{\I_M\otimes \operatorname{ev}_M}
}
\qquad\qquad
\xymatrix{
M^\vee \ar[r]^-{\imath^{-1}_{M^\vee}} \ar[d]_-{\jmath^{-1}_{M^\vee}}&
\Bbbk \otimes M^\vee\\
M^\vee\otimes \Bbbk \ar[r]^-{\I_{M^\vee}\otimes\operatorname{cv}_M}&
M^\vee\otimes M \otimes M^\vee \ar[u]_-{\operatorname{ev}_M\otimes\I_{M^\vee}}
},
\]
i.e.
\eqn{ev and coev relation}{
(\I_M\otimes\operatorname{ev}_M)\circ(\operatorname{cv}_M\otimes\I_M)\circ\imath^{-1}_M
=\jmath^{-1}_M,
\quad\text{and}\quad
(\operatorname{ev}_M\otimes\I_{M^\vee})\circ(\I_{M^\vee}\otimes\operatorname{cv}_M)\circ\jmath^{-1}_{M^\vee}
=\imath^{-1}_{M^\vee}.
}

Using the antipode axiom of $\vs_B$, we can check that $\operatorname{ev}_M$ and $\operatorname{cv}_M$ are morphisms
of right dg-comodules over $B$:
\eqalign{
\xymatrixcolsep{3pc}
\xymatrix{(M^\vee,\g^{M^\vee})\otimes_{m_B}(M,\g^M)\ar[r]^-{\operatorname{ev}_M}& (\Bbbk,\g^\Bbbk),}
\qquad
\xymatrix{(\Bbbk,\g^\Bbbk) \ar[r]^-{\operatorname{cv}_M}& (M,\g^M)\otimes_{m_B}(M^\vee,\g^{M^\vee}).}
}
We call the triple $\big((M^\vee,\g^{M^\vee}),\operatorname{ev}_M,\operatorname{cv}_M\big)$ the \emph{dual} of $(M,\g^M)$
in the dg-tensor category $\big(\dgcat{dgComod}_R(B)_{\mathit{f}},\otimes_{m_B}, (\Bbbk,\g^{\Bbbk})\big)$.
For convenience, we will often write  the dual of $(M,\g^M)$ as $(M^\vee,\g^{M^\vee})$.
Let us fix the evaluation map $\operatorname{ev}_M$ for each right dg-comodule $(M,\g^M)$ over $B$.
Then the dual $(M^\vee,\g^{M^\vee})$ of $(M,\g^M)$ is unique up to a unique isomorphism, since it is the representing object of the dg-functor
\[
\HOM_{\dgcat{dgComod}_R(B)_{\mathit{f}}}(-\otimes M,\Bbbk):\mathring{\dgcat{dgComod}_R(B)_{\mathit{f}}}\rightsquigarrow \dgcat{CoCh}(\Bbbk)_{\!f}.
\]
This is analogous to \cite[Lemma 6.3.2]{Szamuely} for the non dg-version.
For every morphism $\p:(M,\g^M)\to (M^\pr,\g^{M^\pr})$ of right dg-comodules over $B$, we have a morphism $\p^\vee:(M^{\pr\vee},\g^{M^{\pr\vee}})\to(M^\vee,\g^{M^\vee})$ of right dg-comodules over $B$, called the \emph{dual} of $\p$, which is defined as follows:
\[
\p^\vee:=\imath_{M^\vee}
\circ(\operatorname{ev}_{M^\pr}\otimes\I_{M^\vee})
\circ(\I_{M^{\pr\vee}}\otimes \p\otimes \I_{M^\vee})
\circ(\I_{M^{\pr\vee}}\otimes\operatorname{cv}_M)
\circ\jmath^{-1}_{M^{\pr\vee}}
:M^{\pr\vee}\to M^\vee.
\]
Equivalently, the dual $\p^\vee$ of $\p$ is the unique morphism of right dg-comodules over $B$ making the following diagram commutative:
\eqn{characterization of dual morphism}{
\xymatrixrowsep{1.3pc}
\xymatrix{
M^{\pr\vee}\otimes M \ar[r]^-{\p^\vee\otimes \I_M} \ar[d]_-{\I_{M^{\pr\vee}}\otimes \p}&
M^\vee\otimes M \ar[d]^-{\operatorname{ev}_M}\\
M^{\pr\vee}\otimes M^\pr \ar[r]^-{\operatorname{ev}_{M^\pr}}&
\Bbbk
},
\quad\hbox{i.e.}\quad
\operatorname{ev}_{M^\pr}\circ(\I_{M^{\pr\vee}}\otimes \p)=\operatorname{ev}_M\circ(\p^\vee\otimes \I_M).
}
We can also check that 
$(\p^\pr\circ\p)^\vee=\p^\vee\circ\p^{\pr\vee}$
for every pair of  morphisms $\p:(M,\g^M)\to (M^\pr,\g^{M^\pr})$ and $\p^\pr:(M^\pr,\g^{M^\pr})\to (M^{\pr\pr},\g^{M^{\pr\pr}})$
of right dg-comodules over $B$.
%

Remind that the tensor product $\otimes_{m_B}$ is symmetric.
Indeed, for right dg-comodules $(M,\g^M)$ and $(M^\pr,\g^{M^\pr})$ over $B$, the usual isomorphism
$\t:M\otimes M^\pr\cong M^\pr\otimes M$
of the underlying cochain complexes becomes an isomorphism
\[
\t:(M,\g^M)\otimes_{m_B}(M^\pr,\g^{M^\pr})\cong(M^\pr,\g^{M^\pr})\otimes_{m_B}(M,\g^M)
\]
of right dg-comodules over $B$, since $m_B$ is commutative.
Then the dual of $(M^\vee,\g^{M^\vee})$ is $(M,\g^M)$, with the following evaluation and coevaluation maps:
\[
\begin{aligned}
\operatorname{ev}_{M^\vee}:=\operatorname{ev}_{M}\circ\tau:M\otimes M^\vee&\to \Bbbk,&
\operatorname{cv}_{M^\vee}:=\t\circ\operatorname{cv}_M:\Bbbk&\to M^\vee\otimes M,
\\
m\otimes h&\mapsto (-1)^{|m||h|}h(m),&
1&\mapsto (-1)^{|i|}e^i\otimes x_i.
\end{aligned}
\]
Moreover, for any morphism $\p:(M,\g^M)\to(M^\pr,\g^{M^\pr})$ of right dg-comodules over $B$, we have $\p^{\vee\vee}=\p$
which follows from applying $\t$ on \eq{characterization of dual morphism}.

We will adopt the following conventions and notations for the upcoming two lemmas to save spaces:
\begin{itemize}
\item We will omit the tensor product $\otimes$ over $\Bbbk$.
\item We will write $\otimes_{m_A}=\otimes$. So every unadorned tensor product is over $A$.
\item We will omit the isomorphisms $(-)\otimes_{m_A}A\cong(-)\cong A\otimes_{m_A}(-)$ over $A$.
\item We will write $\bm{\o}_\mathit{f}\otimes A=\bm{\o}_A$.
\end{itemize}

\begin{lemma} \label{rigidity lemma 1}
For every cdg-algebra $A$ and $\xi_A\in \mathsf{End}(\bm{\o}_A)$,
define the dual $(\xi_A^M)^\vee$ of the component $\xi_A^M$ of $\xi_A$
at a finite dimensional right dg-comodule $(M,\g^M)$ over $B$ as
\[
(\xi_A^M)^\vee:=
\big(\bm{\o}_A(\operatorname{ev}_M)\otimes\I_{M^\vee A}\big)
\circ\big(\I_{M^\vee A}\otimes\xi_A^M\otimes\I_{M^\vee A}\big)
\circ\big(\I_{M^\vee A}\otimes\bm{\o}_A(\operatorname{cv}_M)\big)
:M^\vee A\to M^\vee A.
\]
Then
\begin{enumerate}[label=({\alph*})]
\item
the following diagram commutes:
\[
\xymatrixrowsep{2pc}
\xymatrixcolsep{4pc}
\xymatrix{
M^\vee A\otimes MA \ar[r]^-{(\xi_A^M)^\vee\otimes\I_{MA}} \ar[d]_-{\I_{M^\vee A}\otimes\xi_A^M}&
M^\vee A\otimes MA \ar[d]^-{\bm{\o}_A(\operatorname{ev}_M)}\\
M^\vee A\otimes MA \ar[r]^-{\bm{\o}_A(\operatorname{ev}_M)}&
A
}.
\]

\item
for any morphism $\p:(M,\g^M)\to(M^\pr,\g^{M^\pr})$ of finite dimensional right dg-comodules over $B$, the following diagram commutes:
\[
\xymatrixrowsep{1.6pc}
\xymatrixcolsep{3.5pc}
\xymatrix{
M^{\pr\vee}A \ar[r]^-{(\xi_A^{M^\pr})^\vee} \ar[d]_-{\bm{\o}_A(\p^\vee)}&
M^{\pr\vee}A \ar[d]^-{\bm{\o}_A(\p^\vee)}\\
M^\vee A \ar[r]^-{(\xi_A^M)^\vee}&
M^\vee A
},
\quad\text{i.e.,}\quad
\bm{\o}_A(\p^\vee)\circ(\xi_A^{M^\pr})^\vee
=(-1)^{|\xi_A||\p|}(\xi_A^M)^\vee\circ\bm{\o}_A(\p^\vee).
\]
\end{enumerate}
\end{lemma}
\begin{proof}
We can show $(a)$ as follows:
\[
\begin{aligned}
\bm{\o}_A&(\operatorname{ev}_M)\circ\big((\xi_A^M)^\vee\otimes\I_{MA}\big)\\
&=
\big(\bm{\o}_A(\operatorname{ev}_M)\otimes\bm{\o}_A(\operatorname{ev}_M)\big)
\circ\big(\I_{M^\vee A}\otimes\xi_A^M\otimes\I_{M^\vee A\otimes MA}\big)
\circ\big(\I_{M^\vee A}\otimes\bm{\o}_A(\operatorname{cv}_M)\otimes\I_{MA}\big)\\
&=
\bm{\o}_A(\operatorname{ev}_M)
\circ\big(\I_{M^\vee A}\otimes \xi_A^M\big)
\circ\big(\I_{M^\vee A\otimes MA}\otimes\bm{\o}_A(\operatorname{ev}_M)\big)
\circ\big(\I_{M^\vee A}\otimes\bm{\o}_A(\operatorname{cv}_M)\otimes\I_{MA}\big)\\
&=
\bm{\o}_A(\operatorname{ev}_M)\circ\big(\I_{M^\vee A}\otimes\xi_A^M\big).
\end{aligned}
\]

Next, we show $(b)$.
Since $\p:(M,\g^M)\to(M^\pr,\g^{M^\pr})$
is a morphism of finite dimensional right dg-comodules over $B$ and $\xi_A$ is a natural transformation, we have
\eqn{rigidity lemma 1 eq 1}{
\xi_A^{M^\pr}\circ\bm{\o}_A(\p)
=(-1)^{|\p||\xi_A|}\bm{\o}_A(\p)\circ\xi_A^M.
}
We first show that
\[
\bm{\o}_A(\p^\vee)\circ(\xi_A^{M^\pr})^\vee
=
\big(\bm{\o}_A(\operatorname{ev}_{M^\pr})\otimes\I_{M^\vee A}\big)
\circ\big(\I_{M^{\pr\vee}A}\otimes(\xi_A^{M^\pr}\circ\bm{\o}_A(\p))\otimes \I_{M^\vee A}\big)
\circ(\I_{M^{\pr\vee} A}\otimes\bm{\o}_A(\operatorname{cv}_M)),
\]
which is essentially computing the dual of $\xi_A^{M^\pr}\circ\bm{\o}_A(\p)$:
\[
\begin{aligned}
\big(&\bm{\o}_A(\operatorname{ev}_{M^\pr})\otimes\I_{M^\vee A}\big)
\circ\big(\I_{M^{\pr\vee}A}\otimes(\xi_A^{M^\pr}\circ\bm{\o}_A(\p))\otimes \I_{M^\vee A}\big)
\circ(\I_{M^{\pr\vee} A}\otimes\bm{\o}_A(\operatorname{cv}_M))\\
&=
\big(\bm{\o}_A(\operatorname{ev}_{M^\pr})\otimes\I_{M^\vee A}\big)
\circ\big(\I_{M^{\pr\vee}A}\otimes\xi_A^{M^\pr}\otimes \I_{M^\vee A}\big)
\circ\big(\I_{M^{\pr\vee}A}\otimes\bm{\o}_A(\p)\otimes\I_{M^\vee A}\big)
\circ(\I_{M^{\pr\vee} A}\otimes\bm{\o}_A(\operatorname{cv}_M))\\
&=
\big(\bm{\o}_A(\operatorname{ev}_{M^\pr})\otimes\I_{M^\vee A}\big)
\circ\big((\xi_A^{M^\pr})^\vee\otimes\I_{M^\pr A\otimes M^\vee A}\big)
\circ\big(\I_{M^{\pr\vee}A}\otimes\bm{\o}_A(\p)\otimes\I_{M^\vee A}\big)
\circ(\I_{M^{\pr\vee} A}\otimes\bm{\o}_A(\operatorname{cv}_M))\\
&=
\big(\bm{\o}_A(\operatorname{ev}_{M^\pr})\otimes\I_{M^\vee A}\big)
\circ\big(\I_{M^{\pr\vee} A}\otimes\bm{\o}_A(\p)\otimes\I_{M^\vee A}\big)
\circ\big((\xi_A^{M^\pr})^\vee\otimes\I_{MA\otimes M^\vee A}\big)
\circ(\I_{M^{\pr\vee} A}\otimes\bm{\o}_A(\operatorname{cv}_M))\\
&=
\big(\bm{\o}_A(\operatorname{ev}_M)\otimes\I_{M^\vee A}\big)
\circ\big(\bm{\o}_A(\p^\vee)\otimes\I_{MA\otimes M^\vee A}\big)
\circ\big((\xi_A^{M^\pr})^\vee\otimes\I_{MA\otimes M^\vee A}\big)
\circ(\I_{M^{\pr\vee} A}\otimes\bm{\o}_A(\operatorname{cv}_M))\\
&=
\big(\bm{\o}_A(\operatorname{ev}_M)\otimes\I_{M^\vee A}\big)
\circ\big(\I_{M^\vee A}\otimes\bm{\o}_A(\operatorname{cv}_M)\big)
\circ\bm{\o}_A(\p^\vee)\circ(\xi_A^{M^\pr})^\vee\\
&=
\bm{\o}_A(\p^\vee)\circ(\xi_A^{M^\pr})^\vee.
\end{aligned}
\]
We used $(a)$ on the $2$nd equality, the relation \eq{characterization of dual morphism} on the $4$th equality and the relation
\eq{ev and coev relation} on the last equality. Note that $\bm{\o}_A$ is a dg-tensor functor. Similarly, we have
\[
(\xi_A^M)^\vee\circ\bm{\o}_A(\p^\vee)
=
\big(\bm{\o}_A(\operatorname{ev}_{M^\pr})\otimes\I_{M^\vee A}\big)
\circ\big(\I_{M^{\pr\vee}A}\otimes(\bm{\o}_A(\p)\circ\xi_A^M)\otimes \I_{M^\vee A}\big)
\circ(\I_{M^{\pr\vee} A}\otimes\bm{\o}_A(\operatorname{cv}_M)).
\]
From \eq{rigidity lemma 1 eq 1}, together with the above calculations, we conclude that
$\bm{\o}_A(\p^\vee)\circ(\xi_A^{M^\pr})^\vee
=(-1)^{|\xi_A||\p|}(\xi_A^M)^\vee\circ\bm{\o}_A(\p^\vee)$.
\qed
\end{proof}

\begin{lemma}\label{rigidity lemma 2}
For every cdg-algebra $A$ and $\xi_A\in \mathsf{End}(\bm{\o}_A)$
we have another natural transformation $S(\xi_A)\in\mathsf{End}(\bm{\o}_A)$,
whose component at a finite dimensional right dg-comodule $(M,\g^M)$ over $B$ is given by
\[
S(\xi_A)^M:=(\xi_A^{M^\vee})^\vee:MA\to MA,
\]
such that $S(\xi_A)\in Z^0\mathsf{End}_{\otimes}(\bm{\o}_A)$
whenever $\xi_A\in Z^0\mathsf{End}_{\otimes}(\bm{\o}_A)$.
Moreover, we have a group
$\bm{\CG}_\otimes^{\bm{\o}_\mathit{f}}(A)
=\big(Z^0\mathsf{End}_{\otimes}(\bm{\o}_A), \I_{\bm{\o}_A}, \circ\big)$,
where the inverse of an element $\xi_A\in Z^0\mathsf{End}_{\otimes}(\bm{\o}_A)$ 
is $S(\xi_A)$.  
\end{lemma}

Remark that the above lemma, when $A=\Bbbk$ and the degree of everything is concentrated to zero,
follows from \cite[Prop. 5.3.1]{EGNO}.

\begin{proof}
Given $\xi_A\in \mathsf{End}(\bm{\o}_A)$,
we first show that $S(\xi_A)$ is a natural transformation.
Let $\p:(M,\g^M)\to(M^\pr,\g^{M^\pr})$ be a morphism of finite dimensional right dg-comodules over $B$.
Substituting $\p$ in Lemma \ref{rigidity lemma 1}\emph{(b)}
by the dual $\p^\vee:(M^{\pr\vee},\g^{M^\pr\vee})\to(M^\vee,\g^{M^\vee})$
and using $\p^{\vee\vee}=\p$, we obtain that
$\bm{\o}_A(\p)\circ S(\xi_A)^M
=(-1)^{|\xi_A||\p|}S(\xi_A)^{M^\pr}\circ\bm{\o}_A(\p)$.
Therefore $S(\xi_A)$ is a natural transformation of degree $|\xi_A|$.

Suppose further that $\xi_A$ is in $Z^0\mathsf{End}_{\otimes}(\bm{\o}_A)$.
Then, for every finite dimensional right dg-comodule $(M,\g^M)$ over $B$ 
we have $S(\xi_A)^M\circ \xi_A^M=\xi_A^M\circ S(\xi_A)^M=\I_{MA}$,
which can be checked by the following two commutative diagrams:
\[
\xymatrixcolsep{4.2pc}
\xymatrix{
MA \ar[r]^-{\I_{MA}\otimes\bm{\o}_A(\operatorname{cv}_{M^\vee})} \ar@{=}[d]&
MA\!\otimes\! M^\vee A\!\otimes\! MA \ar[r]^-{\I_{MA}\otimes\xi_A^{M^\vee}\otimes\I_{MA}}
\ar[d]^-{\I_{MA}\otimes\xi_A^{M^\vee}\otimes\xi_A^M}&
MA\!\otimes\! M^\vee A\!\otimes\! MA \ar[r]^-{\bm{\o}(\operatorname{ev}_{M^\vee})\otimes\I_{MA}}
\ar[d]^-{\I_{MA\otimes M^\vee A}\otimes\xi_A^M}&
MA \ar[d]^-{\xi_A^M}\\
MA \ar[r]^-{\I_{MA}\!\otimes\!\bm{\o}_A(\operatorname{cv}_{M^\vee})}&
MA\otimes M^\vee A\!\otimes\! MA \ar@{=}[r]&
MA\otimes M^\vee A\!\otimes\! MA \ar[r]^-{\bm{\o}(\operatorname{ev}_{M^\vee})\otimes\I_{MA}}&
MA
}
\]
and
\[
\xymatrixcolsep{4.2pc}
\xymatrix{
MA \ar[r]^-{\I_{MA}\otimes\bm{\o}_A(\operatorname{cv}_{M^\vee})}
\ar[d]^-{\xi_A^M}&
MA\!\otimes\! M^\vee A\!\otimes\! MA \ar@{=}[r] \ar[d]^-{\xi_A^M\otimes\I_{M^\vee A\otimes MA}}&
MA\!\otimes\! M^\vee A\!\otimes\! MA \ar[r]^-{\bm{\o}(\operatorname{ev}_{M^\vee})\otimes\I_{MA}}
\ar[d]^-{\xi_A^M\otimes \xi_A^{M^\vee}\otimes\I_{MA}}&
MA \ar@{=}[d]\\
MA \ar[r]^-{\I_{MA}\!\otimes\!\bm{\o}_A(\operatorname{cv}_{M^\vee})}&
MA\otimes M^\vee A\!\otimes\! MA \ar[r]^-{\I_{MA}\otimes\xi_A^{M^\vee}\otimes\I_{MA}}&
MA\otimes M^\vee A\!\otimes\! MA \ar[r]^-{\bm{\o}(\operatorname{ev}_{M^\vee})\otimes\I_{MA}}&
MA
}.
\]
The top horizontal composition in the $1$st diagram and the bottom horizontal composition 
in the $2$nd diagram are equal to $S(\xi_A)^M$, and we used the relations in \eq{ev and coev relation} twice.
Since $\xi_A$ is a tensor natural transformation,
the left square on the $1$st diagram and the right square on the $2$nd diagram commute.
Moreover, we have $S(\xi_A)\in Z^0\mathsf{End}_{\otimes}(\bm{\o}_A)$ since $S(\xi_A)$ is the inverse of $\xi_A$.
\qed
\end{proof}

We return to our original notations. We can express the component $S(\xi_A)^M$ of $S(\xi_A)$ introduced in Lemma \ref{rigidity lemma 2} as follows:
\[
\begin{aligned}
S(\xi_A)^M=(\xi_A^{M^\vee})^\vee
=&\t\circ\imath_{A\otimes M}\circ(\operatorname{ev}_{M^\vee}\otimes\I_{A\otimes M})\circ(\I_M\otimes\xi^{M^\vee}_A\otimes\I_M)\\
&\circ(\I_{M\otimes M^\vee}\otimes\tau)\circ(\I_M\otimes\operatorname{cv}_{M^\vee}\otimes\I_A)\circ(\jmath^{-1}_M\otimes\I_A)
:M\otimes A\to M\otimes A.
\end{aligned}
\]

We end this subsection with the following lemma.

\begin{lemma}
For every cdg-algebra $A$ and $\eta_A\in\mathsf{End}(\bm{\o}{\otimes}A)$, we have
\[
\vs(\eta_A)=\text{$\varinjlim$}_AS(\operatorname{res}_A\eta_A).
\]
\end{lemma}

\begin{proof}
It suffices to show that the components of $\vs(\eta_A)$ and $S(\operatorname{res}_A\eta_A)$ 
at a finite dimensional right dg-comodule $(M,\g^M)$ over $B$ agree.
From Proposition \ref{group valued},
the component of $\vs(\eta_A)\in\mathsf{End}(\bm{\o}\otimes A)$ at a right dg-comodule $(M,\g^M)$ over $B$ is
$\vs(\eta_A)^M=(\jmath_M\otimes\I_A)\circ(\I_M\otimes \ep_B\otimes \I_A)\circ(\I_M\otimes \eta^{B^*}_A)\circ(\g^M\otimes \I_A)$.

Fix a basis $\{a_t\}$ of $A$ and  a basis $\{x_i\}$ of $M$ as $\Z$-graded vector spaces over $\Bbbk$.
Let  $\{e^i\}$ be the basis of $M^\vee$ dual to $\{x_i\}$.
Again, we will write $(-1)^{|i|}=(-1)^{|x_i|}=(-1)^{|e^i|}$.
Note that the index $i$ runs finite.
For each $x_i$, we have $\g^M(x_i)=x_j\otimes \g_i^j$
for some $\g_i^j\in B$. Fix elements $a\in A$ and $m\in M$,
and denote $m^i:=e^i(m)\in \Bbbk$ which is zero if $|m|\neq|i|$.
We denote $\cdot$ for the product of an element in $\Bbbk$.

1.
We calculate the value of $\vs(\eta_A)^M\big(m\otimes a\big)$.
To do so, we need to describe $\eta_A^{B^*}$ first.
Using the isomorphism in Lemma \ref{ctanha}, we get
\[
\eta_A^{B^*}=(\jmath_B\otimes\I_A)\circ(\I_B\otimes \ep_B\otimes \I_A)\circ(\I_B\otimes \eta^{B}_A)\circ(\g^{B^*}\otimes\I_A)
\]
where we defined $\g^{B^*}:=\tau\circ(\vs_B\otimes \I_B)\circ\cp_B$. Therefore we obtain that
\[
\begin{aligned}
\vs(\eta_A)^M
:=&(\jmath_M\otimes\I_A)\circ(\I_M\otimes \ep_B\otimes \I_A)\circ(\I_M\otimes \eta^{B^*}_A)\circ(\g^M\otimes \I_A)\\
=&(\jmath_M\otimes \imath_A)\circ(\I_M\otimes \ep_B\otimes \ep_B\otimes A)\circ(\I_{M\otimes B}\otimes \eta^B_A)\\
&\circ(\I_M\otimes \t\otimes \I_A)\circ(\I_M\otimes \vs_B\otimes \I_{B\otimes A})\circ(\I_M\otimes \cp_B\otimes \I_A)\circ(\g^M\otimes \I_A)\\
=&(\jmath_M\otimes \I_A)\circ(\I_M\otimes \ep_B\otimes \I_A)\circ(\I_M\otimes \eta^B_A)\circ(\I_M\otimes \vs_B\otimes \I_A)\circ(\g^M\otimes \I_A).
\end{aligned}
\]
We are ready to calculate $\vs(\eta_A)^M\big(m\otimes a\big)$.
Let us write $m=m^k\cdot x_k$ for the indices $k$ with $|k|=|m|$. Then we have
\[
\begin{aligned}
\vs(\eta_A)^M & \big(m\otimes a\big)\\
&=
(\jmath_M\otimes \I_A)\circ(\I_M\otimes \ep_B\otimes \I_A)\circ(\I_M\otimes \eta^B_A)\circ(\I_M\otimes \vs_B\otimes \I_A)\circ(\g^M\otimes \I_A)
\big(m\otimes a\big)\\
&=
m^k\cdot
(\jmath_M\otimes \I_A)\circ(\I_M\otimes \ep_B\otimes \I_A)\circ(\I_M\otimes \eta^B_A)\circ(\I_M\otimes \vs_B\otimes \I_A)
\big(x_i\otimes \g^i_k\otimes a\big)\\
&=
m^k\cdot
(\jmath_M\otimes \I_A)\circ(\I_M\otimes \ep_B\otimes \I_A)\circ(\I_M\otimes \eta^B_A)
\big(x_i\otimes \vs_B(\g^i_k)\otimes a\big)\\
&=
(-1)^{|i||\eta_A|}m^k\cdot
(\jmath_M\otimes \I_A)\circ(\I_M\otimes \ep_B\otimes \I_A)
\Big(x_i\otimes \eta_A^B\big(\vs_B(\g^i_k)\otimes a\big)\Big)\\
&=
(-1)^{|i||\eta_A|}m^k\cdot
(\jmath_M\otimes \I_A)\circ(\I_M\otimes \ep_B\otimes \I_A)
\big(x_i\otimes b^{i t}_k\otimes a_t\big)\\
&=
(-1)^{|i||\eta_A|}m^k\ep_B(b^{i t}_k)\cdot \big(x_i\otimes a_t\big)
\end{aligned}
\]
where $b^{i t}_k\in B$ is defined by the relation $\eta^B_A\big(\vs_B(\g^i_k)\otimes a\big)= b^{i t}_k\otimes a_t$.

2.
We calculate the value of $S(\operatorname{res}_A\eta_A)^M\big(m\otimes a\big)$.
First, note that
$\overline{\g}^{M^\vee}(e^i)=\imath_B\circ(e^i\otimes \vs_B)\circ\g^M:M\to B$
is of degree $|i|$ and
\[
\overline{\g}^{M^\vee}(e^i)\big(x_k\big)
=\imath_B\circ(e^i\otimes \vs_B)\circ\g^M\big(x_k\big)
=\imath_B\circ(e^i\otimes \vs_B)\big(x_j\otimes \g^j_k\big)
=\vs_B(\g^i_k).
\]
Therefore we have
\[
\g^{M^\vee}(e^i)
=(-1)^{|k|(|k|+|i|)}\cdot
e^k\otimes \overline{\g}^{M^\vee}(e^i)\big(x_k\big)
=(-1)^{|k|+|i||k|}\cdot
e^k\otimes\vs_B(\g^i_k).
\]
Next, we calculate $\eta_A^{M^\vee}(e^i\otimes a)$:
\[
\begin{aligned}
\eta_A^{M^\vee}(e^i\otimes a)
&=
(\I_{M^\vee}\otimes \imath_A)\circ(\I_{M^\vee}\otimes\ep_B\otimes\I_A)\circ(\I_{M^\vee}\otimes\eta_A^B)\circ(\g^{M^\vee}\otimes\I_A)
\big(e^i\otimes a\big)\\
&=
(-1)^{|k|+|i||k|}\cdot
(\I_{M^\vee}\otimes \imath_A)\circ(\I_{M^\vee}\otimes\ep_B\otimes\I_A)\circ(\I_{M^\vee}\otimes\eta_A^B)
\big(e^k\otimes\vs_B(\g^i_k)\otimes a\big)\\
&=
(-1)^{|k|+|i||k|+|k||\eta_A|}\cdot
(\I_{M^\vee}\otimes \imath_A)\circ(\I_{M^\vee}\otimes\ep_B\otimes\I_A)
\Big(e^k\otimes
\eta_A^B\big(\vs_B(\g^i_k)\otimes a\big)
\Big)\\
&=
(-1)^{|k|+|i||k|+|k||\eta_A|}\cdot
(\I_{M^\vee}\otimes \imath_A)\circ(\I_{M^\vee}\otimes\ep_B\otimes\I_A)
\big(e^k\otimes b^{i t}_k\otimes a_t\big)\\
&=
(-1)^{|k|+|i||k|+|k||\eta_A|}\ep_B(b^{i t}_k)\cdot
\big(e^k\otimes a_t\big).
\end{aligned}
\]
Here, we used the defining relation
$\eta^B_A\big(\vs_B(\g^i_k)\otimes a\big)= b^{i t}_k\otimes a_t$.

Finally, we calculate $S(\operatorname{res}_A\eta_A)^M\big(m\otimes a\big)$:
\[
\begin{aligned}
S(\operatorname{res}_A & \eta_A)^M\big(m\otimes a\big)\\
=&
\t\circ\imath_{A\otimes M}\circ(\operatorname{ev}_{M^\vee}\otimes\I_{A\otimes M})\circ(\I_M\otimes\eta_A^{M^\vee}\otimes\I_M)\\
&\circ(\I_{M\otimes M^\vee}\otimes\t)\circ(\I_M\otimes \operatorname{cv}_{M^\vee}\otimes \I_A)\circ(\jmath^{-1}_M\otimes\I_A)
\big(m\otimes a\big)\\
=&
(-1)^{|i|}\cdot
\t\circ\imath_{A\otimes M}\circ(\operatorname{ev}_{M^\vee}\otimes\I_{A\otimes M})
\circ(\I_M\otimes\eta_A^{M^\vee}\otimes\I_M)\circ(\I_{M\otimes M^\vee}\otimes\t)
\big(m\otimes e^i\otimes x_i\otimes a\big)\\
=&
(-1)^{|i|+|i||a|}\cdot
\t\circ\imath_{A\otimes M}\circ(\operatorname{ev}_{M^\vee}\otimes\I_{A\otimes M})
\circ(\I_M\otimes\eta_A^{M^\vee}\otimes\I_M)
\big(m\otimes e^i\otimes a\otimes x_i\big)\\
=&
(-1)^{|i|+|i||a|+|m||\eta_A|}\cdot
\t\circ\imath_{A\otimes M}\circ(\operatorname{ev}_{M^\vee}\otimes\I_{A\otimes M})
\big(m\otimes \eta_A^{M^\vee}(e^i\otimes a)\otimes x_i\big)\\
=&
(-1)^{|i|+|i||a|+|m||\eta_A|+|k|+|i||k|+|k||\eta_A|}\ep_B(b^{i t}_k)\cdot
\t\circ\imath_{A\otimes M}\circ(\operatorname{ev}_{M^\vee}\otimes\I_{A\otimes M})
\big(m\otimes e^k\otimes a_t\otimes x_i\big)\\
=&
(-1)^{|i|+|i||a|+|m||\eta_A|+|k|+|i||k|+|k||\eta_A|+|k||m|}m^k\ep_B(b^{i t}_k)\cdot
\t\big(a_t\otimes x_i\big)\\
=&
(-1)^{|i|+|i||a|+|m||\eta_A|+|k|+|i||k|+|k||\eta_A|+|k||m|+|a_t||i|}m^k\ep_B(b^{i t}_k)\cdot
\big(x_i\otimes a_t\big).
\end{aligned}
\]

3.
We are left to check that the sign factors agree. Note that in the above calculation,
\begin{itemize}
\item $|m|=|k|$ since we calculated $e^k(m)=m^k$ on the $6$th equality.
\item $|b^{i t}_k|=0$ since other degree terms vanish when we calculate $\ep_B(b^{i t}_k)$
on the $5$th equality. Thus we have $|a_t|=|a|+|\eta_A|+|k|+|i|$ from the defining relation
$\eta^B_A\big(\vs_B(\g^i_k)\otimes a\big)= b^{i t}_k\otimes a_t$.
\end{itemize}
Therefore, in modulo $2$, we have
\[
\begin{aligned}
&|i|+|i||a|+|m||\eta_A|+|k|+|i||k|+|k||\eta_A|+|k||m|+|a_t||i|\\
\equiv&\text{ }
|i|+|i||a|+|m||\eta_A|+|m|+|i||m|+|m||\eta_A|+|m||m|+|a_t||i|\\
\equiv&\text{ }
|i|+|i||a|+|i||m|+|a_t||i|\\
\equiv&\text{ }
|i|+|i||a|+|i||m|+\big(|a|+|\eta_A|+|m|+|i|\big)|i|\\
\equiv&\text{ }
|i||\eta_A|.
\end{aligned}
\]
This shows that $\vs(\eta_A)^M\big(m\otimes a\big)=S(\operatorname{res}_A\eta_A)^M\big(m\otimes a\big)$
holds for all $m\in M$ and $a\in A$.
We conclude that $\vs(\eta_A)^M=S(\operatorname{res}_A\eta_A)^M$
holds for every finite dimensional right dg-comodule $(M,\g^M)$ over $B$.
\qed
\end{proof}

%
%
%
%
%

\renewcommand\refname{References}

\end{document}